\documentclass[reqno]{amsart}
\usepackage[top=1.0in,bottom=1.0in,left=1.0in,right=1.0in]{geometry}
\usepackage{amscd}
\usepackage{amssymb}
\usepackage{multicol}
\usepackage{makecell}
\usepackage[T1]{fontenc}
\usepackage[latin1]{inputenc}

\usepackage[all]{xy}
\usepackage{times}

\usepackage{flafter}
\usepackage{placeins}
\usepackage{enumitem}
\usepackage[OT2,T1]{fontenc}
\DeclareSymbolFont{cyrletters}{OT2}{wncyr}{m}{n}
\DeclareMathSymbol{\Sha}{\mathalpha}{cyrletters}{"58}

\usepackage{centernot}

\usepackage{mathrsfs}
\usepackage[color]{showkeys}
\usepackage{url}

\definecolor{refkey}{rgb}{1,1,1}
\definecolor{labelkey}{rgb}{1,1,1}
\definecolor{cite}{rgb}{0.9451,0.2706,0.4941}
\definecolor{ruri}{rgb}{0.0078,0.4022,0.8010}

\usepackage[%
bookmarks=true,bookmarksnumbered=true,%
colorlinks=true,linkcolor=blue,citecolor=red%
 ]{hyperref}

\newcommand{\mcC}{\mathcal{C}}

\newcommand{\mcF}{\mathcal{F}}

\newcommand{\mcI}{\mathcal{I}}

\newcommand{\mcM}{\mathcal{M}}

\newcommand{\mbF}{\mathbb{F}}

\newcommand{\mbN}{\mathbb{N}}

\newcommand{\mbP}{\mathbb{P}}
\newcommand{\mbQ}{\mathbb{Q}}
\newcommand{\mbR}{\mathbb{R}}

\newcommand{\mbZ}{\mathbb{Z}}

\newcommand{\mfP}{\mathfrak{P}}

\newcommand{\msC}{\mathscr{C}}

\newcommand{\msV}{\mathscr{V}}

\newtheorem{weakisomGC}{Weak Isom-version Conjecture}

\newtheorem{weakhomGC}{Weak Hom-version Conjecture}

\newtheorem{pcc}{Pointed Collection Conjecture}

\newcommand{\be}{\begin{enumerate}}
\newcommand{\ee}{\end{enumerate}}
\newcommand{\bcd}{\[\begin{CD}}
\newcommand{\ecd}{\end{CD}\]}
\newcommand{\bit}{\begin{itemize}}
\newcommand{\eit}{\end{itemize}}
\newcommand{\bq}{\begin{quote}}
\newcommand{\eq}{\end{quote}}
\newcommand{\bpf}{\begin{proof}}
\newcommand{\epf}{\end{proof}}

\newcommand{\spec}{\textrm{Spec}}

\theoremstyle{plain}

\newtheorem{theorem}{Theorem}[section]

\newtheorem{proposition/example}[theorem]{Proposition/Example}
\newtheorem{proposition}[theorem]{Proposition}

\newtheorem{lemma}[theorem]{Lemma}

\theoremstyle{definition}
\newtheorem{definition}[theorem]{Definition}
\newtheorem{remark}[theorem]{Remark}

\newtheorem{conjecture/question}[theorem]{Conjecture/Question}

\newtheorem{remark/definition}[theorem]{Remark/Definition}
\newtheorem{definition/notation}[theorem]{Definition/Notation}

\numberwithin{equation}{section}

 \theoremstyle{remark}

\numberwithin{equation}{section}

\usepackage{extarrows}\usepackage{bm}

\begin{document}
\title{\textbf{Topological structures of moduli spaces of curves and anabelian geometry in positive characteristic}}

\author{Zhi Hu}

\address{  \textsc{School of Mathematics, Nanjing University of Science and Technology, Nanjing 210094, China}}

\email{halfask@mail.ustc.edu.cn}

\author{Yu Yang}

\address{ \textsc{ Research Institute for Mathematical Sciences, Kyoto University, Kyoto 606-8502, Japan}}

\email{yuyang@kurims.kyoto-u.ac.jp}
\author{Runhong Zong}

\address{ \textsc{Department of Mathematics, Nanjing University, Nanjing  210093, China}}
\email{rzong@nju.edu.cn}

\date{}
\begin{abstract}
In the present paper, we study a new kind of anabelian phenomenon concerning the smooth pointed stable curves in positive characteristic. It shows that the topological structures of moduli spaces of curves can be understood from the viewpoint of anabelian geometry. We formulate some new anabelian-geometric conjectures relating the  tame fundamental groups of curves over algebraically closed fields of characteristic $p>0$ to the  moduli spaces of curves. These conjectures are generalized versions of the weak Isom-version of the Grothendieck conjecture for curves over algebraically closed fields of characteristic $p>0$ which was formulated by Tamagawa. Moreover, we prove that the conjectures hold for certain points lying in  the moduli space of curves of genus $0$.
\end{abstract}

\maketitle

\tableofcontents

\section{Introduction}

\subsection{The mystery of fundamental groups in positive characteristic}

\subsubsection{}\label{sec111} Let $k$ be an algebraically closed field of characteristic $p\geq 0$, and let $(X, D_{X})$ be a {\it smooth} pointed stable curve of type $(g_{X}, n_{X})$ over $k$ (i.e. $2g_{X}+n_{X}-2>0$, see \cite[Definition 1.1 (iv)]{K}), where $X$ denotes the underlying curve, $D_{X}$ denotes the (ordered) finite set of marked points, $g_{X}$ denotes the genus of $X$, and $n_{X}$ denotes the cardinality $\#(D_{X})$ of $D_{X}$. We put $U_{X}:= X \setminus D_{X}$. By choosing a base point of $U_{X}$, we have the tame fundamental group $\pi_{1}^{\rm t}(U_{X})$ of $U_{X}$.

If $p=0$, it is well-known that $\pi_{1}^{\rm t}(U_{X})$ is isomorphic to the profinite completion of the topological fundamental group of a Riemann surface of type $(g_{X}, n_{X})$. Hence, almost no geometric information about $U_{X}$ can be carried out from $\pi_{1}^{\rm t}(U_{X})$. By contrast, if  $p>0$, the situation is quite different from that in characteristic $0$. The tame fundamental group $\pi_{1}^{\rm t}(U_{X})$ is very mysterious and its structure is no longer known, in particular, there exist {\it anabelian phenomena for curves over algebraically closed fields of characteristic $p>0$}.

\subsubsection{} Firstly, let us explain some general background about anabelian geometry. In the 1980s, A. Grothendieck suggested
a theory of arithmetic geometry called anabelian geometry (\cite{G}). The central question of the theory is as follows: Can we reconstruct the geometric information of a
variety group-theoretically from various versions of its algebraic fundamental group?  The original anabelian geometry suggested by Grothendieck focused on varieties over {\it arithmetic} fields, in particular, the fields finitely generated over $\mbQ$. In the case of curves in characteristic $0$, anabelian geometry has been deeply studied (e.g. \cite{N}, \cite{T1}) and, in particular, the most important case (i.e., the fields finitely generated over $\mbQ$, or more general, sub-$p$-adic fields) has been established completely(\cite{M}).  Note that the actions of the Galois groups of the base fields on the geometric fundamental groups play a crucial role for recovering geometric information of curves over arithmetic fields.

Next, we return to the case where $k$ is an algebraically closed field of characteristic $p>0$.   In \cite{T2}, A. Tamagawa discovered that there also exist {\it anabelian phenomena} for curves over {\it algebraically closed fields of characteristic $p$}. This came rather surprisingly
since it means that, in positive characteristic, the geometry of curves can be determined
by their geometric fundamental groups {\it without Galois actions}. Since the late 1990s, this kind of anabelian phenomenon has been studied further by M. Raynaud (\cite{R2}), F. Pop-M. Sa\"idi (\cite{PS}), Tamagawa (\cite{T2}, \cite{T4}, \cite{T5}), and the second author of the present paper (\cite{Y1}, \cite{Y2}, \cite{Y4}). More precisely, they focused on the so-called {\it weak Isom-version of Grothendieck's anabelian conjecture for curves over algebraically closed fields of characteristic $p>0$} (or the ``weak Isom-version conjecture'' for short) formulated by Tamagawa (\cite[Conjecture 2.2]{T3}), which says that curves are isomorphic if and only if their tame (or \'etale) fundamental groups are isomorphic. At the present, this conjecture is still wide-open.

\subsection{Topology structures of moduli spaces of curves and anabelian geometry}  In the present paper, we study a new kind of anabelian phenomenon concerning curves over algebraically closed fields of characteristic $p>0$ which shows that the topological structures of moduli spaces of curves can be understood by their fundamental groups.

\subsubsection{} Let $\mbF_{p}$ be the prime field of characteristic $p>0$, and let $\mcM_{g, n, \mbZ}^{\rm ord}$ be the moduli stack over $\mbZ$ parameterizing smooth $n$-pointed stable curves of type $(g, n)$ (in the sense of \cite{K}). We put $\mcM^{\rm ord}_{g, n, \mbF_{p}}:= \mcM^{\rm ord}_{g, n, \mbZ} \times_{\mbZ}  \mbF_{p}$. Note that the set of marked points of an $n$-smooth pointed stable curve admits a natural action of the $n$-symmetric group $S_{n}$. Moreover, we denote by $\mcM_{g, n, \mbF_{p}}:=[\mcM^{\rm ord}_{g, n, \mbF_{p}}/S_{n}]$ the quotient stack, and denote by $M_{g, n, \mbF_{p}}$ the coarse moduli space of $\mcM_{g, n, \mbF_{p}}$.

Let $q \in M_{g, n, \mbF_{p}}$ be an arbitrary point, $k(q)$ the residue field of $q$, $k_{q}$ an algebraically closed field containing $k(q)$, and $V_{q}:= \overline {\{q\}}$ the topological closure of $\{q\}$ in $M_{g, n, \mbF_{p}}$. Write $(X_{k_q}, D_{X_{k_q}})$ for the smooth pointed stable curve of type $(g, n)$ over $k_{q}$ determined by the natural morphism $\spec k_{q} \rightarrow M_{g, n, \mbF_{p}}$ and put $U_{X_{k_{q}}} := X_{k_q} \setminus D_{X_{k_q}}$. In particular, we put $(X_{k_{q}}, D_{X_{k_{q}}}):= (X_{q}, D_{X_{q}})$ and $U_{X_q}:= X_{q} \setminus D_{X_{q}}$ if  $k_{q}$ is an algebraic closure of $k(q)$. Since the isomorphism class of the tame fundamental group $\pi^{\rm t}_{1}(U_{X_{k_{q}}})$ depends only on $q$, we shall write $\pi_{1}^{\rm t}(q)$ for the tame fundamental group $\pi^{\rm t}_{1}(U_{X_{k_{q}}})$.

\subsubsection{}\label{sec122} We maintain the notation introduced above. The weak Isom-version conjecture of Tamagawa can be reformulated as follows:

\begin{weakisomGC}
Let $q_{i} \in M_{g, n, \mbF_{p}}$, $i\in \{1,2\}$, be an arbitrary point of $M_{g, n, \mbF_{p}}$. The set of continuous isomorphisms of profinite groups $$ {\rm Isom}_{\rm pg}(\pi_{1}^{\rm t}(q_{1}),\pi_{1}^{\rm t}(q_{2}))$$ is non-empty if and only if $V_{q_{1}}=V_{q_{2}}$ (namely, $U_{X_{q_{1}}} \cong U_{X_{q_{2}}}$ as schemes).
\end{weakisomGC}

\noindent
The weak Isom-version conjecture means that {\it moduli spaces of curves can be reconstructed ``as sets'' from the isomorphism classes of the tame fundamental groups of curves}. This conjecture has been {\it only} confirmed by Tamagawa (\cite[Theorem 0.2]{T4}) in the following case:
\begin{quote}
Suppose that $q_{1}$ is a {\it closed} point of $M_{0, n, \mbF_{p}}$. Then the weak Isom-version conjecture holds true.
\end{quote}

Next, we propose a new conjecture as follows, that is {\it the weak Hom-version of the Grothendieck conjecture for curves over algebraically closed fields of characteristic $p>0$} (or is called \emph{weak Hom-version conjecture} for simplicity), as a  generalization of  the weak Isom-version conjecture.

\begin{weakhomGC}
Let $q_{i} \in M_{g, n, \mbF_{p}}$, $i\in \{1,2\}$, be an arbitrary point of $M_{g, n, \mbF_{p}}$. The set of open continuous homomorphisms of profinite groups $${\rm Hom}^{\rm op}_{\rm pg}(\pi_{1}^{\rm t}(q_{1}), \pi_{1}^{\rm t}(q_{2}))$$ is non-empty if and only if $V_{q_{1}}\supseteq V_{q_{2}}.$
\end{weakhomGC}

\noindent
The weak Hom-version conjecture means that  {\it the sets of deformations} of a smooth pointed stable curve can be reconstructed group-theoretically from the sets of open continuous homomorphisms of their tame fundamental groups.  Therefore, it provides {\it a new kind of anabelian phenomenon}:
\begin{quote}
The moduli spaces of curves in positive characteristic can be understood not only as sets but also ``{\it as topological spaces}'' from the {\it sets of open continuous homomorphisms} of tame fundamental groups of curves in positive characteristic.
\end{quote}
Roughly speaking, this means that a smooth pointed stable curve corresponding to a geometric point over $q_{2}$ can be deformed to a smooth pointed stable curve corresponding to a geometric point over $q_{1}$ if and only if the set of open continuous homomorphisms  of tame fundamental groups $\text{Hom}_{\text{pg}}^{\rm op}(\pi_{1}^{\rm t}(q_{1}), \pi_{1}^{\rm t}(q_{2}))$ is not empty. 

\subsection{Main results}

\subsubsection{}
The main result of the present paper is the following (see Theorem \ref{them-4} (iv) for a more general statement):
\begin{theorem}\label{maintheorem}
The Weak Hom-version Conjecture holds when $q_{1}$ is a closed point of $M_{0, n, \mbF_{p}}$.
\end{theorem}
\noindent
Theorem \ref{maintheorem} follows from the following ``Hom-type" anabelian result (see  Theorem \ref{them-3} for a more precise  statement) which is a generalization of Tamagawa's result (i.e. \cite[Theorem 0.2]{T4}):

\begin{theorem}\label{them-0-1}
Let $q_{i} \in M_{g, n, \mbF_{p}}$, $i\in \{1,2\}$, be an arbitrary point of $M_{g, n, \mbF_{p}}$. Suppose that $q_{1}$ is a closed point of $M_{g, n, \mbF_{p}}$. Then the set of open continuous homomorphisms $${\rm Hom}^{\rm op}_{\rm pg}(\pi_{1}^{\rm t}(q_{1}), \pi_{1}^{\rm t}(q_{2}))$$ is non-empty if and only if $U_{X_{q_1}} \cong U_{X_{q_2}}$ as schemes. 
\end{theorem}

Note that Theorem \ref{them-0-1} is {\it essentially} different from \cite[Theorem 0.2]{T4}. The reason is the following: We a priori {\it do not  know} whether or not $${\rm Isom}_{\rm pg}(\pi_{1}^{\rm t}(q_{1}), \pi_{1}^{\rm t}(q_{2}))$$ is non-empty even through ${\rm Hom}^{\rm op}_{\rm pg}(\pi_{1}^{\rm t}(q_{1}), \pi_{1}^{\rm t}(q_{2}))$ is {\it non-empty}. In fact, for arbitrary $q_{i} \in M_{g, n, \mbF_{p}}$, $i\in \{1,2\}$, we have $${\rm Isom}_{\rm pg}(\pi_{1}^{\rm t}(q_{1}), \pi_{1}^{\rm t}(q_{2}))=\emptyset, \ {\rm Hom}^{\rm op}_{\rm pg}(\pi_{1}^{\rm t}(q_{1}), \pi_{1}^{\rm t}(q_{2}))\neq \emptyset$$ in general (\cite[Theorem 0.3]{T5}).

On the other hand, to verify Theorem \ref{them-0-1}, we need to establish various anabelian reconstructions from {\it open continuous homomorphisms} of tame fundamental groups which are much harder than the case of {\it isomorphisms} in general. We explain in more detail about this point in the  following.

\subsubsection{}
Let us explain the main differences between the proofs of Tamagawa's result (i.e. \cite[Theorem 0.2]{T4}) and our result (i.e. Theorem \ref{them-0-1}), and the new  ingredient  in our proof. First, we recall the key points of the proof of Tamagawa's result. Roughly speaking, Tamagawa's proof consists of two parts:
\begin{itemize}
\item[(1)] He proved that {\it the sets of inertia subgroups} of  marked points and {\it the field structures} associated to inertia subgroups of marked points of smooth pointed stable curves can be reconstructed group-theoretically from tame fundamental groups. This is the most difficult part of Tamagawa's proof.

\item[(2)] By using the inertia subgroups and their associated field structures, if $g=0$, he proved that the coordinates of marked points can be calculated group-theoretically.
\end{itemize}

The group-theoretical reconstructions in Tamagawa's proofs (1) and (2) are {\it isomorphic version reconstructions}. This means that the reconstructions should fix an isomorphism class of a tame fundamental group.  To explain this, let us show an example.  Let $U_{X_{i}}$, $i\in \{1, 2\}$, be a curve of type $(g_{X}, n_{X})$ over an algebraically closed field $k$ of characteristic $p>0$ introduced above, $\pi_{1}^{\rm t}(U_{X_{i}})$ the tame fundamental group of $U_{X_{i}}$,  $\phi: \pi_{1}^{\rm t}(U_{X_{1}}) \rightarrow \pi_{1}^{\rm t}(U_{X_{2}})$ an open continuous homomorphism, $H_{2} \subseteq \pi_{1}^{\rm t}(U_{X_{2}})$ an open subgroup, and $H_{1}:=\phi^{-1}(H_{2})$. In Tamagawa's proof, since $\phi$ is an {\it isomorphism}, we have  $H_{1}\simeq H_{2}$. Then the group-theoretical reconstruction for  types implies that the type $(g_{X_{H_{1}}}, n_{X_{H_{1}}})$ and the type $(g_{X_{H_{2}}}, n_{X_{H_{2}}})$ of the curves corresponding to $H_{1}$ and $H_{2}$, respectively, are equal. This is a key point in the proof of Tamagawa's group-theoretical reconstruction of the inertia subgroups of marked points. Unfortunately, his method cannot be applied
to the present paper. The reason is that we need to treat the case where $\phi$ is an {\it arbitrary open continuous homomorphism}. Since $H_{1}$ is not isomorphic to $H_{2}$ in general (e.g. specialization homomorphism), we {\it do not} know whether or not $(g_{X_{H_{1}}}, n_{X_{H_{1}}})=(g_{X_{H_{2}}}, n_{X_{H_{2}}}).$ This is one of the main difficulties of ``Hom-type'' problems appeared in anabelian geometry. Similar difficulties for generalized Hasse-Witt invariants will appear if we try to reconstruct the field structure associated to inertia subgroups of marked points.

To overcome the difficulties mentioned above, we have the following key observation:
\begin{quote}
The inequalities of $\text{Avr}_{p}(H_{i})$ (i.e., the $p$-averages of generalized Hasse-Witt invariants (see \ref{paverage})) induced by $\phi$ play roles of the
comparability of (outer) Galois representations in the theory of anabelian
geometry of curves over algebraically closed fields of characteristic $p > 0$.
\end{quote}
In the present paper, our method for reconstructing inertia subgroups of marked points is completely different from Tamagawa's reconstruction. We develop a new {\it group-theoretical algorithm} for reconstructing the inertia subgroups of marked points whose input datum is a profinite group which is isomorphic to $\pi_{1}^{\rm t}(U_{X_{i}})$, $i\in \{1, 2\}$, and whose output data are inertia subgroups of marked points (Theorem \ref{them-2}). Moreover, we prove that the group-theoretical algorithm and the reconstructions for field structures are {\it compatible} with arbitrary surjection $\phi$ (Proposition \ref{pro-4}). By using Theorem \ref{them-2} and Proposition \ref{pro-4}, we may prove that Tamagawa's calculation of coordinates is compatible with our reconstructions. This implies Theorem \ref{them-0-1}.


\subsection{Some further developments}

\subsubsection{Moduli spaces of fundamental groups}
Let us explain some further developments for the anabelian phenomenon concerning the weak Hom-verson conjecture. In \cite{Y6}, the second author of the present paper introduced a topological space $\Pi_{g,n}$ (or more general, $\overline \Pi_{g, n}$) determined group-theoretically by the tame fundamental groups of smooth pointed stable curves (or more general, the geometric log \'etale fundamental groups of arbitrary pointed stable curves) of type $(g, n)$ which is called the {\it moduli spaces of fundamental groups of curves}, whose underlying set is the sets of isomorphism classes of fundamental groups, and whose topology is determined by the sets of finite quotients of fundamental groups.  Moreover, he posed the so-called {\it homeomorphism conjecture}, roughly speaking, which says that (by quotiening a certain equivalence relation induced by Frobenius actions) the moduli spaces of curves are \emph{homeomorphic} to the moduli spaces of fundamental groups.

In  the present literatures, the term ``anabelian''  means that a geometric object can be determined by its fundamental group.  Furthermore, the homeomorphism conjecture concerning moduli spaces of fundamental groups supplies a new point of view to understand anabelian phenomena as follows:
\begin{quote}
The term ``anabelian'' means that not only a geometric object can be determined by its fundamental groups, but also a certain {\it moduli space of geometric objects} can be determined by the fundamental groups of geometric objects.
\end{quote}
Under this point of view, {\it the homeomorphism conjecture is regarded as the analogue of a famous theorem in the theory of classic Teichm\"uller spaces} which states that the Teichm\"uller spaces of complex hyperbolic curves are homeomorphic to the spaces of discrete and faithful representations of  topological fundamental groups of underlying surfaces into the group $PSL_{2}(\mbR)$.

Now Theorem \ref{maintheorem} implies that $M_{0, 4, \mbF_{p}}$ is {\it homeomorphic} to $\Pi_{0,4}$ {\it as topological spaces} (note that Tamagawa's result (i.e. \cite[Theorem 0.2]{T4}) only says that the natural map $M_{0, 4, \mbF_{p}} \rightarrow \Pi_{0,4}$ is a bijection {\it as sets}). Moreover, based on \cite{Y1}, \cite{Y3}, \cite{Y4}, \cite{Y5}, and the main results of the present paper, the homeomorphism conjecture is confirmed  for {\it$1$-dimensional  moduli spaces of  pointed stable curves} in \cite{Y6} and \cite{Y7}. For the homeomorphism conjecture in the case of  higher dimensional moduli spaces of curves,  the weak Hom-version conjecture and the pointed collection conjecture (see Section \ref{pcc} of the present paper) are also the main steps toward understanding it (see \cite[Section 1.2.3]{Y8}).

\subsubsection{The sets of finite quotients of tame fundamental groups} We maintain the notation introduced in \ref{sec111}. The techniques developed in \S\ref{mpanabelian} of the present paper have important applications for understanding  the set of finite quotients $\pi_{A}^{\rm t}(U_{X})$ of the {\it  tame fundamental groups} $\pi_{1}^{\rm t}(U_{X})$ of $U_{X}$.

Note that, if $U_{X}$ is {\it affine}, the set $\pi_{A}^{\text{\'et}}(U_{X})$ of finite quotients  of the {\it \'etale fundamental groups} $\pi_{1}^{\text{\'et}}(U_{X})$ of $U_{X}$ can be completely determined by its type $(g_{X}, n_{X})$ (i.e. Abhyankar's conjecture proved by Raynaud and D. Harbater). However, the structure of $\pi_{1}^{\text{\'et}}(U_{X})$ cannot be carried out from $\pi_{A}^{\text{\'et}}(U_{X})$ since $\pi_{1}^{\text{\'et}}(U_{X})$ is not topologically finitely generated when $U_{X}$ is affine.

By contrast, the isomorphism class of $\pi_{1}^{\rm t}(U_{X})$ can be completely determined by $\pi_{A}^{\rm t}(U_{X})$ since $\pi_{1}^{\rm t}(U_{X})$ is topologically finitely generated, and one {\it cannot} expect that there exists an explicit description for the entire set $\pi_{A}^{\rm t}(U_{X})$ since there exists {\it anabelian phenomenon} mentioned above (i.e. $\pi_{A}^{\rm t}(U_{X})$ depends on the isomorphism class of $U_{X}$). On the other hand, for understanding more precisely the relationship between the structures of tame fundamental groups and the anabelian phenomena in positive characteristic world, it is naturally to ask the following  interesting problem:
\begin{quote}
How does the scheme structure of $U_{X}$ affect explicitly the set of finite quotients $\pi_{A}^{\rm t}(U_{X})$?
\end{quote}
In \cite{Y9}, by applying the techniques developed in \S\ref{mpanabelian} of the present paper and \cite[Theorem 1.2]{Y5}, we obtain the following result:
\begin{quote}
Let $q_{1}\in M_{g_{1}, n_{1}, \mbF_{p}}$ and $q_{2}\in M_{0, n_{2}, \mbF_{p}}$ be arbitrary points and $\pi_{A}^{\rm t}(q_{i})$ the set of finite quotients of the tame fundamental group $\pi_{1}^{\rm t}(q_{i})$. Suppose that $q_{2}$ is a closed point of $M_{0, n_{2}, \mbF_{p}}$, and that $\pi_{1}^{\rm t}(q_{1})\not\cong \pi_{1}^{\rm t}(q_{2})$. Then we can {\it construct explicitly} a finite group $G_{q_{2}}$ depending on $q_{2}$ such that $G_{q_{2}} \in \pi_{A}^{\rm t}(q_{1})$ and $G_{q_{2}} \not\in \pi_{A}^{\rm t}(q_{2})$.
\end{quote}

\subsection{Structure of the present paper}
The present paper is organized as follows. In Section \ref{sec-1}, we formulate the the weak Hom-version conjecture and the pointed collection conjecture. In Section \ref{mpanabelian}, we give a group-theoretical algorithm for reconstructions of inertia subgroups associated to marked points, and prove that the group-theoretical algorithm is compatible with arbitrary open surjective homomorphisms of tame fundamental groups. In Section \ref{sec-5}, we prove our main results.

\bigskip
\subsection{Acknowledgements}

The second author was supported by JSPS Grant-in-Aid for Young Scientists Grant Numbers 16J08847 and 20K14283.

\bigskip

\section{Conjectures}\label{sec-1}
In this section, we formulate two new conjectures concerning anabelian geometry of curves over algebraically closed fields of characteristic $p>0$.

\subsection{The weak Hom-version conjecture} In this subsection, we formulate the first conjecture of the present paper which we call ``the weak Hom-version conjecture''.

\subsubsection{}\label{curves}
Let $k$ be an algebraically closed field of characteristic $p>0$, and  let $$(X, D_{X})$$ be a smooth pointed stable curve of type $(g_{X}, n_{X})$ over $k$, where $X$ denotes the (smooth) underlying curve of genus $g_{X}$ and $D_{X}$ denotes the (ordered) finite set of marked points with cardinality $n_{X}:=\#(D_{X})$ satisfying \cite[Definition 1.1 (iv)]{K} (i.e. $2g_{X}+n_{X}-2>0$). Note that $U_{X}:= X \setminus D_{X}$ is a hyperbolic curve over $k$.

Let $(Y, D_{Y})$ and $(X, D_{X})$ be smooth pointed stable curves over $k$, and let $f: (Y, D_{Y}) \rightarrow(X, D_{X})$ be a morphism of smooth pointed stable curves over $k$. We shall say that $f$ is {\it \'etale} (resp. {\it tame}, {\it Galois \'etale}, {\it Galois tame}) if $f$ is \'etale over $X$ (resp. $f$ is \'etale over $U_{X}$ and is at most tamely ramified over $D_{X}$, $f$ is a Galois covering and is \'etale, $f$ is a Galois covering and is tame).

By choosing a base point of $x \in U_{X}$, we have the tame fundamental group $\pi^{\rm t}_{1}(U_{X}, x)$ of $U_{X}$ and the \'etale fundamental group $\pi_{1}(X, x)$ of $X$. Since we only focus on the isomorphism classes of fundamental groups in the present paper, for simplicity of notation, we omit the base point and denote by $\pi_{1}^{\rm t}(U_{X}) \ \text{and} \ \pi_{1}(X)$ the tame fundamental group $\pi^{\rm t}_{1}(U_{X}, x)$ of $U_{X}$ and the \'etale fundamental group $\pi_{1}(X, x)$ of $X$, respectively. Note that there is a natural continuous surjective homomorphism $\pi_{1}^{\rm t}(U_{X}) \twoheadrightarrow \pi_{1}(X).$

\subsubsection{}\label{moduli212}
Let $\overline \mbF_{p}$ be an algebraic closure of $\mbF_{p}$, and let $\mcM^{\rm ord}_{g, n, \mbF_{p}}$ be the moduli stack over $\mbZ$ parameterizing smooth pointed stable curves of type $(g, n)$ in the sense of \cite[Definition 1.1]{K}. The set of marked points of a smooth pointed stable curve admits a natural action of the $n$-symmetric group $S_{n}$, we put $\mcM_{g, n, \mbZ}:= [\mcM_{g, n, \mbZ}^{\rm ord}/S_{n}]$ the quotient stack. Moreover, we denote by $\mcM_{g, n}^{\rm ord} := \mcM_{g, n, \mbZ} \times_{\mbZ} \overline \mbF_{p}$, $\mcM_{g, n, \mbF_{p}}:= \mcM_{g, n, \mbZ} \times_{\mbZ} \mbF_{p}$, and $\mcM_{g, n}:=\mcM_{g, n, \mbZ} \times_{\mbZ} \overline \mbF_{p}$, and denote by $M_{g, n}^{\rm ord}$, $M_{g, n, \mbF_{p}}$, and $M_{g, n}$ the coarse moduli spaces of $\mcM_{g, n}^{\rm ord}$, $\mcM_{g, n, \mbF_{p}}$, and $\mcM_{g, n}$, respectively.




Let $q\in M_{g, n}^{\rm ord}$ be an arbitrary point and $k(q)$ the residue field of $q$, and $k_{q}$ an algebraically closed field containing $k(q)$. Write $(X_{k_q}, D_{X_{k_q}})$ for the smooth pointed stable curve of type $(g, n)$ over $k_{q}$ determined by the natural morphism $\spec k_{q} \rightarrow \spec k(q) \rightarrow M_{g, n}^{\rm ord}$ and $U_{X_{k_q}}$ for $X_{k_q} \setminus D_{X_{k_q}}$. In particular, if $k_{q}$ is an algebraic closure of $k(q)$, we shall write $(X_{q}, D_{X_{q}})$ for $(X_{k_q}, D_{X_{k_q}})$.

Since the isomorphism class of the tame fundamental group $\pi^{\rm t}_{1}(U_{X_{k_q}})$ depends only on $q$ (i.e., the isomorphism class does not depend on the choices of $k_{q}$), we shall write $\pi^{\rm t}_{1}(q) \ \text{and} \ \pi^{\rm t}_{A}(q)$ for $\pi^{\rm t}_{1}(U_{X_{k_q}})$ and the set of finite quotients of $\pi^{\rm t}_{1}(U_{X_{k_q}})$, respectively. \cite[Proposition 16.10.7]{FJ} implies that for any points $q_{1}, q_{2} \in M_{g, n}^{\rm ord}$,  $\pi^{\rm t}_{1}(q_{1})  \cong \pi^{\rm t}_{1}(q_{2})$ as profinite groups if and only if $\pi^{\rm t}_{A}(q_{1})=\pi^{\rm t}_{A}(q_{2})$ as sets.

On the other hand, Let $q\in M_{g, n}^{\rm ord}$ and  $q'\in M_{g, n, \mbF_{p}}$ be arbitrary points. We denote by $V_{q}\subseteq M_{g, n}^{\rm ord}$ and $V_{q'} \subseteq M_{g, n, \mbF_{p}}$ the topological closures of $q$ and $q'$ in $M^{\rm ord}_{g, n}$ and $M_{g, n, \mbF_{p}}$, respectively. 


\subsubsection{} We have the following definition.
\begin{definition}\
\label{def-2}
\begin{enumerate}
  \item  Let $c_{1}, c_{2} \in M_{g, n}^{\rm ord, cl}$ be {\it closed points}, where $(-)^{\rm cl}$ denotes the set of closed points of $(-)$. Then $c_{1} \sim_{fe} c_{2}$ if there exists $m\in\mbZ$ such that $ \nu(c_{2})=\nu(c_{1}^{(m)})$, where $c_{1}^{(m)}$ denotes the closed point corresponding to the curve obtained by $m$th Frobenius twist of the curve corresponding to $c_{1}$. Here ``\emph{fe}" means ``Frobenius equivalence".
  \item Let $q_{1}, q_{2} \in M_{g, n}^{\rm ord}$ be arbitrary points. We denote by $V_{q_1} \supseteq_{fe} V_{q_2}$ if, for each closed point $c_{2} \in V^{\rm cl}_{q_2}$, there exists a closed point $c_{1} \in V^{\rm cl}_{q_{1}}$ such that $c_{1} \sim_{fe} c_{2}$. Moreover, we denote by $V_{q_1} =_{fe} V_{q_2}$ if $V_{q_1} \supseteq_{fe} V_{q_2}$ and $V_{q_1} \subseteq_{fe} V_{q_2}$. Moreover, we also denote by $q_{1} \sim_{fe} q_{2}$ if $V_{q_{1}}=_{fe} V_{q_{2}}$.
\end{enumerate}


\end{definition}



We have the following proposition.

\begin{proposition}\label{pro-5} Let  $\omega: M_{g, n}^{\rm ord} \rightarrow M_{g, n, \mbF_{p}}$ be the morphism induced by the natural morphism $\mcM_{g, n}^{\rm ord} \rightarrow \mcM_{g, n, \mbF_{p}}$. Let $i\in \{1,2\}$, and let $q_{i} \in M_{g, n}^{\rm ord}$ and $q'_{i} := \omega(q_{i}) \in M_{g, n, \mbF_{p}}$. Then we have $V_{q_{1}} \supseteq_{fe} V_{q_{2}}$ if and only if $V_{q'_{1}} \supseteq V_{q'_{2}}$. In particular, we have $V_{q_{1}} =_{fe} V_{q_{2}}$ if and only if $V_{q'_{1}} = V_{q'_{2}}$. Namely, we have $V_{q_{1}}=_{fe} V_{q_{2}}$ if and only if $U_{X_{q_{1}}} \cong U_{X_{q_{2}}}$ as schemes.
\end{proposition}

\begin{proof}
Suppose that $q_{i}$, $i\in \{1, 2\}$, is a closed point of $M_{g, n}^{\rm ord}$. If $V_{q_{1}} \supseteq_{fe} V_{q_{2}}$, we see immeidately  $q_{1} \sim q_{2}$. Thus, we obtain $U_{X_{q_{1}}} \cong U_{X_{q_{2}}}$ as schemes. This means $q'_{1}=q'_{2}$. Conversely, if $V_{q'_{1}} \supseteq V_{q'_{2}}$, then we have $q'_{1}=q'_{2}$. Thus, we obtain $q_{1}\sim q_{2}$.

Suppose that $q_{i}$, $i\in \{1, 2\}$, is an aribtrary point of $M_{g, n}^{\rm ord}$. If $V_{q_{1}} \supseteq_{fe} V_{q_{2}}$, then the case of closed points implies  $V_{q'_{1}}^{\rm cl} \supseteq V_{q'_{2}}^{\rm cl}.$ Since $V_{q'_{1}}$ and $V_{q'_{2}}$ are irreducible, we obtain  $V_{q'_{1}} \supseteq V_{q'_{2}}.$ Conversely, if $V_{q'_{1}} \supseteq V_{q'_{2}}$, we note that $V_{q_{i}}$ is an irreducible component of $(\omega)^{-1}(V_{q'_{i}})$. Then the case of closed points implies  $V_{q_{1}} \supseteq_{fe} V_{q_{2}}$.
\end{proof}

\subsubsection{}
Denote by ${\rm Hom}^{\rm op}_{\rm pg}(-,-)$ the set of open continuous homomorphisms of profinite groups, and by $\text{Isom}_{\rm pg}(-,-)$ the set of isomorphisms of profinite groups. We have the following conjecture.

\begin{weakhomGC}
Let $q_{i} \in M_{g, n}$ (resp. $q_{i} \in M_{g, n,  \mbF_{p}}$), $i\in \{1, 2\}$, be an arbitrary point. Then
we have $${\rm Hom}^{\rm op}_{\rm pg}(\pi_{1}^{\rm t}(q_{1}), \pi_{1}^{\rm t}(q_{2}))$$ is non-empty if and only if $V_{q_{1}}\supseteq_{fe}V_{q_{2}}$ (resp. $V_{q_{1}}\supseteq V_{q_{2}}$).
\end{weakhomGC}
\noindent
The weak Hom-version conjecture means that the {\it topological structures} of the moduli spaces of smooth pointed stable curves can be understood by the tame fundamental groups of curves. In particular, the weak Hom-version conjecture implies the following conjecture which was essentially formulated by Tamagawa (\cite{T3}).



\begin{weakisomGC}
Let $q_{i} \in M_{g, n}$ (resp. $q_{i} \in M_{g, n,  \mbF_{p}}$), $i\in \{1, 2\}$, be an arbitrary point.  Then we have $${\rm Isom}_{\rm pg}(\pi_{1}^{\rm t}(q_{1}), \pi_{1}^{\rm t}(q_{2}))$$ is non-empty if and only if $V_{q_{1}}=_{fe}V_{q_{2}}$ (resp. $V_{q_{1}}=V_{q_{2}}$).
\end{weakisomGC}
\noindent
The weak Isom-version conjecture means that the {\it set structures} of the moduli spaces of smooth pointed stable curves can be understood by the tame fundamental groups of curves.

\subsection{The pointed collection conjecture}\label{pcc} In this subsection, we formulate the second conjecture of the present paper which we call ``the pointed collection conjecture''.
 We maintain the notation introduced in \ref{moduli212}.

\subsubsection{}
Let $q$ be an arbitrary point of $M_{g, n}^{\rm ord}$ and $G \in \pi^{\rm t}_{A}(q)$ an arbitrary finite group. We put $$U_{G}:= \{q' \in M_{g, n}^{\rm ord} \ | \ G \in \pi_{A}^{\rm t}(q')\} \subseteq M_{g, n}^{\rm ord}.$$ Then we have the following result.

\begin{proposition}\label{pro-6}
Let $q$ be an arbitrary point of $M_{g, n}^{\rm ord}$ and $G \in \pi^{\rm t}_{A}(q)$ an arbitrary finite group. Then the set $U_{G}$ contains an open neighborhood of $q$ in $M_{g, n}^{\rm ord}$.
\end{proposition}

\begin{proof}
Proposition \ref{pro-6} was proved by K. Stevenson when $n=0$ and $q$ is a closed point of $M_{g, 0}$ (cf. \cite[Proposition 4.2]{Ste}). Moreover, by similar arguments to the arguments given in the proof of \cite[Proposition 4.2]{Ste}, Proposition \ref{pro-6} also holds for $n\geq 0$.
\end{proof}


\begin{definition}\label{def-3} We denote by $q_{\rm gen}$ the generic point of $M_{g, n}^{\rm ord}$, and let $$\mcC \subseteq \pi_{A}^{\rm t}(q_{\rm gen})=\bigcup_{q \in M_{g, n}^{\rm ord, cl}}\pi_{A}^{\rm t}(q)$$ be a subset of $\pi_{A}^{\rm t}(q_{\rm gen})$. We shall say that $\mcC$ is a {\it pointed collection} if the following conditions are satisfied: (i) $0<\#((\bigcap_{G\in \mcC}U_{G}) \cap M_{g, n}^{\rm ord, cl}) < \infty$; (ii) $U_{G'} \cap (\bigcap_{G\in \mcC}U_{G})\cap M_{g, n}^{\rm ord, cl} =\emptyset$ for each $G' \in \pi_{A}^{\rm t}(q_{\rm gen})$ such that $G' \not\in \mcC$.

On the other hand, for each closed point $t \in M_{g, n}^{\rm ord, cl}$, we may define a set associated to $t$ as follows: $$\mcC_{t}:=\{G \in \pi_{A}^{\rm t}(q_{\rm gen})\ | \ t\in U_{G}\}.$$ Note that, if $t \in V^{\rm cl}_{q}$, then $\mcC_{t} \subseteq \pi_{A}^{\rm t}(q)$.  Moreover, we denote by $$\msC_{q}:= \{\mcC \ \text{is a pointed collection} \ | \ \mcC \subseteq \pi^{\rm t}_{A}(q)\}.$$ 
\end{definition}

\subsubsection{}
At present, no published results are known concerning the weak Hom-version conjecture (or the weak Isom-version conjecture) for {\it non-closed} points. The main difficulty of proving the weak Hom-version conjecture (or the weak Isom-version conjecture) for non-closed points of $M_{g, n}^{\rm ord}$ is the following: For each $q \in M_{g, n}^{\rm ord}$, we {\it do not} know how to reconstruct the tame fundamental groups of closed points of $V_{q}$ group-theoretically from $\pi_{1}^{\rm t}(q)$.

Once the tame fundamental groups of the closed points of $V_q$ can be reconstructed group-theoretically from $\pi_{1}^{\rm t}(q)$, then the weak Hom-version conjecture for closed points of $M_{g, n}^{\rm ord}$ implies that the set of closed points of $V_{q}$ can be reconstructed group-theoretically from $\pi_{1}^{\rm t}(q)$. Thus, the weak Hom-version conjecture for non-closed points of $M_{g, n}^{\rm ord}$ can be deduced from the weak Hom-version conjecture for closed points of $M_{g, n}^{\rm ord}$.

Let $q \in M_{g, n}^{\rm ord}$. Since the isomorphism class of $\pi_{1}^{\rm t}(q)$ as a profinite group can be determined by the set $\pi_{A}^{\rm t}(q)$, the following conjecture tell us how to reconstruct group-theoretically the set of finite quotients of a closed point of $V_{q}$ from $\pi_{A}^{\rm t}(q)$ (or $\pi_{1}^{\rm t}(q)$).

\begin{pcc}
For each $t \in M_{g, n}^{\rm ord, cl}$, the set $\mcC_{t}$ associated to $t$ is a pointed collection. Moreover, let $q\in M_{g, n}^{\rm ord}$. Then the natural map $${\rm colle}_{q}: \msV_{q}^{\rm cl} \rightarrow \msC_{q}, \ [t] \mapsto \mcC_{t},$$ is a bijection, where $[t]$ denotes the image of $t$ in $\msV_{q}^{\rm cl}:= V_{q}^{\rm cl}/\sim_{fe}$.
\end{pcc}
\noindent
Write $q'\in M_{g, n, \mbF_{p}}$ for the image $\omega(q) $. Then we have $\msV^{\rm cl}_{q}=V_{q'}^{\rm cl}$. This means that the pointed collection conjecture holds if and only if the weak Hom-version conjecture holds.

\section{Reconstructions of marked points}\label{mpanabelian}
The main purposes of the present section are as follows: We will give a new mono-anabelian reconstruction of ${\rm Ine}(\pi_{1}^{\rm t}(U_{X}))$, and prove that the mono-anabelian reconstruction (i.e., the group-theoretical algorithm) {\it is compatible with any open continuous homomorphisms} of tame fundamental groups of smooth pointed stable curves with a fixed type.

\subsection{Anabelian reconstructions}

 We maintain the notation introduced in \ref{curves}.

\subsubsection{}
Let us recall the definitions concerning  ``anabelian reconstructions".

\begin{definition}\label{definition 1}
Let $\mcF$ be a geometric object and $\Pi_{\mcF}$ a profinite group associated to the  object $\mcF$. Suppose that we are given an invariant $\text{Inv}_{\mcF
}$ depending on the isomorphism class of $\mcF$ (in a certain category), and that we are given an additional structure $\text{Add}_{\mcF}$ (e.g. a family of subgroups, a family of quotient groups) on the profinite group $\Pi_{\mcF}$ depending functorially on $\mcF$.

We shall say that $\text{Inv}_{\mcF}$ can be {\it mono-anabelian reconstructed from} $\Pi_{\mcF}$ if there exists a group-theoretical algorithm whose input datum is $\Pi_{\mcF}$, and whose output datum is $\text{Inv}_{\mcF}$. We shall say that $\text{Add}_{\mcF}$  {\it can be mono-anabelian reconstructed from} $\Pi_{\mcF}$ if there exists a group-theoretical algorithm whose input datum is $\Pi_{\mcF}$, and whose output datum is $\text{Add}_{\mcF}$.

Let $\mcF_{i},\ i \in \{1, 2\},$ be a geometric object
and $\Pi_{\mcF_{i}}$ a profinite group associated to $\mcF_{i}$. Suppose that we are given an additional structure $\text{Add}_{\mcF_{i}}$ on the profinite group $\Pi_{\mcF_{i}}$ depending functorially on $\mcF_{i}$. We shall say that a map (or a morphism) $\text{Add}_{\mcF_{1}} \rightarrow \text{Add}_{\mcF_{2}}$  {\it can be mono-anabelian reconstructed} from an open continuous homomorphism $\Pi_{\mcF_{1}} \rightarrow \Pi_{\mcF_{2}}$ if there exists a group-theoretical algorithm whose input datum is $\Pi_{\mcF_{1}} \rightarrow \Pi_{\mcF_{2}}$, and whose output datum is $\text{Add}_{\mcF_{1}} \rightarrow \text{Add}_{\mcF_{2}}$.

\end{definition}

\subsubsection{}\label{unicov313}
Let $K$ be the function field of $X$, and let $\widetilde K$ be the maximal Galois extension of $K$ in a fixed separable closure of $K$, unramified over $U_{X}$ and at most tamely ramified over $D_{X}$. Then we may identify $\pi_{1}^{\rm t}(U_{X})$ with $\text{Gal}(\widetilde K/K)$. We define the universal tame covering of $(X, D_{X})$ associated to $\pi_{1}^{\rm t}(U_{X})$ to be $(\widetilde X, D_{\widetilde X})$, where  $\widetilde X$ denotes the normalization of $X$ in $\widetilde K$, and $D_{\widetilde X}$ denotes the inverse image of $D_{X}$ in $\widetilde X$. Then there is a natural action of  $\pi_{1}^{\rm t}(U_{X})$ on $(\widetilde X, D_{\widetilde X})$. For each $\widetilde e \in D_{\widetilde X}$, we denote by $I_{\widetilde e}$ the inertia subgroup of $\pi_{1}^{\rm t}(U_{X})$ associated to $\widetilde e$ (i.e., the stabilizer of $\widetilde e$ in $\pi_{1}^{\rm t}(U_{X})$).  Then we have $I_{\widetilde e} \cong \widehat \mbZ(1)^{p'}$, where $\widehat \mbZ(1)^{p'}$ denotes the prime-to-$p$ part of $\widehat \mbZ(1)$. The following result was proved by Tamagawa (\cite[Lemma 5.1 and Theorem 5.2]{T4}).

\begin{proposition}\ \label{proposition 1}
\begin{enumerate}
  \item The type $(g_{X}, n_{X})$ can be mono-anabelian reconstructed from $\pi_{1}^{\rm t}(U_{X})$.
  \item Let $\widetilde e$ and $\widetilde e'$ be two points of $D_{\widetilde X}$ distinct from each other. Then the intersection of $I_{\widetilde e}$ and $I_{\widetilde e'}$ is trivial in $\pi_{1}^{\rm t}(U_{X})$. Moreover, the map $$D_{\widetilde X} \rightarrow {\rm Sub}(\pi_{1}^{\rm t}(U_{X})), \ \widetilde e \mapsto I_{\widetilde e}, $$ is an injection, where ${\rm Sub}(-)$ denotes the set of closed subgroups of $(-)$.
  \item Write ${\rm Ine}(\pi_{1}^{\rm t}(U_{X}))$ for the set of inertia subgroups in $\pi_{1}^{\rm t}(U_{X})$, namely the image of the map $D_{\widetilde X} \rightarrow {\rm Sub}(\pi_{1}^{\rm t}(U_{X}))$. Then ${\rm Ine}(\pi_{1}^{\rm t}(U_{X}))$ can be mono-anabelian reconstructed from $\pi_{1}^{\rm t}(U_{X})$. In particular, the set of marked points $D_{X}$ and $\pi_{1}(X)$ can be mono-anabelian reconstructed from $\pi_{1}^{\rm t}(U_{X})$.
\end{enumerate}

\end{proposition}

\subsection{The set of marked points}\label{sec-2}
We maintain the notation introduced in \ref{curves}. Moreover, we suppose that $g_{X} \geq 2$ and $n_{X}>0$.

\subsubsection{}We will prove that the set of marked points can be regarded as a quotient set of a set of cohomological classes of a suitable covering of curves (i.e. Proposition \ref{pro-2}).
The main idea is the following: By taking a suitable \'etale covering with a prime degree $f: (Y, D_{Y}) \rightarrow (X, D_{X})$, for every marked point $x \in D_{X}$, there exists a set of tame coverings with a prime degree which is totally ramified over the inverse image $f^{-1}(x)$. Then $x$ can be regarded as the set of cohomological classes corresponding to such coverings.

\subsubsection{}\label{triple}
Let $h: (W, D_{W}) \rightarrow (X, D_{X})$ be a connected Galois tame covering over $k$. We put $$\text{Ram}_{h} := \{e \in D_{X} \ | \ h \ \text{is ramified over} \ e \}.$$ Let $(Y, D_{Y})$ be a smooth pointed stable curve over $k$. We shall say that $$(\ell, d, f: (Y, D_{Y}) \rightarrow (X, D_{X}))$$ is {\it an mp-triple associated to $(X, D_{X})$} if the following conditions hold: (i) $\ell$ and $d$ are prime numbers distinct from each other such that $(\ell, p)=(d, p)=1$ and $\ell \equiv 1 \ (\text{mod}\ d)$; then all $d$th roots of unity are contained in $\mbF_{\ell}$; (ii) $f$ is a Galois {\it\'etale} covering over $k$ whose Galois group is isomorphic to $\mu_{d}$, where $\mu_{d} \subseteq \mbF_{\ell}^{\times}$ denotes the subgroup of $d$th roots of unity. Here, ``{mp}'' means ``marked points''.

Then we have a natural injection $H^{1}_{\text{\'et}}(Y, \mbF_{\ell})\hookrightarrow H^{1}_{\text{\'et}}(U_Y, \mbF_{\ell})$ induced by the natural surjection $\pi_{1}^{\rm t}(U_Y) \twoheadrightarrow \pi_{1}(Y)$. Note that every non-zero element of $H^{1}_{\text{\'et}}(U_Y, \mbF_{\ell})$ induces a connected Galois tame covering of $(Y, D_{Y})$ of degree $\ell$. We obtain an exact sequence $$0\rightarrow H^{1}_{\text{\'et}}(Y, \mbF_{\ell})\rightarrow H^{1}_{\text{\'et}}(U_Y, \mbF_{\ell}) \rightarrow \text{Div}^{0}_{D_{Y}}(Y)\otimes \mbF_{\ell}\rightarrow 0$$ with a natural action of $\mu_{d}$.

\subsubsection{}\label{sec31aaa}
Let $(\text{Div}^{0}_{D_{Y}}(Y)\otimes \mbF_{\ell})_{\mu_{d}} \subseteq \text{Div}^{0}_{D_{Y}}(Y)\otimes \mbF_{\ell}$ be the subset of elements on which $\mu_{d}$ acts via the character $\mu_{d} \hookrightarrow \mbF_{\ell}^{\times}$ and $M^{*}_{Y} \subseteq H^{1}_{\text{\'et}}(U_Y, \mbF_{\ell})$ the subset of elements whose images are non-zero elements of $(\text{Div}^{0}_{D_{Y}}(Y)\otimes \mbF_{\ell})_{\mu_{d}}$. For each $\alpha \in M^{*}_{Y}$, write $g_{\alpha}: (Y_{\alpha}, D_{Y_{\alpha}}) \rightarrow(Y, D_{Y})$ for the tame covering induced by $\alpha$. We define $\epsilon: M_{Y}^{*} \rightarrow \mbZ$, where $ \epsilon(\alpha):=\#D_{Y_{\alpha}}$. Denote by $$M_{Y}:=\{\alpha \in M_{Y}^{*}\ | \ \#\text{Ram}_{g_{\alpha}}=d\}=\{\alpha \in M_{Y}^{*}\ | \ \epsilon(\alpha)=\ell(dn_{X}-d)+d\}.$$ Note that $M_{Y}$ is non-empty.

For each $\alpha \in M_{Y}$, since the image of $\alpha$ is contained in $(\text{Div}^{0}_{D_{Y}}(Y)\otimes \mbF_{\ell})_{\mu_{d}}$, we obtain that the action of $\mu_{d}$ on $\text{Ram}_{g_{\alpha}} \subseteq D_{Y}$ is transitive. Thus, there exists a unique marked point $e_{\alpha} \in D_{X}$ such that $f(y)=e_{\alpha}$ for each $y \in \text{Ram}_{g_{\alpha}}$.

For each $e \in D_{X}$, we put $$M_{Y, e}:= \{\alpha \in M_{Y} \ | \ g_{\alpha} \ \text{is ramified over} \ f^{-1}(e) \}.$$ Then, for any marked points $e, e' \in D_{X}$ distinct from each other, we have $M_{Y, e} \cap M_{Y, e'}=\emptyset$ and the disjoint union $$M_{Y}=\bigsqcup_{e \in D_{X}} M_{Y, e}.$$

\subsubsection{}\label{315}
Next, we define a pre-equivalence relation $\sim$ on $M_{Y}$ as follows:
Let $\alpha, \beta \in M_{Y}$. Then $\alpha \sim \beta$ if $\lambda\alpha+\mu\beta \in M_{Y}$ for each $\lambda, \mu \in \mbF^{\times}_{\ell}$ for which $\lambda\alpha+\mu\beta \in M_{Y}^{*}$. Then we have the following proposition.

\begin{proposition}\label{pro-2}
The pre-equivalence relation $\sim$ on $M_{Y}$ is an equivalence relation. Moreover, the map $$\vartheta_{X}: M_{Y}/\sim \rightarrow D_{X}, \ [\alpha] \mapsto e_{\alpha},$$ is a bijection, where $[\alpha]$ denotes the image of $\alpha$ in $M_{Y}/\sim$.
\end{proposition}

\begin{proof}

Let $\beta, \gamma \in M_{Y}.$ If $\text{Ram}_{g_{\beta}}=\text{Ram}_{g_{\gamma}}$, then, for each $\lambda, \mu \in \mbF_{\ell}^{\times}$ for which $\lambda\beta+\mu\gamma \neq 0$, we have $\text{Ram}_{g_{\lambda\beta+\mu\gamma}}=\text{Ram}_{g_{\beta}}=\text{Ram}_{g_{\gamma}}.$ Thus we obtain that $\beta \sim \gamma$. On the other hand, if $\beta \sim \gamma$, we have $\text{Ram}_{g_{\beta}}=\text{Ram}_{g_{\gamma}}$. Otherwise, we have $\#\text{Ram}_{g_{\beta+\gamma}}=2d$. This means that $\beta \sim \gamma \ \text{if and only if} \ \text{Ram}_{g_{\beta}}=\text{Ram}_{g_{\gamma}}.$ Then $\sim$ is an equivalence relation on $M_{Y}$.

Let us prove that $\vartheta_{X}$ is a bijection. It is easy to see that $\vartheta_{X}$ is an injection. On the other hand, for each $e \in D_{X}$, the structure of the maximal pro-$\ell$ tame fundamental groups implies that we may construct a connected tame Galois covering of $h: (Z, D_{Z}) \rightarrow(Y, D_{Y})$ such that the element of $H^{1}_{\text{\'et}}(U_Y, \mbF_{\ell})$ induced by $h$ is contained in $M_{Y}$. Then $\vartheta_{X}$ is a surjection. This completes the proof of Proposition \ref{pro-2}.
\end{proof}

\begin{remark}\label{rem-2-1}
We claim that the set $M_{Y}/\sim$ does not depend on the choices of mp-triples associated to $(X, D_{X})$. Let $$(\ell^{*}, d^{*}, f^{*}: (Y^{*}, D_{Y^{*}}) \rightarrow(X, D_{X}))$$ be an arbitrary mp-triple associated to $(X, D_{X})$. Hence we obtain a resulting set $M_{Y^{*}}/\sim$ and a natural bijection $\vartheta_{X}^{*}: M_{Y^{*}}/\sim \rightarrow D_{X}.$ We will prove that there exists a natural bijection $\delta: M_{Y^{*}}/\sim \xrightarrow{\simeq} M_{Y}/\sim$ such that $\vartheta_{X}^{*}=\vartheta_{X}\circ\delta$.

First, suppose that $\ell \neq \ell^{*}$ and $d\neq d^{*}$. Then we may construct a natural bijection $\delta: M_{Y^{*}}/\sim \xrightarrow{\simeq} M_{Y}/\sim$ as follows. Let $\alpha\in M_{Y}$ and $\alpha^{*}\in M_{Y^{*}}$. Write $(Y_{\alpha}, D_{Y_{\alpha}}) \rightarrow (Y, D_{Y})$ and $(Y_{\alpha^{*}}, D_{Y_{\alpha^{*}}}) \rightarrow (Y^{*}, D_{Y^{*}})$ for the Galois tame coverings induced by $\alpha$ and $\alpha^{*}$, respectively. We consider the following fiber product in the category of smooth pointed stable curves $$(Y_{\alpha}, D_{Y_{\alpha}})\times_{(X, D_{X})} (Y_{\alpha^{*}}, D_{Y_{\alpha^{*}}})$$ which is a smooth pointed stable curve over $k$. Thus, we obtain a connected tame covering $(Y_{\alpha}, D_{Y_{\alpha}})\times_{(X, D_{X})} (Y_{\alpha^{*}}, D_{Y_{\alpha^{*}}}) \rightarrow(X, D_{X})$ of degree $dd^{*}\ell\ell^{*}$. Then it is easy to check that $\vartheta_{X}([\alpha])=\vartheta_{X}^{*}([\alpha^{*}])$ if and only if the cardinality of the set of marked points of $(Y_{\alpha}, D_{Y_{\alpha}})\times_{(X, D_{X})} (Y_{\alpha^{*}}, D_{Y_{\alpha^{*}}})$ is equal to $dd^{*}(\ell\ell^{*}(n_{X}-1)+1).$ We put $[\alpha] := \delta([\alpha^{*}])$ if $\vartheta_{X}([\alpha])=\vartheta_{X}^{*}([\alpha^{*}])$. Moreover, by the construction above, we obtain that $\vartheta_{X}^{*}=\vartheta_{X}\circ\delta$. In the general case, we may choose an mp-triple $$(\ell^{**}, d^{**}, f^{**}: (Y^{**}, D_{Y^{**}}) \rightarrow (X, D_{X}))$$ associated to $(X, D_{X})$ such that $\ell^{**} \neq \ell$, $\ell^{**} \neq \ell^{*}$, $d^{**}\neq d$, and $d^{**} \neq d^{*}$. Hence we obtain a resulting set $M_{Y^{**}}/\sim$ and a natural bijection $\vartheta_{X}^{**}: M_{Y^{**}}/\sim \rightarrow D_{X}$. Then the proof given above implies that there are natural bijections $\delta_{1}: M_{Y^{**}}/\sim \xrightarrow{\simeq} M_{Y}/\sim$ and $\delta_{2}: M_{Y^{**}}/\sim \xrightarrow{\simeq} M_{Y^{*}}/\sim$. Thus, we may put $$\delta:= \delta_{1} \circ\delta_{2}^{-1}: M_{Y^{*}}/\sim \xrightarrow{\simeq} M_{Y}/\sim.$$
\end{remark}

\begin{remark}\label{rem-2-2}
Let $H \subseteq \pi_{1}^{\rm t}(U_{X})$ be an arbitrary open normal subgroup and $f_{H}: (X_{H}, D_{X_{H}}) \rightarrow (X, D_{X})$ the Galois tame covering over $k$ induced by the natural inclusion $H \hookrightarrow \pi_{1}^{\rm t}(U_{X})$. Let $$(\ell, d, f: (Y, D_{Y}) \rightarrow (X, D_{X}))$$ be an mp-triple associated to $(X, D_{X})$ such that $(\#(\pi_{1}^{\rm t}(U_{X})/H), \ell)=(\#(\pi_{1}^{\rm t}(U_{X})/H), d)=1$. Then we obtain an mp-triple $$(\ell, d, g: (Z, D_{Z}):=(Y, D_{Y}) \times_{(X, D_{X})}
(X_{H}, D_{X_{H}})\rightarrow (X_{H}, D_{X_{H}}))$$ associated to $(X_{H}, D_{X_{H}})$ induced by $(\ell, d, f: (Y, D_{Y}) \rightarrow (X, D_{X}))$, where $(Y, D_{Y}) \times_{(X, D_{X})}
(X_{H}, D_{X_{H}})$ denotes the fiber product in the category of smooth pointed stable curves. The mp-triple associated to $(X_{H}, D_{X_{H}})$ induces a set $M_{Z}/\sim$ which can be identified with the set of marked points $D_{X_{H}}$ of $(X_{H}, D_{X_{H}})$ by applying Proposition \ref{pro-2}. Moreover, for each $e_{X} \in D_{X}$ and each $\alpha_{Y,e_{X}} \in M_{Y, e_{X}}$, $\alpha_{Y, e_{X}}$ induces an element $$\alpha_{Z}=\sum_{e_{X_{H}} \in f^{-1}_{H}(e_{X})} \alpha_{Z, e_{X_{H}}}$$ over $(Z, D_{Z})$ via the natural morphism $(Z, D_{Z}) \rightarrow (Y, D_{Y})$, where $\alpha_{Z, e_{X_{H}}} \in M_{Z, e_{X_{H}}}$. On the other hand, for each $e'_{X_{H}} \in D_{X_{H}}$ and each $e'_{X} \in D_{X}$, we have that $f_{H}(e'_{X_{H}})=e'_{X}$ if and only if there exists an element $\alpha_{Y, e'_{X}} \in M_{Y, e'_{X}}$ such that the following two conditions hold:
 \begin{itemize}
   \item the element $\alpha'_{Z}$, induced by $\alpha_{Y, e'_{X}}$ via the natural morphism $(Z, D_{Z}) \rightarrow (Y, D_{Y})$, can be represented by a linear combination $$\alpha_{Z}'=\sum_{e_{X_{H}} \in S_{X_{H}}} \alpha'_{Z, e_{X_{H}}},$$ where $S_{X_{H}}$ is a subset of $D_{X_{H}}$, and $\alpha_{Z, e_{X_{H}}} \in M_{Z, e_{X_{H}}}$;
   \item  $e'_{X_{H}} \in S_{X_{H}}$.
 \end{itemize}

\end{remark}

\begin{lemma}\label{lem-1}
Let $(\ell, d, f: (Y, D_{Y}) \rightarrow (X, D_{X}))$ be a triple associated to $(X, D_{X})$ and $g_{Y}$ the genus of $Y$. Then we have $\#(M_{Y, e})=\ell^{2g_{Y}+1}-\ell^{2g_{Y}}, \ e \in D_{X}.$  Moreover, we have $\#(M_{Y})=n_{X}(\ell^{2g_{Y}+1}-\ell^{2g_{Y}}).$
\end{lemma}

\begin{proof}
Let $e \in D_{X}$. Write $D_{e} \subseteq D_{Y}$ for the set $f^{-1}(e)$. Then $M_{Y, e}$ can be naturally regarded as a subset of $H^{1}_{\text{\rm \'et}}(Y \setminus D_{e}, \mbF_{\ell} )$ via the natural open immersion $Y\setminus D_{e}  \hookrightarrow Y.$ Write $L_{e}$ for the $\mbF_{\ell}$-vector space generated by $M_{Y, e}$ in $H^{1}_{\text{\rm \'et}}(Y \setminus D_{e}, \mbF_{\ell})$. Then we have $M_{Y, e}=\ L_{e}\setminus H^{1}_{\text{\rm \'et}}(Y, \mbF_{\ell}).$ Write $H_{e}$ for the quotient $L_{e}/H^{1}_{\text{\rm \'et}}(Y, \mbF_{\ell})$. We have an exact sequence as follows:
$$0 \rightarrow H^{1}_{\text{\rm \'et}}(Y, \mbF_{\ell}) \rightarrow L_{e}\rightarrow H_{e}\rightarrow0.$$  Since the action of $\mu_{d}$ on $f^{-1}(e)$ is transitive, we obtain $\text{dim}_{\mbF_{\ell}}(H_{e})=1.$ On the other hand, since $\text{dim}_{\mbF_{\ell}}(H^{1}_{\text{\rm \'et}}(Y, \mbF_{\ell}))=2g_{Y},$ we obtain $\#(M_{Y, e})=\ell^{2g_{Y}+1}-\ell^{2g_{Y}}.$ Thus, we have $\#(M_{Y})=n_{X}(\ell^{2g_{Y}+1}-\ell^{2g_{Y}}).$ This completes the proof of the lemma.
\end{proof}


\bigskip

\subsection{Reconstructions of inertia subgroups}\label{sec-3}
 We maintain the notation introduced in \ref{curves}.

\subsubsection{}We will prove that the inertia subgroups of marked points can be mono-anabelian reconstructed from $\pi_{1}^{\rm t}(U_{X})$ (i.e. Proposition \ref{them-1}).
The main idea is as follows: Let $H \subseteq \pi_{1}^{\rm t}(U_{X})$ be an arbitrary normal open subgroup and $(X_{H}, D_{X_{H}}) \rightarrow(X, D_{X})$ the tame covering corresponding to $H$. Firstly, by using some numerical conditions induced by the Riemann-Hurwitz formula, the \'etale fundamental group $\pi_{1}(X)$ can be mono-anabelian reconstructed from $\pi^{\rm t}_{1}(U_{X})$. Then the results obtained in Section \ref{sec-2} implies that $D_{X}$ can be mono-anabelian reconstructed  from $\pi^{\rm t}_{1}(U_{X})$. Moreover, $D_{X_{H}}$ can also be mono-anabelian reconstructed from $H$. Secondly, since the natural injection $H \hookrightarrow\pi^{\rm t}_{1}(U_{X})$ induces a map of sets of cohomological classes obtained in Section \ref{sec-2}, we obtain that the natural map $D_{X_{H}} \rightarrow D_{X}$ can be mono-anabelian reconstructed  from $H \hookrightarrow \pi^{\rm t}_{1}(U_{X})$. Thus, by taking a cofinal system of open normal subgroups of $\pi^{\rm t}_{1}(U_{X})$, we obtain a new mono-anabelian reconstruction of $\text{Ine}(\pi_{1}^{\rm t}(U_{X}))$.

\subsubsection{}
First, we have the following lemma.
\begin{lemma}\label{lem-2}\
\begin{enumerate}
  \item The prime number $p$ (i.e., the characteristic of $k$) can be mono-anabelian reconstructed from $\pi_{1}^{\rm t}(U_{X})$.
  \item The \'etale fundamental group $\pi_{1}(X)$ can be mono-anabelian reconstructed from $\pi_{1}^{\rm t}(U_{X})$.
\end{enumerate}

\end{lemma}

\begin{proof}
(1) Let $\mfP$ be the set of prime numbers, and let $Q$ be an arbitrary open subgroup of $\pi_{1}^{\rm t}(U_{X})$ and $r_{Q}$ an integer such that $$\#\{l \in \mfP \ | \ r_{Q}=\text{dim}_{\mbF_{l}}(Q^{\rm ab}\otimes\mbF_{l})\}=\infty.$$ Then we see immediately that the characteristic of $k$ is the unique prime number $p$ such that there exists an open subgroup $T \subseteq \pi_{1}^{\rm t}(U_{X})$ and $r_{T}\neq \text{dim}_{\mbF_{p}}(T^{\rm ab}\otimes\mbF_{p}).$

(2) Let $H$ be an arbitrary open normal subgroup of $\pi_{1}^{\rm t}(U_{X})$. We denote by $(X_{H}, D_{X_{H}})$ the smooth pointed stable curve of type $(g_{X_{H}}, n_{X_{H}})$ over $k$ induced by $H$, and denote by $f_{H}: (X_{H}, D_{X_{H}}) \rightarrow(X, D_{X})$ the morphism of smooth pointed stable curves over $k$ induced by the natural inclusion $H \hookrightarrow\pi_{1}^{\rm t}(U_{X})$. We note that $f_{H}$ is  \'etale if and only if $g_{X_{H}}-1=\#(\pi_{1}^{\rm t}(U_X)/H)(g_{X}-1).$  We put
 \begin{align*}
  \text{Et}(\pi_{1}^{\rm t}(U_X)):=\{H \subseteq \pi_{1}^{\rm t}(U_X) \ \text{is an open normal subgroup}: g_{X_{H}}-1=\#(\pi_{1}^{\rm t}(U_X)/H)(g_{X}-1)\}.
 \end{align*}
 Moreover, Proposition \ref{proposition 1} (1) implies that $g_{X_{H}}$ and $g_{X}$ can be mono-anabelian reconstructed from $H$ and $\pi_{1}^{\rm t}(U_X)$, respectively. Then the set $\text{Et}(\pi_{1}^{\rm t}(U_X))$ can be mono-anabelian reconstructed from $\pi_{1}^{\rm t}(U_X)$. We obtain that $$\pi_{1}(X)= \pi_{1}^{\rm t}(U_{X})/\bigcap_{H \in \text{Et}(\pi_{1}^{\rm t}(U_{X}))}H .$$ This completes the proof of the lemma.
\end{proof}

\subsubsection{}\label{gpmptriple}
Suppose $g_{X} \geq 2$. Let us define a group-theoretical object corresponding to an mp-triple which was introduced in \ref{triple}. We shall say that $(\ell, d, y)$ is {\it an mp-triple associated to $\pi_{1}^{\rm t}(U_{X})$} if the following two conditions hold:
\begin{itemize}
  \item  $\ell$ and $d$ are prime numbers distinct from each other such that $(\ell, p)=(d, p)=1$ and $\ell \equiv 1 \ (\text{mod}\ d)$; then all $d$-th roots of unity are contained in $\mbF_{\ell}$;
  \item $y \in \text{Hom}(\pi_{1}(X), \mu_{d})$ such that $y\neq 0$, where $\mu_{d} \subseteq \mbF_{\ell}^{\times}$ denotes the subgroup of $d$-th roots of unity.
\end{itemize}

\subsubsection{}
Moreover, by applying Lemma \ref{lem-2}, there is a triple $(\ell, d, y)$ associated to $\pi_{1}^{\rm t}(U_{X})$ which can be mono-anabelian reconstructed from $\pi_{1}^{\rm t}(U_{X})$. Let $f: (Y, D_{Y}) \rightarrow(X, D_{X})$ be a Galois \'etale covering induced by $y$. Then we see immediately that $(\ell, d, f: (Y, D_{Y}) \rightarrow (X, D_{X}))$ is an mp-triple associated to $(X, D_{X})$ defined in \ref{triple}. We denote by $\pi_{1}^{\rm t}(U_{Y})$ the kernel of the composition of the surjections $\pi_{1}^{\rm t}(U_{X}) \twoheadrightarrow \pi_{1}(X) \overset{y}{\twoheadrightarrow} \mu_{d}.$ Since $H^{1}_{\text{\'et}}(Y, \mbF_{\ell}) \cong \text{Hom}(\pi_{1}(Y), \mbF_{\ell})$ and $H^{1}_{\text{\'et}}(U_Y, \mbF_{\ell}) \cong \text{Hom}(\pi^{\rm t}_{1}(U_{Y}), \mbF_{\ell})$, Lemma \ref{lem-2} implies immediately that the following exact sequence $$0\rightarrow H^{1}_{\text{\'et}}(Y, \mbF_{\ell})\rightarrow H^{1}_{\text{\'et}}(U_Y, \mbF_{\ell}) \rightarrow \text{Div}^{0}_{D_{Y}}(Y)\otimes \mbF_{\ell}\rightarrow 0$$ can be mono-anabelian reconstructed from $\pi_{1}^{\rm t}(U_{Y})$. Thus, Proposition \ref{proposition 1} (1) implies that the set $M_{Y}/\sim$ defined in \ref{315} can be mono-anabelian reconstructed from $\pi_{1}^{\rm t}(U_{Y})$. Note that, by Remark \ref{rem-2-1}, the set $M_{Y}/\sim$ does not depend on the choices of mp-triples. Then we  put $$D_{X}^{\rm gp}:= M_{Y}/\sim,$$ where ``$\rm{gp}$" means ``group-theoretical". By Proposition \ref{pro-2}, we may identify $D^{\rm gp}_{X}$ with the set of marked points $D_{X}$ of $(X, D_{X})$ via the bijection $\vartheta_{X}: D^{\rm gp}_{X} \xrightarrow{\simeq}D_{X}$ defined in Proposition \ref{pro-2}.


\begin{proposition}\label{pro-3}
Let $H \subseteq \pi_{1}^{\rm t}(U_{X})$ be an arbitrary open normal subgroup and $$f_{H}: (X_{H}, D_{X_{H}}) \rightarrow (X, D_{X})$$ the morphism of smooth pointed stable curves over $k$ induced by the natural inclusion $H \hookrightarrow \pi_{1}^{\rm t}(U_{X})$. Suppose  $g_{X} \geq 2$. Then the sets $D_{X}^{\rm gp}$ and $D_{X_{H}}^{\rm gp}$ can be mono-anabelian reconstructed from $\pi_{1}^{\rm t}(U_{X})$ and $H$, respectively. Moreover, the inclusion $H \hookrightarrow\pi_{1}^{\rm t}(U_{X})$ induces a map $\gamma_{H, \pi_{1}^{\rm t}(U_{X})}: D_{X_{H}}^{\rm gp} \rightarrow D_{X}^{\rm gp}$ such that the following commutative diagram holds:
\[
\begin{CD}
D_{X_{H}}^{\rm gp} @>\vartheta_{X_{H}}>> D_{X_{H}}
\\
@V\gamma_{H, \pi_{1}^{\rm t}(U_{X})}VV@VV\gamma_{f_{H}}V
\\
D_{X}^{\rm gp} @>\vartheta_{X}>> D_{X},
\end{CD}
\]
where $\gamma_{f_{H}}$ denotes the map of the sets of marked points induced by $f_{H}$.
\end{proposition}

\begin{proof}
We only need to prove the ``moreover" part of Proposition \ref{pro-3}. We maintain the notation introduced in Remark \ref{rem-2-2}. Note that, for each $e_{X} \in D_{X}$ and each $e_{X_{H}} \in D_{X_{H}}$, the sets $M_{Y, e_{X}}$ and $M_{Z, e_{X_{H}}}$ can be mono-anabelian reconstructed from $\pi_{1}^{\rm t}(U_{X})$ and $H$, respectively. Then the ``moreover" part follows from Remark \ref{rem-2-2}.
\end{proof}

\begin{remark}\label{rem-pro-3-1}
We maintain the notation introduced in Proposition \ref{pro-3}. Let $\pi_{1}(X_{H})$ be the \'etale fundamental group of $X_{H}$. Then we have a natural surjection $H \twoheadrightarrow \pi_{1}(X_{H})$. Note that $\pi_{1}(X_{H})$ admits an action of $\pi_{1}^{\rm t}(U_{X})/H$ induced by the outer action of $\pi_{1}^{\rm t}(U_{X})/H$ on $H$ which is induced by the exact sequence $$1 \rightarrow H \rightarrow \pi_{1}^{\rm t}(U_{X}) \rightarrow \pi_{1}^{\rm t}(U_{X})/H \rightarrow 1.$$ Moreover, the action of $\pi_{1}^{\rm t}(U_{X})/H$ on $\pi_{1}(X_{H})$ induces an action of $\pi_{1}^{\rm t}(U_{X})/H$ on $D^{\rm gp}_{X_{H}}$. On the other hand, it is easy to check that the action of $\pi_{1}^{\rm t}(U_{X})/H$ on $D^{\rm gp}_{X_{H}}$ coincides with the natural action of $\pi_{1}^{\rm t}(U_{X})/H$ on $D_{X_{H}}$ when we identify $D_{X}^{\rm gp}$ with $D_{X}$.
\end{remark}


\subsubsection{}
We have the following result.

\begin{proposition}\label{them-1}
Write ${\rm Ine}(\pi_{1}^{\rm t}(U_{X}))$ for the set of inertia subgroups in $\pi_{1}^{\rm t}(U_{X})$. Then ${\rm Ine}(\pi_{1}^{\rm t}(U_{X}))$ can be mono-anabelian reconstructed from $\pi_{1}^{\rm t}(U_{X})$.
\end{proposition}

\begin{proof}
Let $C_{X}:= \{H_{i}\}_{i \in \mbZ_{> 0}}$ be a set of open normal subgroups of $\pi_{1}^{\rm t}(U_{X})$ such that $\varprojlim_{i}\pi_{1}^{\rm t}(U_{X})/H_{i} \cong \pi_{1}^{\rm t}(U_{X})$ (i.e., a cofinal system of open normal subgroups).

Let $\widetilde e \in D_{\widetilde X}$. For each $i \in \mbZ_{> 0}$, we write $(X_{H_{i}}, D_{X_{H_{i}}})$ for the smooth pointed stable curve of type $(g_{X_{H_{i}}}, n_{X_{H_{i}}})$ induced by $H_{i}$ and $e_{X_{H_{i}}} \in D_{X_{H_{i}}}$ for the image of $\widetilde e$. Then we obtain a sequence of marked points
$$\mcI_{\widetilde e}^{C_{X}}: \dots \mapsto e_{X_{H_{2}}}\mapsto e_{X_{H_{1}}}$$ induced by $C_{X}$. Note that the sequence $\mcI^{C_{X}}_{\widetilde e}$ admits a natural action of $\pi_{1}^{\rm t}(U_{X})$.
We may identify the inertia subgroup $I_{\widetilde e}$ associated to $\widetilde e$ with the stabilizer of $\mcI^{C_{X}}_{\widetilde e}$.

Moreover, since Proposition \ref{proposition 1} (1) implies that $(g_{X_{H_{i}}}, n_{X_{H_{i}}})$ can be mono-anabelian reconstructed from $H_{i}$, by choosing a suitable set of open normal subgroups $C_{X}$, we may assume that  $g_{X_{H_{1}}}\geq 2$. If $n_{X_{H_{1}}}=0$, Proposition \ref{them-1} is trivial. Then we may assume that $n_{X_{H_{1}}}> 0$.

On the other hand, Proposition \ref{pro-3} implies that, for each $H_{i}$, $i\in \mbZ_{> 0}$,  the set $D^{\rm gp}_{X_{H_{i}}}$ can be mono-anabelian reconstructed from $H_{i}$. For each $e_{X_{H_{i}}} \in D_{X_{H_{i}}}$, we denote by $$e^{\rm gp}_{X_{H_{i}}}:= \vartheta_{X_{H_{i}}}^{-1}(e_{X_{H_{i}}}).$$ Then the sequence of marked points $\mcI^{C_{X}}_{\widetilde e}$ induces a sequence $$\mcI^{C_{X}}_{\widetilde e^{\rm gp}}: \dots \mapsto e^{\rm gp}_{X_{H_{2}}}\mapsto e^{\rm gp}_{X_{H_{1}}}.$$ By applying the ``moreover'' part of Proposition \ref{pro-3}, we see that $\mcI^{C_{X}}_{\widetilde e^{\rm gp}}$ can be mono-anabelian reconstructed from $C_{X}$.  Then Remark \ref{rem-pro-3-1} implies that the stabilizer of $\mcI^{C_{X}}_{\widetilde e^{\rm gp}}$ is equal to the stabilizer of $\mcI^{C_{X}}_{\widetilde e}$. This completes the proof of the proposition.
\end{proof}

\subsection{Reconstructions of inertia subgroups  via surjections}\label{sec-4} In this subsection, we will prove that the {\it mono-anabelian reconstructions} obtained in Proposition \ref{them-1} are {\it compatible} with any open continuous homomorphisms (i.e. Theorem \ref{them-2}).

\subsubsection{\bf Settings}\label{sett331}

Let $(X_{i}, D_{X_{i}})$, $i \in \{1, 2\}$, be a smooth pointed stable curve of type $(g_{X}, n_{X})$ over an algebraically closed field $k_{i}$ of characteristic $p>0$, $U_{X_i}:= X_{i} \setminus D_{X_{i}}$, $\pi_{1}^{\rm t}(U_{X_{i}})$ the tame fundamental group of $U_{X_{i}}$, and $\pi_{1}(X_{i})$ the \'etale fundamental group of $X_{i}$. Then Lemma \ref{lem-2} implies that $\pi_{1}(X_{i})$ can be mono-anabelian reconstructed from $\pi_{1}^{\rm t}(U_{X_{i}})$. Moreover, in this subsection, we suppose that  $n_{X}>0$, and that $\phi: \pi_{1}^{\rm t}(U_{X_{1}}) \twoheadrightarrow \pi_{1}^{\rm t}(U_{X_{2}})$ is an arbitrary open continuous surjective homomorphism of profinite groups.

Note that, since $(X_{i}, D_{X_{i}})$, $i\in \{1, 2\}$, is a smooth pointed stable curve of type $(g_{X}, n_{X})$, $\phi$ induces a natural surjection $\phi^{p'}: \pi_{1}^{\rm t}(U_{X_{1}})^{p'} \twoheadrightarrow \pi_{1}^{\rm t}(U_{X_{2}})^{p'}$, where $(-)^{p'}$ denotes the maximal prime-to-$p$ quotient of $(-)$. Since $\pi_{1}^{\rm t}(U_{X_{i}})^{p'}$, $i \in \{1, 2\}$, is topologically finitely generated, and $\pi_{1}^{\rm t}(U_{X_{1}})^{p'}$ is isomorphic to $\pi_{1}^{\rm t}(U_{X_{2}})^{p'}$ as abstract profinite groups, we obtain that $\phi^{p'}: \pi_{1}^{\rm t}(U_{X_{1}})^{p'} \xrightarrow{\simeq} \pi_{1}^{\rm t}(U_{X_{2}})^{p'}$ is an isomorphism (\cite[Proposition 16.10.6]{FJ}).


\subsubsection{}
We explain the main idea in the proof of Theorem \ref{them-2}. Let $H_{2} \subseteq \pi_{1}^{\rm t}(U_{X_{2}})$ be an arbitrary open normal subgroup and $H_{1}:= \phi^{-1}(H_{2}) \subseteq \pi_{1}^{\rm t}(U_{X_{1}})$. We write $(X_{H_i}, D_{X_{H_i}})$, $i\in \{1, 2\}$, for the smooth pointed smooth curve of type $(g_{X_{H_{i}}}, n_{X_{H_i}})$ over $k_{i}$ induced by $H_{i}$. To prove the compatibility, we need to prove that, for any prime number $\ell\neq p$, the weight-monodromy filtration of $H_{2}^{\rm ab}\otimes \mbF_{\ell}$ induces the weight-monodromy filtration of $H_{1}^{\rm ab}\otimes \mbF_{\ell}$ via the natural surjection $\phi|_{H_{1}}: H_{1} \twoheadrightarrow H_{2}$. Note that the weight $1$ part of $H_{i}^{\rm ab}\otimes \mbF_{\ell}$ corresponds to $\pi_{1}(X_{H_{i}})^{\rm ab}\otimes \mbF_{\ell}$, and the weight $2$ part of $H_{i}^{\rm ab}\otimes \mbF_{\ell}$ corresponds to the image of the subgroup of $H_{i}$ generated by the inertia subgroups of the marked points of $D_{X_{H_{i}}}$. The key observation is as follows: 
\begin{quote}
The inequality of the limit of $p$-averages (see Proposition \ref{coro-p-average} (1) below) $$\text{Avr}_{p}(H_{1}) \geq \text{Avr}_{p}(H_{2})$$ of $H_1$ and $H_2$ induced by the surjection $\phi|_{H_{1}}: H_{1} \twoheadrightarrow H_{2}$ plays a role of the comparability of ``Galois actions'' in the theory of the anabelian geometry of curves over {\it algebraically closed fields of characteristic $p>0$}.
\end{quote}


\subsubsection{}\label{paverage}
Firstly, we have the following proposition.

\begin{proposition}\
\label{coro-p-average}\begin{enumerate}
                        \item Let $(X, D_{X})$ be a pointed stable curve of type $(g_{X}, n_{X})$ over an algebraically closed field $k$ of characteristic $p>0$, $U_{X} :=X \setminus D_{X}$, and $\pi^{\rm t}_{1}(U_{X})$ the tame fundamental group of $U_{X}$. Let $r \in \mbN$ be a natural number, and let $K_{p^{r}-1}$ be the kernel of the natural surjection $\pi_{1}^{\rm t}(U_{X}) \twoheadrightarrow \pi_{1}^{\rm t}(U_{X})^{\rm ab} \otimes \mbZ/(p^{r}-1)\mbZ$, where $(-)^{\rm ab}$ denotes the abelianization of $(-)$. Then we have
\begin{eqnarray*}
{\rm Avr}_{p}(\pi_{1}^{\rm t}(U_{X})) := \lim_{r\rightarrow\infty}\frac{{\rm dim}_{\mbF_{p}}(K^{\rm ab}_{p^{r}-1} \otimes \mbF_{p})}{\#( \pi_{1}^{\rm t}(U_{X})^{\rm ab} \otimes \mbZ/(p^{r}-1)\mbZ)}= \left\{ \begin{array}{ll}
g_{X}-1, & \text{if} \ n_{X}\leq 1,
\\
g_{X}, & \text{if} \ n_{X} > 1.
\end{array} \right.
\end{eqnarray*}

                        \item We maintain the setting introduced in \ref{sett331}. Let $H_{2} \subseteq \pi_{1}^{\rm t}(U_{X_{2}})$ be an open normal subgroup such that $([\pi_{1}^{\rm t}(U_{X_{2}}): H_{2}], p)=1$ and $H_{1}:=\phi^{-1}(H_{2})$. Write $g_{H_{i}}$, $i\in \{1, 2\}$, for the genus of the smooth pointed stable curve over $k_{i}$ corresponding to $H_{i} \subseteq \pi_{1}^{\rm t}(U_{X_{i}})$. Then we have  $g_{H_{1}} \geq g_{H_{2}}.$
                      \end{enumerate}

\end{proposition}

\begin{proof}
(1) is the Tamagawa's result concerning the limit of $p$-averages of $\pi_{1}^{\rm t}(U_{X})$ (\cite[Theorem 0.5]{T4}). Let us prove (2). The surjection $\phi$ induces a surjection $\phi^{p'}: \pi^{\rm t}_{1}(U_{X_{1}})^{p'} \twoheadrightarrow \pi_{1}^{\rm t}(U_{X_{2}})^{p'},$ where $(-)^{p'}$ denotes the maximal prime-to-$p$ quotient of $(-)$. Moreover, since $\pi_{1}^{\rm t}(U_{X_{i}})^{p'}$, $i \in \{1, 2\}$, is topologically finitely generated, and $\pi_{1}^{\rm t}(U_{X_{1}})^{p'}$ is isomorphic to $\pi_{1}^{\rm t}(U_{X_{2}})^{p'}$ as abstract profinite groups (since the types of $(X_{1}, D_{X_{1}})$ and $(X_{2}, D_{X_{2}})$ are equal to $(g_{X}, n_{X})$), we obtain that $\phi^{p'}$ is an isomorphism (cf. \cite[Proposition 16.10.6]{FJ}).

On the other hand, since $[\pi_{1}^{\rm t}(U_{X_{1}}): H_{1}]=[\pi_{1}^{\rm t}(U_{X_{2}}): H_{2}]$ and $([\pi_{1}^{\rm t}(U_{X_{2}}): H_{2}], p)=1$, we obtain that the natural homomorphism $\phi_{H}^{p'}: H_{1}^{p'} \twoheadrightarrow H_{2}^{p'}$ induced by $\phi_{H}:= \phi|_{H_{1}}: H_{1} \twoheadrightarrow H_{2}$ is also an isomorphism. This implies $$\#(H_{1}^{\rm ab}\otimes \mbZ/(p^{r}-1)\mbZ)=\#(H_{2}^{\rm ab}\otimes \mbZ/(p^{r}-1)\mbZ)$$ for all $r\in \mbN$. Let $K_{H_{i}, p^{r-1}}$, $i\in \{1, 2\}$, be the kernel of the natural surjection $H_{i} \twoheadrightarrow H_{i}^{\rm ab}\otimes \mbZ/(p^{r}-1)\mbZ.$ Then the surjection $\phi_{H}$ implies  $$\text{Avr}_{p}(H_{1}):= \lim_{r\rightarrow \infty}\frac{\text{dim}_{\mbF_{p}}(K^{\rm ab}_{H_{1}, p^{r}-1}\otimes\mbF_{p})}{\#(H_{1}^{\rm ab}\otimes \mbZ/(p^{r}-1)\mbZ)}\geq \text{Avr}_{p}(H_{2}) := \lim_{r\rightarrow \infty}\frac{\text{dim}_{\mbF_{p}}(K^{\rm ab}_{H_{2}, p^{r}-1}\otimes\mbF_{p})}{\#(H_{2}^{\rm ab}\otimes \mbZ/(p^{r}-1)\mbZ)}.$$ Thus, the corollary follows from (2).
\end{proof}

\subsubsection{} We have the following lemmas.

\begin{lemma}\label{lem-3}
Let $\ell$ be a prime number distinct from $p$. Then the isomorphism $(\phi^{p'})^{-1}: \pi_{1}^{\rm t}(U_{X_{2}})^{p'} \xrightarrow{\simeq} \pi_{1}^{\rm t}(U_{X_{1}})^{p'}$ induces an isomorphism $$\psi_{X}^{\ell}: H^{1}_{\text{\rm \'et}}(X_{1}, \mbF_{\ell})\simeq {\rm Hom}(\pi_{1}(X_{1}), \mbF_{\ell}) \xrightarrow{\simeq} {\rm Hom}(\pi_{1}(X_{2}), \mbF_{\ell}) \simeq H^{1}_{\text{\rm \'et}}(X_{2}, \mbF_{\ell}).$$ 
\end{lemma}

\begin{proof}
Let $f_{1}: (Y_{1}, D_{Y_{1}}) \rightarrow (X_{1}, D_{X_{1}})$ be an \'etale covering of degree $\ell$ over $k_{1}$. Write $f_{2}: (Y_{2}, D_{Y_{2}}) \rightarrow(X_{2}, D_{X_{2}})$ for the connected Galois tame covering of degree $\ell$ over $k_{2}$ induced by $\phi^{p'}$. Then we will prove that $f_{2}$ is also an \'etale covering over $k_{2}$.

Write $g_{Y_{1}}$ and $g_{Y_{2}}$ for the genus of $Y_{1}$ and $Y_{2}$, respectively. Since $f_{1}$ is an \'etale covering of degree $\ell$, the Riemann-Hurwitz formula implies  $g_{Y_{1}}=\ell(g_{X_{1}}-1)+1.$ On the other hand, the Riemann-Hurwitz formula implies  $g_{Y_{2}}=\ell(g_{X_{2}}-1)+1+\frac{1}{2}(\ell -1)\#(\text{Ram}_{f_{2}}).$ By applying Proposition \ref{coro-p-average} (2), the surjection $\phi$ implies  
$g_{Y_{1}} \geq g_{Y_{2}}.$ This means $\#(\text{Ram}_{f_{2}})=0.$ So $f_{2}$ is an \'etale covering over $k_{2}$. Then the morphism $(\phi^{p'})^{-1}$ induces an injection $$\psi_{X}^{\ell}: {\rm Hom}(\pi_{1}(X_{1}), \mbF_{\ell}) \hookrightarrow {\rm Hom}(\pi_{1}(X_{2}), \mbF_{\ell}).$$ Furthermore, since $\text{dim}_{\mbF_{\ell}}({\rm Hom}(\pi_{1}(X_{1}), \mbF_{\ell})) =\text{dim}_{\mbF_{\ell}}( {\rm Hom}(\pi_{1}(X_{2}), \mbF_{\ell}))=2g_{X}$, we obtain that $\psi_{X}^{\ell}$ is a bijection. This completes the proof of the lemma.
\end{proof}

\begin{lemma}\label{lem-4}
Suppose  $g_{X} \geq 2$. Then the surjection $\phi: \pi_{1}^{\rm t}(U_{X_{1}})\twoheadrightarrow\pi_{1}^{\rm t}(U_{X_{2}})$ induces a bijection $$\rho_{\phi}: D^{\rm gp}_{X_{1}}\xrightarrow{\simeq} D^{\rm gp}_{X_{2}},$$ and the bijection $\rho_{\phi}$ can be mono-anabelian reconstructed from $\phi$.
\end{lemma}

\begin{proof}

Let $(\ell, d, y_{2})$ be an mp-triple associated to $\pi_{1}^{\rm t}(U_{X_{2}})$ (see \ref{gpmptriple}). Then Lemma \ref{lem-3} implies that $\phi$ induces an mp-triple $(\ell, d, y_{1})$ associated to $\pi_{1}^{\rm t}(U_{X_{1}})$, where $y_{1}:=(\psi^{d}_{X})^{-1}(y_{2}) \in {\rm Hom}(\pi_{1}(X_{1}), \mu_{d})$.

Let $f_{i}: (Y_{i}, D_{Y_i}) \rightarrow(X_{i}, D_{X_{i}})$, $i \in \{1, 2\}$, be the \'etale covering of degree $d$ over $k_{i}$ induced by $y_{i}$. Then the mp-triple $(\ell, d, y_{i})$ associated to $\pi_{1}^{\rm t}(U_{X_{i}})$ determines an mp-triple $$(\ell, d, f_{i}: (Y_{i}, D_{Y_{i}}) \rightarrow (X_{i}, D_{X_{i}}))$$ associated to $(X_{i}, D_{X_{i}})$ over $k_{i}$. Note that the types of $(Y_{1}, D_{Y_{1}})$ and $(Y_{2}, D_{Y_{2}})$ are equal.

Write $\pi_{1}^{\rm t}(U_{Y_{i}})$, $i \in \{1, 2\}$, for the kernel of $\pi_{1}^{\rm t}(U_{X_{i}}) \twoheadrightarrow \pi_{1}(X_{i}) \overset{y_{i}}\twoheadrightarrow \mu_{d}$. 
By replacing $(X_{i}, D_{X_{i}})$ by $(Y_{i}, D_{Y_{i}})$, Lemma \ref{lem-3} implies that $(\phi|_{\pi_{1}^{\rm t}(U_{Y_{1}})}^{p'})^{-1}$ induces a commutative diagram as follows:
\[
\begin{CD}
0@>>>H^{1}_{\text{\'et}}(Y_{1}, \mbF_{\ell}) @>>> H^{1}_{\text{\'et}}(U_{Y_{1}}, \mbF_{\ell})@>>>\text{Div}^{0}_{D_{Y_{1}}}(Y_{1})\otimes \mbF_{\ell}@>>>0
\\
@.@V\psi_{Y}^{\ell}VV@V\psi_{Y}^{\rm t, \ell}VV@VVV@.
\\
0@>>>H^{1}_{\text{\'et}}(Y_{2}, \mbF_{\ell}) @>>> H^{1}_{\text{\'et}}(U_{Y_{2}}, \mbF_{\ell})@>>>\text{Div}^{0}_{D_{Y_{2}}}(Y_{2})\otimes \mbF_{\ell}@>>>0,
\end{CD}
\]
where all the vertical arrows are isomorphisms. We note that $H^{1}_{\text{\'et}}(Y_{i}, \mbF_{\ell})$, $H^{1}_{\text{\'et}}(U_{Y_{i}}, \mbF_{\ell})$, and $\text{Div}^{0}_{D_{Y_{i}}}(Y_{i})\otimes \mbF_{\ell}$, $i, \in \{1, 2\}$, are naturally isomorphic to $\text{Hom}(\pi_{1}(Y_{i}), \mbF_{\ell})$, $\text{Hom}(\pi_{1}^{\rm t}(U_{Y_{i}}), \mbF_{\ell})$, and  $\text{Hom}(\pi_{1}^{\rm t}(U_{Y_{i}}), \mbF_{\ell})/\text{Hom}(\pi_{1}(Y_{i}), \mbF_{\ell}),$  respectively. Then Lemma \ref{lem-2} implies that the commutative diagram above can be mono-anabelian reconstructed from $\phi|_{\pi_{1}^{\rm t}(U_{Y_{1}})}: \pi_{1}^{\rm t}(U_{Y_{1}}) \twoheadrightarrow \pi_{1}^{\rm t}(U_{Y_{2}})$.

Write $M_{Y_{i}}\subseteq M_{Y_{i}}^{*}$ for the subsets of $H^{1}_{\text{\'et}}(U_{Y_{i}}, \mbF_{\ell})$ defined in \ref{sec31aaa}. Since the actions of $\mu_{d}$ on the exact sequences
are compatible with the isomorphisms appearing in the commutative diagram above, we have $\psi_{Y}^{\rm t, \ell}(M_{Y_{1}}^{*})=M_{Y_{2}}^{*}.$ Next, we prove $\psi_{Y}^{\rm t, \ell}(M_{Y_{1}})=M_{Y_{2}}.$

 Let $\alpha_{1} \in M_{Y_{1}}$ and $g_{\alpha_{1}}: (Y_{\alpha_{1}}, D_{Y_{\alpha_{1}}}) \rightarrow (Y_{1}, D_{Y_{1}})$ the Galois tame covering of degree $\ell$ over $k_{1}$ induced by $\alpha_{1}$. Write $g_{\alpha_{2}}: (Y_{\alpha_{2}}, D_{Y_{\alpha_{2}}}) \rightarrow (Y_{2}, D_{Y_{2}})$ for the Galois tame covering of degree $\ell$ over $k_{2}$ induced by $\alpha_{2}:=\psi^{\rm t, \ell}_{Y}(\alpha_{1})$. Write $g_{Y_{\alpha_{1}}}$ and $g_{Y_{\alpha_{2}}}$ for the genus of $Y_{\alpha_{1}}$ and $Y_{\alpha_{2}}$, respectively. Then Proposition \ref{coro-p-average} (2) and the Riemann-Hurwitz formula imply that $g_{Y_{\alpha_{1}}}-g_{Y_{\alpha_{2}}}=\frac{1}{2}(d-\#(\text{Ram}_{g_{\alpha_{2}}}))(\ell-1)\geq 0.$ This means  $d-\#(\text{Ram}_{g_{\alpha_{2}}}) \geq 0.$ Since $\alpha_{2} \in M_{Y_{2}}^{*}$, we have $d \ | \ \#(\text{Ram}_{g_{\alpha_{2}}})$. Thus, either $\#(\text{Ram}_{g_{\alpha_{2}}})=0$ or $\#(\text{Ram}_{g_{\alpha_{2}}})=d$ holds.

If $\#(\text{Ram}_{g_{\alpha_{2}}})=0$, then $g_{\alpha_{2}}$ is an \'etale covering over $k_{2}$. Then Lemma \ref{lem-3} implies that $g_{\alpha_{1}}$ is an \'etale covering over $k_{1}$. This provides a contradiction to the fact that $\alpha_{1} \in M_{Y_{1}}$. Then we have $\#(\text{Ram}_{g_{\alpha_{2}}})=d$. This means $\alpha_{2} \in M_{Y_{2}}$. Thus, we obtain $\psi_{Y}^{\rm t, \ell}(M_{Y_{1}})\subseteq M_{Y_{2}}.$ On the other hand, Lemma \ref{lem-1} implies $\#(M_{Y_{1}})=\#(M_{Y_{2}}).$ We have $\psi_{Y}^{\rm t, \ell}: M_{Y_{1}} \xrightarrow{\simeq} M_{Y_{2}}.$ Then Proposition \ref{pro-2} implies that $\psi^{\rm t, \ell}_{Y}$ induces a bijection $$\rho_{\phi}: D^{\rm gp}_{X_{1}} \xrightarrow{\simeq} D^{\rm gp}_{X_{2}}.$$ Moreover, since $M_{Y_{i}}$ and $M_{Y_{i}}^{*}$ can be mono-anabelian reconstructed from $\pi_{1}^{\rm t}(U_{Y_{i}})$, the bijection $\rho_{\phi}$ can be mono-anabelian reconstructed from $\phi$. This completes the proof of the lemma.
\end{proof}



\subsubsection{}\label{sec334}
Let $H_{2} \subseteq \pi_{1}^{\rm t}(U_{X_{2}})$ be an arbitrary open normal subgroup and $H_{1}:=\phi^{-1}(H_{2})$. We write $(X_{H_i}, D_{X_{H_i}})$, $i\in \{1, 2\}$, for the smooth pointed stable curve of type $(g_{X_{H_{i}}}, n_{X_{H_i}})$ over $k_{i}$ induced by $H_{i}$ and $f_{H_{i}}: (X_{H_{i}}, D_{X_{H_{i}}}) \rightarrow (X_{i}, D_{X_{i}})$ for the Galois tame coverings over $k_{i}$ induced by the inclusion $H_{i} \hookrightarrow \pi_{1}^{\rm t}(U_{X_{i}})$. Moreover, Proposition \ref{pro-3} implies that the inclusion $H_{i} \hookrightarrow\pi_{1}^{\rm t}(U_{X_{i}})$ induces a map $\gamma_{H_{i}, \pi_{1}^{\rm t}(U_{X_{i}})}: D_{X_{H_i}}^{\rm gp} \rightarrow D_{X_{i}}^{\rm gp}$ which fits into the following commutative diagram:
\[
\begin{CD}
D_{X_{H_i}}^{\rm gp} @>\vartheta_{X_{H_i}}>> D_{X_{H_i}}
\\
@V\gamma_{H_i, \pi_{1}^{\rm t}(U_{X_{i}})}VV@VV\gamma_{f_{H_i}}V
\\
D_{X_{i}}^{\rm gp} @>\vartheta_{X_{i}}>> D_{X_{i}},
\end{CD}
\]
where $\gamma_{f_{H_{i}}}$ denotes the map of the sets of marked points induced by $f_{H_{i}}$. We may identify $\pi_{1}^{\rm t}(U_{X_{1}})/H_{1}$ with $\pi_{1}^{\rm t}(U_{X_{2}})/H_{2}$ via the isomorphism $\pi_{1}^{\rm t}(U_{X_{1}})/H_{1} \xrightarrow{\simeq} \pi_{1}^{\rm t}(U_{X_{2}})/H_{2}$ induced by $\phi$, and denote by $G:=\pi_{1}^{\rm t}(U_{X_{1}})/H_{1} \cong \pi_{1}^{\rm t}(U_{X_{2}})/H_{2}$. Then we have the following lemma.

\begin{lemma}\label{lem-5}
Suppose that $g_{X} \geq 2$, and that $(g_{X_{H_{1}}}, n_{X_{H_{1}}})=(g_{X_{H_{2}}}, n_{X_{H_{2}}})$. Then the commutative diagram of profinite groups
\begin{equation}
\begin{CD}
H_{1} @>\phi|_{H_{1}}>>H_{2}
\\
@VVV@VVV
\\
\pi_{1}^{\rm t}(U_{X_{1}}) @>\phi>>\pi_{1}^{\rm t}(U_{X_{2}})
\end{CD}
\end{equation}
induces a commutative diagram
\begin{equation}
\begin{CD}
D^{\rm gp}_{X_{H_{1}}} @>\rho_{\phi|_{H_{1}}}>>D^{\rm gp}_{X_{H_{2}}}
\\
@V\gamma_{H_{1}, \pi_{1}^{\rm t}(U_{X_{1}})}VV@VV\gamma_{H_{2}, \pi_{1}^{\rm t}(U_{X_{2}})}V
\\
D^{\rm gp}_{X_{1}} @>\rho_{\phi}>>D^{\rm gp}_{X_{2}}.
\end{CD}
\end{equation}
Moreover, the commutative diagram $(2)$ can be mono-anabelian reconstructed from $(1)$.
\end{lemma}

\begin{proof}
Proposition \ref{pro-3} and Lemma \ref{lem-4} imply the diagram
$$
\begin{CD}
D^{\rm gp}_{X_{H_{1}}} @>\rho_{\phi|_{H_{1}}}>>D^{\rm gp}_{X_{H_{2}}}
\\
@V\gamma_{H_{1}, \pi_{1}^{\rm t}(U_{X_{1}})}VV@VV\gamma_{H_{2}, \pi_{1}^{\rm t}(U_{X_{2}})}V
\\
D^{\rm gp}_{X_{1}} @>\rho_{\phi}>>D^{\rm gp}_{X_{2}}
\end{CD}
$$
can be mono-anabelian reconstructed from the commutative diagram of profinite groups
$$
\begin{CD}
H_{1} @>\phi|_{H_{1}}>>H_{2}
\\
@VVV@VVV
\\
\pi_{1}^{\rm t}(U_{X_{1}}) @>\phi>>\pi_{1}^{\rm t}(U_{X_{2}}).
\end{CD}
$$
To verify Lemma \ref{lem-5}, it is sufficient to check that the diagram is commutative.

Let $e^{\rm gp}_{X_{H_{1}}} \in D^{\rm gp}_{X_{H_{1}}}$, $e^{\rm gp}_{X_{H_{2}}}:= \rho_{\phi|_{H_{1}}}(e^{\rm gp}_{X_{H_{1}}}) \in D^{\rm gp}_{X_{H_{2}}}$, $e^{\rm gp}_{X_1}:= \gamma_{H_{1}, \pi_{1}^{\rm t}(U_{X_{1}})}(e^{\rm gp}_{X_{H_{1}}}) \in D^{\rm gp}_{X_{1}}$, $e^{\rm gp}_{X_2}:=(\gamma_{H_{2}, \pi_{1}^{\rm t}(U_{X_{2}})}\circ \rho_{\phi|_{H_{1}}})(e^{\rm gp}_{X_{H_{1}}}) \in D^{\rm gp}_{X_{2}}$, and $e^{\rm gp, *}_{X_1}:= \rho^{-1}_{\phi}(e^{\rm gp}_{X_2}) \in D^{\rm gp}_{X_{1}}$. Let us prove $$e^{\rm gp}_{X_1}=e^{\rm gp, *}_{X_1}.$$ We put $S^{\rm gp}_{X_{H_{1}}} :=\gamma_{H_{1}, \pi_{1}^{\rm t}(U_{X_{1}})}^{-1}(e_{X_{1}}^{\rm gp, *})$ and $S^{\rm gp}_{X_{H_{2}}}:= \gamma_{H_{2}, \pi_{1}^{\rm t}(U_{X_{2}})}^{-1}(e^{\rm gp}_{X_2})$, respectively. Note that $e^{\rm gp}_{X_{H_{2}}} \in S^{\rm gp}_{X_{H_{2}}}$. To verify $e^{\rm gp}_{X_1}=e^{\rm gp, *}_{X_1}$, it is sufficient to prove that $e^{\rm gp}_{X_{H_{1}}} \in S^{\rm gp}_{X_{H_{1}}}$. Moreover, for each $i\in \{1, 2\}$, we put $$e_{X_{i}} := \vartheta_{X_{i}}(e^{\rm gp}_{X_{i}}), \ e_{X_{H_i}} := \vartheta_{X_{H_i}}(e^{\rm gp}_{X_{i}}), \ e^{*}_{X_{1}} := \vartheta_{X_{1}}(e_{X_{1}}^{\rm gp, *}), \ S_{X_{i}} := S^{\rm gp}_{X_{i}}, \ S_{X_{H_{i}}} := S_{X_{H_i}}^{\rm gp}.$$ Then to verify the lemma, we only need to prove that $e_{X_{H_{1}}} \in \vartheta_{X_{H_1}}(S_{X_{H_{1}}})$.


Let $(\ell, d, y_{2})$ be an mp-triple associated to $\pi_{1}^{\rm t}(U_{X_{2}})$. Then Lemma \ref{lem-3} implies that $\phi$ induces an mp-triple $(\ell, d, y_{1})$ associated to $\pi_{1}^{\rm t}(U_{X_{1}})$, where $y_{1}:= (\psi^{d}_{X})^{-1}(y_{2}) \in {\rm Hom}(\pi_{1}(X_{1}), \mu_{d})$.  Let $f_{i}: (Y_{i}, D_{Y_i}) \rightarrow (X_{i}, D_{X_{i}})$, $i \in \{1, 2\}$, be the tame covering of degree $d$ over $k_{i}$ induced by $y_{i}$. Then the mp-triple $(\ell, d, y_{i})$ associated to $\pi_{1}^{\rm t}(U_{X_{i}})$ induces an mp-triple $$(\ell, d, f_{i}: (Y_{i}, D_{Y_{i}}) \rightarrow(X_{i}, D_{X_{i}}))$$ associated to $(X_{i}, D_{X_{i}})$ over $k_{i}$. Note that since $f_{1}$ and $f_{2}$ are \'etale, the types of $(Y_{1}, D_{Y_{1}})$ and $(Y_{2}, D_{Y_{2}})$ are equal. On the other hand, we have an mp-triple $$(\ell, d, g_{2}: (Z_{2}, D_{Z_{2}}):=(Y_{2}, D_{Y_{2}}) \times_{(X_{2}, D_{X_{2}})} (X_{H_2}, D_{X_{H_2}}) \rightarrow (X_{H_2}, D_{X_{H_2}}))$$ associated to $(X_{H_2}, D_{X_{H_2}})$ induced by the natural inclusion $H_{2} \hookrightarrow\pi_{1}^{\rm t}(U_{X_{2}})$ and the mp-triple $(\ell, d, f_{2}: (Y_{2}, D_{Y_{2}}) \rightarrow(X_{2}, D_{X_{2}}))$. By Lemma \ref{lem-3} again, we obtain an mp-triple $$(\ell, d, g_{1}: (Z_{1}, D_{Z_{1}}):= (Y_{1}, D_{Y_{1}}) \times_{(X_{1}, D_{X_{1}})} (X_{H_1}, D_{X_{H_1}}) \rightarrow (X_{H_1}, D_{X_{H_1}}))$$ associated to $(X_{H_1}, D_{X_{H_1}})$ induced by $\phi|_{H_{1}}$ and the triple $(\ell, d, g_{2}: (Z_{2}, D_{Z_{2}})\rightarrow (X_{H_2}, D_{X_{H_2}}))$.

Let $\alpha_{2} \in M_{Y_{2}, e_{X_2}}$. The final paragraph of the proof of Lemma \ref{lem-4} implies that we have a bijection $M_{Y_{1}}=\bigsqcup_{e \in D_{X_{1}}} M_{Y_{1}, e} \xrightarrow{\simeq}M_{Y_{2}}=\bigsqcup_{e \in D_{X_{2}}} M_{Y_{2}, e}$ induced by $\phi$. Then $\alpha_{2}$ induces an element $\alpha_{1} \in M_{Y_{1}, e^{*}_{X_1}}.$ Write $(Y_{\alpha_{1}}, D_{Y_{\alpha_{1}}})$ and $(Y_{\alpha_{2}}, D_{Y_{\alpha_{2}}})$ for the smooth pointed stable curves over $k_{1}$ and $k_{2}$ induced by $\alpha_{1}$ and $\alpha_{2}$, respectively. Consider the connected Galois tame covering $$(Y_{\alpha_{2}}, D_{Y_{\alpha_{2}}}) \times_{(X_{2}, D_{X_{2}})} (X_{H_2}, D_{X_{H_2}}) \rightarrow (Z_{2}, D_{Z_{2}})$$ of degree $\ell$ over $k_{2}$, and write $\beta_{2}$ for an element of $M_{Z_{2}}^{*}$ corresponding to this connected Galois tame covering. Then we have $$\beta_{2}=\sum_{c_{2}\in S_{X_{H_{2}}}} t_{c_{2}}\beta_{c_{2}},$$ where $t_{c_{2}} \in (\mbZ/\ell\mbZ)^{\times}$ and $\beta_{c_{2}} \in M_{Z_{2}, c_{2}}$. On the other hand, the proof of Lemma \ref{lem-4} implies that $\beta_{2}$ induces an element

$$\beta_{1}:= \sum_{c_{2}\in S_{X_{H_{2}}}\setminus \{e_{X_{H_{2}}}\}} t_{c_{2}}\beta_{\rho^{-1}_{\phi|_{H_{1}}}(c_{2})} +t_{e_{X_{H_{2}}}}\beta_{\rho^{-1}_{\phi|_{H_{1}}}(e_{X_{H_2}})}$$$$=\sum_{c_{2}\in S_{X_{H_{2}}}\setminus \{e_{X_{H_{2}}}\}} t_{c_{2}}\beta_{\rho^{-1}_{\phi|_{H_{1}}}(c_{2})} +t_{e_{X_{H_{2}}}}\beta_{e_{X_{H_1}}}\in M_{Z_{1}}^{*}.$$ Then we have that the coefficient $t_{e_{X_{H_{2}}}}$ of $\beta_{e_{X_{H_1}}}$ is not equal to $0$. Thus, the composition $$(Y_{\alpha_{1}}, D_{Y_{\alpha_{1}}}) \times_{(X_{1}, D_{X_{1}})} (X_{H_1}, D_{X_{H_1}}) \rightarrow(Z_{1}, D_{Z_{1}}) \overset{g_{1}}\rightarrow (X_{H_1}, D_{X_{H_{1}}})$$ is tamely ramified over $e_{X_{H_1}}$. This means that $e_{X_{H_{1}}}$ is contained in $S_{X_{H_1}}$. This completes the proof of the lemma.
\end{proof}

\begin{remark}\label{rem-lem-5-1}
Remark \ref{rem-pro-3-1} implies that $D_{X_{H_i}}^{\rm gp}$, $i\in \{1, 2\}$, admits a natural action of $G$. Moreover, the commutative diagram
\[
\begin{CD}
D^{\rm gp}_{X_{H_{1}}} @>\rho_{\phi|_{H_{1}}}>>D^{\rm gp}_{X_{H_{2}}}
\\
@V\gamma_{H_{1}, \pi_{1}^{\rm t}(U_{X_{1}})}VV@VV\gamma_{H_{2}, \pi_{1}^{\rm t}(U_{X_{2}})}V
\\
D^{\rm gp}_{X_{1}} @>\rho_{\phi}>>D^{\rm gp}_{X_{2}}
\end{CD}
\]
is compatible with the actions of $G$.
\end{remark}


\subsubsection{} Next, we prove that the condition $(g_{X_{H_{1}}}, n_{X_{H_{1}}})=(g_{X_{H_{2}}}, n_{X_{H_{2}}})$ mentioned in Lemma \ref{lem-5} can be omitted. Firstly, we treat the case of abelian groups.

\begin{lemma}\label{lem-6}
We maintain the notation introduced in \ref{sec334}. Suppose that $g_{X}\geq2$, and that $G$ is an abelian group. Then we have $(g_{X_{H_{1}}}, n_{X_{H_{1}}})=(g_{X_{H_{2}}}, n_{X_{H_{2}}}).$
\end{lemma}

\begin{proof}
We write $m$ for $\#G$ and put $K_{2}:= \text{ker}(\pi_{1}^{\rm t}(U_{X_{2}}) \twoheadrightarrow \pi_{1}^{\rm t}(U_{X_{2}})^{\rm ab}\otimes \mbZ/m\mbZ).$ Then we see immediately that $K_{2}$ is contained in $H_{2}$. Let $K_{1}:=\phi^{-1}(K_{2}) \subseteq H_{1}$. Write $(X_{K_{i}}, D_{X_{K_{i}}})$ for the smooth pointed stable curves of type $(g_{X_{K_{i}}}, n_{X_{K_{i}}})$ over $k_{i}$ induced by $K_{i}$ and $f_{K_{i}}: (X_{K_{i}}, D_{X_{K_{i}}}) \rightarrow (X_{i}, D_{X_{i}})$ for the tame covering over $k_{i}$ induced by the inclusion $K_{i}\hookrightarrow \pi_{1}^{\rm t}(U_{X_{i}})$. We identify $\pi_{1}^{\rm t}(U_{X_{1}})/K_{1}$ with $\pi_{1}^{\rm t}(U_{X_{2}})/K_{2}$ via the isomorphism induced by $\phi$, and denote by $A:= \pi_{1}^{\rm t}(U_{X_{1}})/K_{1} \simeq \pi_{1}^{\rm t}(U_{X_{2}})/K_{2}$.

Since each $p$-Galois tame covering is \'etale (i.e., Galois tame coverings whose Galois group is a $p$-group), we see immediately that $(g_{X_{K_{1}}}, n_{X_{K_{1}}})=(g_{X_{K_{2}}}, n_{X_{K_{2}}}).$ Then Lemma \ref{lem-5} implies that the commutative diagram
\[
\begin{CD}
K_{1} @>\phi|_{K_{1}}>>K_{2}
\\
@VVV@VVV
\\
\pi_{1}^{\rm t}(U_{X_{1}}) @>\phi>>\pi_{1}^{\rm t}(U_{X_{2}})
\end{CD}
\]
of profinite groups induces a commutative diagram
\[
\begin{CD}
D^{\rm gp}_{X_{K_{1}}} @>\rho_{\phi|_{K_{1}}}>>D^{\rm gp}_{X_{K_{2}}}
\\
@V\gamma_{K_{1}, \pi_{1}^{\rm t}(U_{X_{1}})}VV@VV\gamma_{K_{2}, \pi_{1}^{\rm t}(U_{X_{2}})}V
\\
D^{\rm gp}_{X_{1}} @>\rho_{\phi}>>D^{\rm gp}_{X_{2}}.
\end{CD}
\]
Moreover, Remark \ref{rem-lem-5-1} implies that the commutative diagram above admits a natural action of $A$. Then, for each $e^{\rm gp}_{X_{K_{1}}} \in D^{\rm gp}_{X_{K_{1}}}$, the inertia subgroup $I_{e^{\rm gp}_{X_{K_{1}}}}$ in $A$ associated to $e^{\rm gp}_{X_{K_{1}}}$ (i.e., the stabilizer of $e^{\rm gp}_{X_{K_{1}}}$ under the action of $A$) is equal to the inertia subgroup $I_{e^{\rm gp}_{X_{K_{2}}}}$ in $A$ associated to $e^{\rm gp}_{X_{K_{2}}}:= \rho_{\phi|_{K_{1}}}(e^{\rm gp}_{X_{K_{1}}}) \in D_{X_{K_{2}}}^{\rm gp}$. On the other hand, write $F$ for the kernel of the natural morphism $A \twoheadrightarrow G$ induced by the inclusion $K_{i} \hookrightarrow H_{i}$, $i \in \{1, 2\}.$ Since $(X_{H_{i}}, D_{X_{H_{i}}}) \simeq (X_{K_{i}}, D_{X_{K_{i}}})/F$, the set of ramification indices of the Galois tame covering $(X_{K_{i}}, D_{X_{K_{i}}})\rightarrow (X_{H_{i}}, D_{X_{H_{i}}})$ with Galois group $F$ are equal to $\{\#(F \cap I_{e^{\rm gp}_{X_{K_{i}}}})\}_{e^{\rm gp}_{X_{K_{i}}} \in D^{\rm gp}_{X_{K_{i}}}}$. Then by the Riemann-Hurwitz formula, we have $(g_{X_{H_{1}}}, n_{X_{H_{1}}})=(g_{X_{H_{2}}}, n_{X_{H_{2}}}).$ This completes the proof of the lemma.
\end{proof}

Next, we treat the general case.

\begin{lemma}\label{lem-7}
We maintain the notation introduced in \ref{sec334}. Suppose that $g_{X}\geq2$ and $n_{X}\geq2$. Then there exists an open normal subgroup $P_{2} \subseteq \pi_{1}^{\rm t}(U_{X_{2}})$ which is contained in $H_{2}$ such that the following holds:
\begin{quote}
Write $(X_{P_{i}}, D_{X_{P_{i}}})$, $i \in \{1, 2\}$, for the smooth pointed stable curve of type $(g_{X_{P_{i}}}, n_{X_{P_{i}}})$ over $k_{i}$ induced by $P_{i}$, where $P_{1}=\phi^{-1}(P_{2})$.  We have $(g_{X_{P_{1}}}, n_{X_{P_{1}}})=(g_{X_{P_{2}}}, n_{X_{P_{2}}}).$
\end{quote}
\end{lemma}

\begin{proof}
First, suppose that $G$ is a {\it simple} finite group. By applying Lemma \ref{lem-6}, we may assume that $G$ is {\it non-abelian}. Moreover, we claim that we may assume that $n_{X}$ is a positive {\it even} number. Let us prove this claim. Suppose $p\neq 2$. Let $R_{2} \subseteq \pi_{1}^{\rm t}(U_{X_{2}})$ be an open subgroup such that $\#(\pi_{1}^{\rm t}(U_{X_{2}})/R_{2})=2$, and that $R_{2} \supseteq \text{ker}(\pi_{1}^{\rm t}(U_{X_{2}}) \twoheadrightarrow \pi_{1}(X_{2}))$ (i.e., the cyclic Galois tame covering corresponding to $R_{2}$ is \'etale). Let $R_{1} := \phi^{-1}(R_{2}) \subseteq \pi_{1}^{\rm t}(U_{X_{1}})$. Then we have that $\#(\pi_{1}^{\rm t}(U_{X_{1}})/R_{1})=2$, and that Lemma \ref{lem-3} implies $R_{1} \supseteq \text{ker}(\pi_{1}^{\rm t}(U_{X_{1}}) \twoheadrightarrow \pi_{1}(X_{1}))$. By replacing $H_{i}$ and $\pi_{1}^{\rm t}(U_{X_{i}})$, $i\in\{1, 2\}$, by $H_{i} \cap R_{i}$ and $R_{i}$, respectively, we may assume that $n_{X}$ is a even positive number. Suppose that $p=2$. Let $\ell$ be a prime number such that $(\ell, 2)=(\ell, \#G)=1$. By \cite[Th\'eor\`eme 4.3.1]{R1}, there exists an open subgroup $R^{*}_{2} \subseteq \pi_{1}^{\rm t}(U_{X_{2}})$ such that $\#(\pi_{1}^{\rm t}(U_{X_{2}})/R^{*}_{2})=\ell$, that $R^{*}_{2} \supseteq \text{ker}(\pi_{1}^{\rm t}(U_{X_{2}}) \twoheadrightarrow \pi_{1}(X_{2}))$, and that $$\text{dim}_{\mbF_{p}}(R_{2}^{*, \rm ab} \otimes \mbF_{p}) >0.$$ Let $R^{*}_{1} := \phi^{-1}(R^{*}_{2}) \subseteq \pi_{1}^{\rm t}(U_{X_{1}})$. Then we have that $\#(\pi_{1}^{\rm t}(U_{X_{1}})/R^{*}_{1})=\ell$, that $\text{dim}_{\mbF_{p}}(R_{1}^{*, \rm ab} \otimes \mbF_{p}) >0$, and that Lemma \ref{lem-3} implies $R^{*}_{1} \supseteq \text{ker}(\pi_{1}^{\rm t}(U_{X_{1}}) \twoheadrightarrow \pi_{1}(X_{1}))$. Thus, we may take an open subgroup $R'_{2} \subseteq R^{*}_{2}$ such that $$\pi_{1}^{\rm t}(U_{X_{2}})/R'_{2} \simeq \mbZ/2\mbZ \rtimes \mbZ/\ell\mbZ,$$ and that $R'_{2} \supseteq \text{ker}(\pi_{1}^{\rm t}(U_{X_{2}}) \twoheadrightarrow\pi_{1}(X_{2}))$. We put $R_{1}' := \phi^{-1}(R_{2}')$. Then the construction of $R_{1}'$ implies $\pi_{1}^{\rm t}(U_{X_{1}})/R'_{1} \simeq \mbZ/2\mbZ \rtimes \mbZ/\ell\mbZ$ and $R'_{1} \supseteq \text{ker}(\pi_{1}^{\rm t}(U_{X_{1}}) \twoheadrightarrow\pi_{1}(X_{1}))$. By replacing $H_{i}$ and $\pi_{1}^{\rm t}(U_{X_{i}})$, $i\in\{1, 2\}$, by $H_{i} \cap R'_{i}$ and $R'_{i}$, respectively, we may assume that $n_{X}$ is a even positive number. This completes the proof of the claim.

Let $\#G:= p^{t}m'$ such that $(m', p)=1$. Since $n_{X}$ is a positive {\it even} number, we may choose a Galois tame covering $$f_{2}: (Y_{2}, D_{Y_{2}}) \rightarrow (X_{2}, D_{X_{2}})$$ over $k_{2}$ with Galois group $\mbZ/m'\mbZ$ such that $f_{2}$ is totally ramified over every marked point of $D_{X_{2}}$. Write $(g_{Y_{2}}, n_{Y_{2}})$ for the type of $(Y_{2}, D_{Y_{2}})$, $Q_{2} \subseteq \pi_{1}^{\rm t}(U_{X_{2}})$ for the open normal subgroup induced by $f_{2}$, $Q_{1}:= \phi^{-1}(Q_{2}) \subseteq \pi_{1}^{\rm t}(U_{X_{1}})$, $$f_{1}: (Y_{1}, D_{Y_{1}})\rightarrow (X_{1}, D_{X_{1}})$$ for the Galois tame covering over $k_{1}$ with Galois group $\mbZ/m'\mbZ$ induced by the natural inclusion $Q_{1} \hookrightarrow \pi_{1}^{\rm t}(U_{X_{1}})$, and $(g_{Y_{1}}, n_{Y_{1}})$ for the type of $(Y_{1}, D_{Y_{1}})$. Then Lemma \ref{lem-6} implies that $(g_{Y_{1}}, n_{Y_{1}})=(g_{Y_{2}}, n_{Y_{2}})$ and $f_{1}$ is also totally ramified over every marked point of $D_{X_{1}}$.

We consider the Galois tame covering $$(Z_{i}, D_{Z_{i}}) := (X_{H_i}, D_{X_{H_{i}}}) \times_{(X_{i}, D_{X_{i}})} (Y_{i}, D_{Y_{i}}) \rightarrow (X_{i}, D_{X_{i}}), \ i \in \{1, 2\},$$ over $k_{i}$ with Galois group $G \times \mbZ/m'\mbZ$ which is the composition of $(Z_{i}, D_{Z_{i}}) \rightarrow (Y_{i}, D_{Y_{i}})$ and $(Y_{i}, D_{Y_{i}}) \rightarrow (X_{i}, D_{X_{i}})$. Note that since $G$ is a non-abelian simple finite group, $(Z_{i}, D_{Z_{i}})$ is connected. Moreover, by Abhyankar's lemma, we obtain that $(Z_{i}, D_{Z_{i}}) \rightarrow (Y_{i}, D_{Y_{i}})$ is an \'etale covering over $k_{i}$. Since $(g_{Y_{1}}, n_{Y_{1}})=(g_{Y_{2}}, n_{Y_{2}})$ and $(Z_{i}, D_{Z_{i}}) \rightarrow (Y_{i}, D_{Y_{i}})$ is unramified, the Riemann-Hurwitz formula implies  $(g_{Z_{1}}, n_{Z_{1}})=(g_{Z_{2}}, n_{Z_{2}}).$

Next, let us prove the lemma in the case where $G$ is an arbitrary finite group. Let $G_{1} \subseteq G_{2} \subseteq \dots \subseteq G_{n}:= G$ be a sequence of subgroups of $G$ such that $G_{i}/G_{i-1}$ is a simple group for all $i\in\{2, \dots n\}$. In order to verify the lemma, we see that it is sufficient to prove the lemma when $n=2$. Let $N_{2}$ be the kernel of the natural homomorphism $\pi_{1}^{\rm t}(U_{X_{2}}) \twoheadrightarrow G \twoheadrightarrow G_{1}$ and $N_{1} := \phi^{-1}(N_{2})$. Then by replacing $G$ by $G_{1}$ and by applying the lemma for the simple group $G_{1}$, we obtain an open normal subgroup $M_{2} \subseteq \pi_{1}^{\rm t}(U_{X_{2}})$ which is contained in $N_{2}$ such that $(g_{X_{M_{1}}}, n_{X_{M_{1}}})=(g_{X_{M_{2}}}, n_{X_{M_{2}}}),$ where $M_{1} := \phi^{-1}(M_{2})$, and $(g_{X_{M_{i}}}, n_{X_{M_{i}}})$, $i\in \{1, 2\}$, denotes the type of the smooth pointed stable curve corresponding to $M_{i}$.

If $M_{i} \subseteq H_{i}$, $i\in \{1, 2\}$, then we may put $P_{i} :=M_{i}.$ If $H_{i}$, $i\in \{1, 2\}$, does not contain $M_{i}$,  we put $O_{i} := M_{i} \cap H_{i}$. Then we have $M_{i}/O_{i}\simeq G/G_{1}.$ Note that $G/G_{1}$ is a simple group. Then the lemma follows from the lemma when we replace $(X_{i}, D_{X_{i}})$ and $G$ by $(X_{M_{i}}, D_{X_{M_{i}}})$ and the simple group $G/G_{1}$, respectively. This completes the proof of the lemma.
\end{proof}

\subsubsection{}
Now, we prove the main result of the present section.

\begin{theorem}\label{them-2}
Let $(\widetilde X_{i}, D_{\widetilde X_{i}}),$ $i \in \{1, 2\},$ be the universal tame covering of $(X_{i}, D_{X_{i}})$ defined in \ref{unicov313}. Let $\phi: \pi_{1}^{\rm t}(U_{X_{1}}) \twoheadrightarrow\pi_{1}^{\rm t}(U_{X_{2}})$ be an arbitrary open continuous surjective homomorphism. Then the group-theoretical algorithm of the mono-anabelian reconstruction concerning ${\rm Ine}(\pi_{1}^{\rm t}(U_{X_{i}}))$  obtained in Proposition \ref{them-1} is compatible with the surjection $\phi: \pi_{1}^{\rm t}(U_{X_{1}}) \twoheadrightarrow \pi_{1}^{\rm t}(U_{X_{2}})$. Namely, the following holds: Let $\widetilde e_{2} \in D_{\widetilde X_{2}}$ and $I_{\widetilde e_{2}} \in {\rm Ine}(\pi_{1}^{\rm t}(U_{X_{2}}))$ the inertia subgroup associated to $\widetilde e_{2}$. Then there exists an inertia subgroup $I_{\widetilde e_{1}} \in {\rm Ine}(\pi_{1}^{\rm t}(U_{X_{1}}))$ associated to a point $\widetilde e_{1} \in D_{\widetilde X_{1}}$ such that $$\phi(I_{\widetilde e_{1}})=I_{\widetilde e_{2}},$$ and that the restriction homomorphism $\phi|_{I_{\widetilde e_{1}}}: I_{\widetilde e_{1}} \twoheadrightarrow I_{\widetilde e_{2}}$ is an isomorphism.

\end{theorem}

\begin{proof}

If $n_{X}=0$, then the theorem is trivial. We suppose $n_{X}>0$. Let $m>>0$ be an integer such that $(m, p)=1$. We put $K_{i}:= \text{ker}(\pi_{1}^{\rm t}(U_{X_{i}}) \twoheadrightarrow \pi_{1}^{\rm t}(U_{X_{i}})^{\rm ab}\otimes\mbZ/m\mbZ)$, $i \in \{1, 2\}$. Write $(X_{K_{i}}, D_{K_{i}})$ for the smooth pointed stable curve of type $(g_{X_{K_{i}}}, n_{X_{K_{i}}})$ over $k_{i}$ induced by $K_{i}$. Moreover, the condition $m>>0$ implies  $g_{X_{K_{1}}}=g_{X_{K_{2}}} \geq 2, \ n_{X_{K_{1}}}=n_{X_{K_{2}}}\geq 2.$

By applying Lemma \ref{lem-7}, we may choose a set of open subgroups $C_{X_{2}}:= \{H_{2, j}\}_{j \in \mbZ_{> 0}}$ of $\pi_{1}^{\rm t}(U_{X_{2}})$ such that the following three conditions hold:
 \begin{itemize}
   \item $H_{2, 1}=K_{2}$;
   \item  $\varprojlim_{j}\pi_{1}^{\rm t}(U_{X_{2}})/H_{2, j} \simeq \pi_{1}^{\rm t}(U_{X_{2}})$ (i.e. $C_{X_{2}}$ is a cofinal system);
   \item write $\{H_{1, j}:= \phi^{-1}(H_{2, j})\}_{j \in \mbZ_{> 0}}$ for the set of open subgroups of $\pi_{1}^{\rm t}(U_{X_{1}})$ induced by $\phi$, and for each $j \in \mbZ_{> 0}$, write $(X_{H_{i, j}}, D_{X_{H_{i, j}}})$, $i\in \{1, 2\}$, for the smooth pointed stable curve of type $(g_{X_{H_{i,j}}}, n_{X_{H_{i,j}}})$ over $k_{i}$ induced by $H_{i, j}$, then we have $(g_{X_{H_{1, j}}}, n_{X_{H_{1, j}}})=(g_{X_{H_{2, j}}}, n_{X_{H_{2, j}}})$.
 \end{itemize}

For each $j \in \mbZ_{> 0}$, we write $e_{X_{H_{2, j}}} \in D_{X_{H_{2, j}}}$ for the image of $\widetilde e_{2}$. Then we obtain a sequence of marked points
$$\mcI_{\widetilde e_{2}}^{C_{X_{2}}}: \dots \mapsto e_{H_{2, 2}}\mapsto e_{H_{2, 1}}.$$ Proposition \ref{pro-3} implies that, for each $H_{2, j}$, $j\in \mbZ_{> 0}$, the set $D^{\rm gp}_{X_{H_{2, j}}}$ can be mono-anabelian reconstructed from $H_{2, j}$. For each $e_{X_{H_{2, j}}} \in D_{X_{H_{2, j}}}$, we denote by $$e^{\rm gp}_{X_{H_{2, j}}}:= \vartheta_{X_{H_{2, j}}}^{-1}(e_{X_{H_{2, j}}}).$$ Then the sequence of marked points $\mcI^{C_{X}}_{\widetilde e_{2}}$ induces a sequence $$\mcI^{C_{X}}_{\widetilde e_{2}^{\rm gp}}: \dots \mapsto e^{\rm gp}_{X_{H_{2, 2}}}\mapsto e^{\rm gp}_{X_{H_{2,1}}}.$$ Then Remark \ref{rem-pro-3-1} implies that the inertia subgroup associated to $\widetilde e_{2}$ is equal to the stabilizer of $\mcI^{C_{X}}_{\widetilde e_{2}^{\rm gp}}$.

By Lemma \ref{lem-5} and Lemma \ref{lem-7}, $\mcI_{\widetilde e^{\rm gp}_{2}}^{C_{X_{2}}}$ induces a sequence as follows: $$\dots \mapsto e^{\rm gp}_{X_{H_{1, 2}}}:=\rho^{-1}_{\phi|_{H_{1, 2}}}(e^{\rm gp}_{X_{H_{2, 2}}}) \in D^{\rm gp}_{X_{H_{1,2}}}\mapsto e^{\rm gp}_{X_{H_{1, 1}}}:=\rho^{-1}_{\phi|_{H_{1, 1}}}(e^{\rm gp}_{X_{H_{2, 1}}})\in D^{\rm gp}_{X_{H_{1, 1}}}$$ with an action of $I_{\widetilde e_{2}}$. Then Proposition \ref{them-1} implies that we have a sequence $$\dots \mapsto e_{X_{H_{1, 2}}}:= \vartheta_{X_{H_{1, 2}}}(e^{\rm gp}_{X_{H_{1, 2}}})\in D_{X_{H_{1,2}}} \mapsto e_{X_{H_{1, 1}}}:=  \vartheta_{X_{H_{1, 1}}}(e^{\rm gp}_{X_{H_{1, 1}}})\in D_{X_{H_{1,1}}}$$ with an action of $I_{\widetilde e_{2}}$

Let $K_{\text{ker}(\phi)}$ be the subfield of $\widetilde K$ induced by the closed subgroup $\text{ker}(\phi)$ of $\pi_{1}^{\rm t}(U_{X_{1}})$, $\widetilde X_{1, \text{ker}(\phi)}$ the normalization of $X_{1}$ in $K_{\text{ker}(\phi)}$, and $D_{\widetilde X_{1, \text{ker}(\phi)}}$  the inverse image of $D_{X_{1}}$ in $\widetilde X_{1, \text{ker}(\phi)}$. Then the sequence
$$\dots \mapsto e_{X_{H_{1, 2}}} \mapsto e_{X_{H_{1, 1}}}.$$  determines a point $\widetilde e_{1, \text{ker}(\phi)} \in D_{\widetilde X_{1, \text{ker}(\phi)}}.$ We choose a point of $\widetilde e_{1} \in D_{\widetilde X_{1}}$ such that the image of $\widetilde e_{1}$ in $D_{\widetilde X_{1, \text{ker}(\phi)}}$ is $\widetilde e_{1, \text{ker}(\phi)}$. Then we have $\phi(I_{\widetilde e_{1}})=I_{\widetilde e_{2}}$. Moreover, since $I_{\widetilde e_{1}}$ and $I_{\widetilde e_{2}}$ are isomorphic to $\widehat \mbZ(1)^{p'}$, the restriction homomorphism $\phi|_{I_{\widetilde e_{1}}}$ is an isomorphism. This completes the proof of the theorem.
\end{proof}


\subsection{Reconstructions of additive structures via surjections}\label{sec-new6}
We maintain the settings introduced in \ref{sett331}.


\subsubsection{}
Let $\widetilde e_{2}$ be an arbitrary point of $D_{\widetilde X_{2}}$. By applying Theorem \ref{them-2}, there exists a point $\widetilde e_{1}\in D_{\widetilde X_{1}}$ such that $\phi|_{I_{\widetilde e_{1}}}: I_{\widetilde e_{1}} \xrightarrow{\simeq} I_{\widetilde e_{2}}$ is an isomorphism. Write $\overline \mbF_{p, i}, \ i \in \{1, 2\},$ for the algebraic closure of $\mbF_{p}$ in $k_{i}$. We put $$\mbF_{\widetilde e_{i}}:= (I_{\widetilde e_{i}}\otimes_{\mbZ} (\mbQ/\mbZ)_{i}^{p'}) \sqcup\{*_{\widetilde e_{i}}\}, \ \ i \in \{1, 2\},$$ where $\{*_{\widetilde e_{i}}\}$ is an one-point set, and $(\mbQ/\mbZ)_{i}^{p'}$ denotes the prime-to-$p$ part of $\mbQ/\mbZ$ which can be canonically identified with $\bigcup_{(p, m)=1}{\mathbb{\mu}}_{m}(k_{i}).$ Moreover, let $a_{\widetilde e_{i}}$ be a generator of $I_{\widetilde e_{i}}$. Then we have a natural bijection $$I_{\widetilde e_{i}}\otimes_{\mbZ} (\mbQ/\mbZ)_{i}^{p'} \xrightarrow{\simeq}\mbZ \otimes_{\mbZ}(\mbQ/\mbZ)_{i}^{p'}, \ a_{\widetilde e_{i}} \otimes 1 \mapsto 1\otimes 1.$$ Thus, we obtain the following bijections
$$I_{\widetilde e_{i}}\otimes_{\mbZ} (\mbQ/\mbZ)_{i}^{p'} \xrightarrow{\simeq}\mbZ \otimes_{\mbZ}  (\mbQ/\mbZ)_{i}^{p'} \xrightarrow{\simeq} \bigcup_{(p, m)=1}{\mathbb{\mu}}_{m}(k_{i}) \xrightarrow{\simeq} \overline \mbF_{p, i}^{\times}.$$ This means that
$\mbF_{\widetilde e_{i}}$ can be identified with $\overline \mbF_{p, i}$ as sets, and hence admits a structure of field whose multiplicative group is $I_{\widetilde e_{i}}\otimes_{\mbZ} (\mbQ/\mbZ)^{p'}_{i}$, and whose zero element is $*_{\widetilde e_{i}}$.

\subsubsection{}
We will  prove that $\phi|_{I_{\widetilde e_{1}}}: I_{\widetilde e_{1}}\xrightarrow{\simeq}I_{\widetilde e_{2}}$ induces an isomorphism $\mbF_{\widetilde e_{1}}\xrightarrow{\simeq} \mbF_{\widetilde e_{2}}$ {\it as fields} (i.e. Proposition \ref{pro-4}). The main idea is as follows: First, we reduce the problem to the case where $n_{X}=3$ by applying Theorem \ref{them-2}. Second, the field structure of $\mbF_{\widetilde e_{i}}$ (i.e., the set of isomorphisms of $\mbF_{\widetilde e_{i}}$ and $\overline \mbF_{p, i}$ as fields) can be translated to certain problem concerning generalized Hasse-Witt invariants (e.g. $\gamma_{\chi_{i}}(M_{\chi_{i}})$ in the proof of Proposition \ref{pro-4}). Then by applying Theorem \ref{them-2} again, we obtained the result by comparing $\gamma_{\chi_{1}}(M_{\chi_{1}})$ with $\gamma_{\chi_{2}}(M_{\chi_{2}})$.

\subsubsection{}
We have the following proposition.

\begin{proposition}\label{pro-4}
The field structure of $\mbF_{\widetilde e_{i}}$, $i\in \{1,2\}$, can be mono-anabelian reconstructed from $\pi_{1}^{\rm t}(U_{X_{i}})$. Moreover, the isomorphism $\phi|_{I_{\widetilde e_{1}}}: I_{\widetilde e_{1}} \xrightarrow{\simeq} I_{\widetilde e_{2}}$ induces an isomorphism $$\theta_{\phi, \widetilde e_{1}, \widetilde e_{2}}: \mbF_{\widetilde e_{1}}\xrightarrow{\simeq} \mbF_{\widetilde e_{2}}$$ as fields. 
\end{proposition}

\begin{proof}

First, we claim that we may assume $n_{X}=3.$ If $g_{X}=0$, then $n_{X}\geq 3$. Suppose that $g_{X}\geq 1$. Theorem \ref{them-2} implies that $\phi: \pi_{1}^{\rm t}(U_{X_{1}}) \twoheadrightarrow\pi_{1}^{\rm t}(U_{X_{2}})$ induces an open continuous surjection $\phi^{\text{\'et}}: \pi_{1}(X_{1})\twoheadrightarrow \pi_{1}(X_{2}).$ Let $H'_{2} \subseteq \pi_{1}(X_{2})$ be an open normal subgroup such that $\#(\pi_{1}(X_{2})/H'_{2})\geq 3$ and $H'_{1} :=(\phi^{\text{\'et}})^{-1}(H_{2}')$. Write $H_{i}\subseteq \pi_{1}^{\rm t}(U_{X_{i}})$, $i \in \{1, 2\}$, for the inverse image of $H_{i}'$ of the natural surjection $\pi_{1}^{\rm t}(U_{X_{i}})\twoheadrightarrow\pi_{1}(X_{i})$, and $(X_{H_{i}}, D_{X_{H_i}})$ for the smooth pointed stable curve of type $(g_{X_{H_i}}, n_{X_{H_{i}}})$ over $k_{i}$ induced by $H_{i}$. Note that  $g_{X_{H_1}}=g_{X_{H_2}}\geq 2$ and $n_{X_{H_{1}}}=n_{X_{H_{2}}}\geq 3$. By replacing $(X_{i}, D_{X_{i}})$ by $(X_{H_{i}}, D_{X_{H_i}})$, we may assume $g_{X}\geq 2$ and $n_{X} \geq 3$. The surjection $\phi$ induces a bijection $$D_{X_{1}} \xrightarrow{\vartheta_{X_{1}}^{-1}} D_{X_{1}}^{\rm gp} \xrightarrow{\rho_{\phi}} D_{X_{2}}^{\rm gp} \xrightarrow{\vartheta_{X_{2}}} D_{X_{2}}.$$ Let $D'_{X_1}:= \{e_{1, 1}, e_{1, 2}, e_{1, 3}\} \subseteq D_{X_{1}}$ and $D'_{X_2} := \{e_{2, 1}:= \vartheta_{X_{2}}\circ\rho_{\phi}\circ\vartheta_{X_{1}}^{-1}(e_{1, 1}), e_{2, 2}:= \vartheta_{X_{2}}\circ\rho_{\phi}\circ\vartheta_{X_{1}}^{-1}(e_{1, 2}), e_{2, 3}:= \vartheta_{X_{2}}\circ\rho_{\phi}\circ\vartheta_{X_{1}}^{-1}(e_{1, 3})\} \subseteq D_{X_{2}}$. Then $(X_{i}, D'_{X_i})$, $i\in \{1, 2\}$, is a smooth pointed stable curve of type $(g_{X}, 3)$ over $k_{i}$. Write $I_{i}$, $i \in \{1, 2\}$, for the closed subgroup of $\pi_{1}^{\rm t}(U_{X_i})$ generated by the inertia subgroups associated to the elements of $D_{\widetilde X_{i}}$ whose images in $D_{X_{i}}$ are contained in $D_{X_{i}} \setminus D'_{X_i}$. Then we have an isomorphism $$\pi_{1}^{\rm t}(X_{i} \setminus D'_{X_i}) \cong \pi_{1}^{\rm t}(U_{X_i})/I_{i}, \ i \in \{1, 2\}.$$ Moreover, Theorem \ref{them-2} implies that $\phi$ induces an open continuous surjective homomorphism $$\phi': \pi_{1}^{\rm t}(X_{1} \setminus D'_{X_1}) \twoheadrightarrow \pi_{1}^{\rm t}(X_{2} \setminus D'_{X_2}).$$ Thus, by replacing $(X_{i}, D_{X_{i}})$, $\pi_{1}^{\rm t}(U_{X_{i}})$, and $\phi$ by $(X_{i}, D'_{X_i})$, $\pi_{1}^{\rm t}(X_{i} \setminus D'_{X_i})$, and $\phi'$, respectively, we may assume $n_{X}=3.$

Let $r \in \mbN$. We denote by $\mbF_{p^{r}, \widetilde e_{i}}, \ i \in \{1, 2\},$ the unique subfield of $\mbF_{\widetilde e_{i}}$ whose cardinality is equal to $p^{r}$. On the other hand, we fix any finite field $\mbF_{p^{r}}$ of cardinality $p^{r}$ and an algebraic closure $\overline \mbF_{p}$ of $\mbF_{p}$. By Proposition \ref{them-1}, we have that $\mbF_{p^{r}, \widetilde e_{i}}^{\times}=I_{\widetilde e_{i}}/(p^{r}-1)$ can be mono-anabelian reconstructed from $\pi_{1}^{\rm t}(U_{X_{i}})$. Then reconstructing the field structure of $\mbF_{p^{r}, \widetilde e_{i}}$ is equivalent to reconstructing $\text{Hom}_{\rm fields}(\mbF_{p^{r}, \widetilde e_{i}}, \mbF_{p^{r}})$ as a subset of $\text{Hom}_{\rm group}(\mbF_{p^{r}, \widetilde e_{i}}^{\times}, \mbF_{p^{r}}^{\times})$. Note that, in order to reconstruct the field structure of $\mbF_{\widetilde e_{i}}$, it is sufficient to reconstruct the subset $\text{Hom}_{\rm fields}(\mbF_{p^{r}, \widetilde e_{i}}, \mbF_{p^{r}})$ for $r$ in a cofinal subset of $\mbN$ with respect to division.

Let $\chi_{i} \in \text{Hom}_{\rm groups}(\pi_{1}^{\rm t}(U_{X_{i}})^{\rm ab}\otimes \mbZ/(p^{r}-1)\mbZ, \mbF^{\times}_{p^{r}}).$ Write $H_{\chi_{i}}$ for the kernel of $\pi_{1}^{\rm t}(U_{X_{i}}) \twoheadrightarrow \pi_{1}^{\rm t}(U_{X_{i}})^{\rm ab}\otimes \mbZ/(p^{r}-1)\mbZ \overset{\chi_{i}}\rightarrow\mbF_{p^{r}}^{\times}$, $M_{\chi_{i}}$ for $H^{\rm ab}_{\chi_{i}} \otimes \mbF_{p}$, and $(X_{H_{\chi_{i}}}, D_{X_{H_{\chi_{i}}}})$ for the smooth pointed stable  curve over $k_{i}$ induced by $H_{\chi_{i}}$. 
We define $$M_{\chi_{i}}[\chi_{i}]:= \{a \in M_{\chi_{i}} \otimes_{\mbF_{p}} \overline \mbF_{p} \ | \ \sigma(a)=\chi_{i}(\sigma)a \ \text{for all} \ \sigma \in \pi_{1}^{\rm t}(U_{X_{i}})^{\rm ab}\otimes \mbZ/(p^{r}-1)\mbZ \}$$
and $\gamma_{\chi_{i}}(M_{\chi_{i}}):=\text{dim}_{\overline \mbF_{p}}(M_{\chi_{i}}[\chi_{i}])$ (i.e. a generalized Hasse-Witt invariant (see \cite[Section 2.2]{Y5})). Then \cite[Remark 3.7]{T4} implies  $\gamma_{\chi_{i}}(M_{\chi_{i}}) \leq g_{X}+1$. Moreover, we define two maps $$\text{Res}_{i, r}: \text{Hom}_{\rm groups}(\pi_{1}^{\rm t}(U_{X_{i}})^{\rm ab}\otimes \mbZ/(p^{r}-1)\mbZ, \mbF^{\times}_{p^{r}})\rightarrow \text{Hom}_{\rm groups}(\mbF^{\times}_{p^{r}, \widetilde e_{i}}, \mbF_{p^{r}}^{\times}),$$
$$\Gamma_{i, r}: \text{Hom}_{\rm groups}(\pi_{1}^{\rm t}(U_{X_{i}})^{\rm ab}\otimes \mbZ/(p^{r}-1)\mbZ, \mbF^{\times}_{p^{r}}) \rightarrow \mbZ_{\geq 0},  \ \chi_{i}\mapsto\gamma_{\chi_{i}}(M_{\chi_{i}}),$$
where the map $\text{Res}_{i, r}$ is the restriction with respect to the natural inclusion $\mbF^{\times}_{p^{r}, \widetilde e_{i}} \hookrightarrow \pi_{1}^{\rm t}(U_{X_{i}})^{\rm ab}\otimes \mbZ/(p^{r}-1)\mbZ.$ 

Let $m_{0}$ be the product of all prime numbers $\leq p-2$ if $p\neq 2, 3$ and $m_{0}=1$ if $p=2, 3$. Let $r_{0}$ be the order of $p$ in the multiplicative group $(\mbZ/m_{0}\mbZ)^{\times}$. Then \cite[Claim 5.4]{T4} implies the following result:
\begin{quote}
there exists a constant $C(g_{X})$ which depends only on $g_{X}$ such that, for each $r > \text{log}_{p}(C(g_{X})+1)$ divisible by $r_{0}$, we have $$\text{Hom}_{\rm fields}(\mbF_{p^{r}, \widetilde e_{i}}, \mbF_{p^{r}})= \text{Hom}^{\rm surj}_{\rm groups}(\mbF_{p^{r}, \widetilde e_{i}}^{\times}, \mbF_{p^{r}}^{\times}) \setminus \text{Res}_{i, r}(\Gamma_{i, r}^{-1}(\{g_{X}+1\})), \ i \in \{1, 2\},$$
where $\text{Hom}^{\rm surj}_{\rm groups}(-, -)$ denotes the set of surjections whose elements are contained in $\text{Hom}_{\rm groups}(-, -)$.
\end{quote}
Let $\kappa_{2} \in \text{Hom}_{\rm groups}(\pi_{1}^{\rm t}(U_{X_{2}})^{\rm ab}\otimes \mbZ/(p^{r}-1)\mbZ, \mbF^{\times}_{p^{r}}).$ Then $\phi$ induces a character $$\kappa_{1} \in \text{Hom}_{\rm groups}(\pi_{1}^{\rm t}(U_{X_{1}})^{\rm ab}\otimes \mbZ/(p^{r}-1)\mbZ, \mbF^{\times}_{p^{r}}).$$ Moreover, the surjection $\phi|_{H_{\kappa_{1}}}$ induces a surjection $M_{\kappa_{1}}[\kappa_{1}] \twoheadrightarrow M_{\kappa_{2}}[\kappa_{2}].$  Suppose that $\kappa_{2} \in \Gamma_{2, r}^{-1}(\{g_{X}+1\})$. The surjection $M_{\kappa_{1}}[\kappa_{1}] \twoheadrightarrow M_{\kappa_{2}}[\kappa_{2}]$ implies $\gamma_{\kappa_{1}}(M_{\kappa_{1}})=g_{X}+1$. This means  $\kappa_{1} \in \Gamma_{1, r}^{-1}(\{g_{X}+1\})$. On the other hand, the isomorphism $\phi|_{I_{\widetilde e_{1}}}: I_{\widetilde e_{1}} \xrightarrow{\simeq}I_{\widetilde e_{2}}$ induces an injection $$\text{Res}_{2, r}(\Gamma_{2, r}^{-1}(\{g_{X}+1\})) \hookrightarrow \text{Res}_{1, r}(\Gamma_{1, r}^{-1}(\{g_{X}+1\})).$$ Since $\#(\text{Hom}_{\rm fields}(\mbF_{p^{r}, \widetilde e_{1}}, \mbF_{p^{r}}))=\#(\text{Hom}_{\rm fields}(\mbF_{p^{r}, \widetilde e_{2}}, \mbF_{p^{r}}))$, we obtain that $\phi|_{I_{\widetilde e_{1}}}$ induces a bijection $\text{Hom}_{\rm fields}(\mbF_{p^{r}, \widetilde e_{2}}, \mbF_{p^{r}}) \xrightarrow{\simeq} \text{Hom}_{\rm fields}(\mbF_{p^{r}, \widetilde e_{1}}, \mbF_{p^{r}}).$ Thus, $\phi|_{I_{\widetilde e_{1}}}$ induces a bijection $$\text{Hom}_{\rm fields}(\mbF_{\widetilde e_{2}}, \overline \mbF_{p}) \xrightarrow{\simeq} \text{Hom}_{\rm fields}(\mbF_{\widetilde e_{1}}, \overline \mbF_{p}).$$ If we choose $\overline \mbF_{p}=\mbF_{\widetilde e_{2}}$, then the image of $\text{id}_{\mbF_{\widetilde e_{2}}}$ via the bijection above induces an isomorphism $\theta_{\phi, \widetilde e_{1}, \widetilde e_{2}}:\mbF_{\widetilde e_{1}}\xrightarrow{\simeq}\mbF_{\widetilde e_{2}}$ as fields. This completes the proof of the proposition.
\end{proof}

\section{Main theorems}\label{sec-5}

\subsection{The first main theorem}
In this subsection, we apply the results obtained in the previous sections to prove that the curves of type $(0,n)$ over $\overline \mbF_{p}$ can be reconstructed group-theoretically from {\it open continuous homomorphism} (i.e. Theorem \ref{them-3}). 

\subsubsection{\bf Settings}
We fix some notation. Let $(X_{i}, D_{X_{i}})$, $i \in \{1, 2\}$, be a smooth pointed stable curve of type $(g_{X}, n_{X})$ over an algebraically closed field $k_{i}$ of characteristic $p>0$, $U_{X_i}:= X_{i} \setminus D_{X_{i}}$, $\pi_{1}^{\rm t}(U_{X_{i}})$ the tame fundamental group of $U_{X_{i}}$, $\pi_{1}(X_{i})$ the \'etale fundamental group of $X_{i}$, and $(\widetilde X_{i}, D_{\widetilde X_{i}})$ the universal tame covering of $(X_{i}, D_{X_i})$ associated to $\pi_{1}^{\rm t}(U_{X_{i}})$ (\ref{unicov313}). Let $k_{i}^{\rm m}$, $i\in \{1, 2\}$, be the {\it minimal} algebraically closed subfield of $k_{i}$ over which $U_{X_{i}}$ can be defined. Thus, by considering the function field of $X_{i}$, we obtain a smooth pointed stable curve $(X^{\rm m}_{i}, D_{X^{\rm m}_{i}})$  (i.e., {\it a minimal model of} $(X_{i}, D_{X_{i}})$ (cf. \cite[Definition 1.30 and Lemma 1.31]{T3})) such that $U_{X_{i}} \cong U_{X^{\rm m}_{i}} \times_{k^{\rm m}_{i}}k_{i}$ as $k_{i}$-schemes, where $U_{X_{i}^{\rm m}}:= X^{\rm m}_{i} \setminus D_{X_{i}^{\rm m}}.$ Note that $\pi_{1}^{\rm t}(U_{X^{\rm m}_{i}})$ is naturally isomorphic to $\pi_{1}^{\rm t}(U_{X_{i}})$. We shall denote by $\overline \mbF_{p, i}$ the algebraic closure of $\mbF_{p}$ in $k_{i}$. Moreover, we put
\begin{eqnarray*}
d_{(X_{i}, D_{X_{i}})}:=  \left\{ \begin{array}{ll}
0, & \text{if} \ k_{i}^{\rm m} \cong \overline \mbF_{p, i},
\\
1, & \text{if} \ k_{i}^{\rm m}\not\cong \overline \mbF_{p, i}.
\end{array} \right.
\end{eqnarray*}

\subsubsection{} Firstly, we have the following lemma.
\begin{lemma}\label{lemsurj}
Let $\phi: \pi_{1}^{\rm t}(U_{X_{1}}) \rightarrow\pi_{1}^{\rm t}(U_{X_{2}})$ be an arbitrary open continuous homomorphism. Then $\phi$ is a surjection.
\end{lemma}

\begin{proof}
We denote by $\Pi_{\phi}$ the image of $\phi$ which is an open subgroup of $\pi_{1}^{\rm t}(U_{X_{2}})$. Let $(X_{\phi}, D_{X_{\phi}})$ be the smooth pointed stable curve of type $(g_{X_{\phi}}, n_{X_{\phi}})$ over $k_{2}$ induced by $\Pi_{\phi}$ and $f_{\phi}: (X_{\phi}, D_{X_{\phi}}) \rightarrow(X_{2}, D_{X_{2}})$ the tame covering of smooth pointed stable curves over $k_{2}$ induced by the inclusion $\Pi_{\phi}\hookrightarrow \pi_{1}^{\rm t}(U_{X_{2}})$. Since $f_{\phi}$ is a tame covering, we have that $n_{X_{\phi}} \geq n_{X}$. On the other hand, if $g_{X}=0$, we have $g_{\phi} \geq 0$. If $g_{X}>0$, the Riemann-Hurwitz formula implies  $g_{X_{\phi}}\geq [\pi_{1}^{\rm t}(U_{X_{2}}): \Pi_{\phi}](g_{X}-1)+1\geq g_{X}$. Then we have $g_{\phi}\geq g_{X}$ and $n_{X_{\phi}}\geq n_{X}$. Note that $\pi_{1}^{\rm t}(U_{X_{1}}) \twoheadrightarrow\Pi_{\phi} \hookrightarrow \pi_{1}^{\rm t}(U_{X_{2}})$ implies  $$2g_{X}+n_{X}-1\geq 2g_{X_{\phi}}+n_{X_{\phi}}-1 \geq 2g_{X}+n_{X}-1.$$ Then we obtain that $2g_{X}+n_{X}-1=2g_{X_{\phi}}+n_{X_{\phi}}-1$. Moreover, Proposition \ref{coro-p-average} (ii) and the natural surjection $\pi_{1}^{\rm t}(U_{X_{1}}) \twoheadrightarrow \Pi_{\phi}$ induced by $\phi$ imply that $g_{X}\geq g_{X_{\phi}}$. Then we obtain that $g_{X}=g_{X_{\phi}}.$ Thus, we have $(g_{X}, n_{X})=(g_{X_{\phi}}, n_{X_{\phi}}).$ This means that the tame covering $f_{\phi}: (X_{\phi}, D_{X_{\phi}}) \rightarrow(X_{2}, D_{X_{2}})$ is totally ramified over every marked point of $D_{X_{2}}$.

Let us prove $[\pi_{1}^{\rm t}(U_{X_{2}}): \Pi_{\phi}]=1$. Suppose  $[\pi_{1}^{\rm t}(U_{X_{2}}): \Pi_{\phi}]\neq 1$. Since $f_{\phi}$ is totally ramified, the Riemann-Hurwitz formula implies  $g_{X_{\phi}}>g_{X}$ if $n_{X} \neq 0$ and $g_{X}\neq 0$. This is a contradiction. If $n_{X}=0$, the Riemann-Hurwitz formula implies  $g_{X}=1$ if $g_{X}\neq 0$. This contradicts the assumption that $(X_{i}, D_{X_{i}})$ is a pointed stable curve. Then we obtain $g_{X}=g_{X_{\phi}}=0$. Moreover, by applying the Riemann-Hurwitz formula again, since $n_{X}=n_{X_{\phi}}$, we obtain  $n_{X}=n_{X_{\phi}}=2$. This contradicts the assumption that $(X_{i}, D_{X_{i}})$ is pointed stable curve. Then we have $[\pi_{1}^{\rm t}(U_{X_{2}}): \Pi_{\phi}]=1$. This means that $\phi$ is a surjection.
\end{proof}

\subsubsection{\bf Further settings}
In the remainder of this subsection, we suppose  $(g_{X}, n_{X})=(0, n)$. We fix two marked points $e_{1, \infty}, e_{1, 0} \in D_{X_{1}}$ distinct from each other. Moreover, we choose any field $k_{1}'\cong k_{1}$, and choose any isomorphism $\varphi_{1}: X_{1} \xrightarrow{\simeq} \mbP_{k_{1}'}^{1}$ as schemes such that $\varphi_{1}(e_{1, \infty})=\infty$ and $\varphi_{1}(e_{1, 0})=0$. Then the set of $k_{1}$-rational points $X_{1}(k_{1}) \setminus \{e_{1, \infty}\}$ is equipped with a structure of $\mbF_{p}$-module via the bijection $\varphi_{1}$. Note that since any $k_{1}'$-isomorphism of $\mbP_{k_{1}'}^{1}$ fixing $\infty$ and $0$ is a scalar multiplication, the $\mbF_{p}$-module structure of $X_{1}(k_{1}) \setminus \{e_{1, \infty}\}$ does not depend on the choices of $k'_{1}$ and $\varphi_{1}$ but depends only on the choices of $e_{1, \infty}$ and $
e_{1, 0}$. Then we shall say that {\it $X_{1}(k_{1}) \setminus \{e_{1, \infty}\}$ is equipped with a structure of $\mbF_{p}$-module with respect to $e_{1, \infty}$ and $e_{1, 0}$}.


By applying Theorem \ref{them-2}, in the next lemma, we will prove that Tamagawa's group-theoretical criterion (i.e., \cite[Lemma 3.3]{T2}) for linear conditions is compatible with arbitrary open continuous surjective homomorphism.

\begin{lemma}\label{lem-8}
Let $\phi: \pi_{1}^{\rm t}(U_{X_{1}}) \twoheadrightarrow\pi_{1}^{\rm t}(U_{X_{2}})$ be an open continuous surjective homomorphism. By Lemma \ref{lem-4}, $\phi$ induces a bijection $\rho_{\phi}: D^{\rm gp}_{X_{1}} \xrightarrow{\simeq}D^{\rm gp}_{X_{2}}$. We may identify $D^{\rm gp}_{X_{i}},$ $i\in \{1, 2\}$, with $D_{X_{i}}$ via the bijection $\vartheta_{X_{i}}: D_{X_{i}}^{\rm gp} \xrightarrow{\simeq}D_{X_{i}}$. Write $e_{2, \infty}$ and $e_{2, 0}$ for $\rho_{\phi}(e_{1, \infty})$ and $\rho_{\phi}(e_{1, 0})$, respectively. Let $$\sum_{e_{1}\in D_{X_{1}}\setminus \{e_{1, \infty}, e_{1, 0}\}}b_{e_{1}}e_{1}=e_{1, 0}$$ be a linear condition with respect to $e_{1, \infty}$ and $e_{1, 0}$ on $(X_{1}, D_{X_{1}})$, where $b_{e_{1}} \in \mbF_{p}$ for each $e_{1}\in D_{X_{1}}\setminus \{e_{1, \infty}, e_{1, 0}\}$. Then the linear condition  $$\sum_{e_{1}\in D_{X_{1}}\setminus \{e_{1, \infty}, e_{1, 0}\}}b_{e_{1}}\rho_{\phi}(e_{1})=\rho_{\phi}(e_{1, 0})=e_{2, 0}$$ with respect to $e_{2, \infty}$ and $e_{2, 0}$ on $(X_{2}, D_{X_{2}})$ also holds.
\end{lemma}

\begin{proof}
Let $\widetilde e_{2, \infty} \in D_{\widetilde X_{2}}$ be a point over $e_{2, \infty}$. The set $\mbF_{\widetilde e_{2, \infty}}:= (I_{\widetilde e_{2, \infty}}\otimes_{\mbZ} (\mbQ/\mbZ)^{p'}_{2}) \sqcup \{*_{\widetilde e_{2, \infty}}\}$ admits a structure of field, and Proposition \ref{pro-4} implies that the field structure can be mono-anabelian reconstructed from $\pi_{1}^{\rm t}(U_{X_{2}})$. Theorem \ref{them-2} implies that there exists a point $\widetilde e_{1, \infty} \in D_{\widetilde X_{1}}$ over $e_{1, \infty}$ such that $\phi(I_{\widetilde e_{1, \infty}})=\widetilde e_{2, \infty}$. By Proposition \ref{pro-4} again, the set $\mbF_{\widetilde e_{1, \infty}}:= (I_{\widetilde e_{1, \infty}}\otimes_{\mbZ} (\mbQ/\mbZ)^{p'}_{1}) \sqcup \{*_{\widetilde e_{1, \infty}}\}$ admits a structure of field which can be mono-anabelian reconstructed from $\pi_{1}^{\rm t}(U_{X_{1}})$, and $\phi$ induces an isomorphism $\theta_{\phi, \widetilde e_{1, \infty}, \widetilde e_{2, \infty}}: \mbF_{\widetilde e_{1, \infty}} \xrightarrow{\simeq} \mbF_{\widetilde e_{2, \infty}}$ as fields.

For each $e_{1} \in D_{X_{1}}$, we take $b'_{e_{1}} \in \mbZ_{\geq 0}$ such that $b'_{e_{1}} \equiv b_{e_{1}} \ (\text{mod} \ p)$ and $$\sum_{e_{1}\in D_{X_{1}}\setminus \{e_{1, \infty}, e_{1, 0}\}}b'_{e_{1}} \geq 2.$$ Let $r \geq 1$ such that $p^{r}-2 \geq \sum_{e_{1}\in D_{X_{1}}\setminus \{e_{1, \infty}, e_{1, 0}\}}b'_{e_{1}}.$ For each $\widetilde e_{1} \in D_{\widetilde X_{1}}$ over $e_{1}$, write $I_{\widetilde e_{1}, \rm ab}$ for the image of the natural morphism $I_{\widetilde e_{1}} \hookrightarrow \pi_{1}^{\rm t}(U_{X_{1}}) \twoheadrightarrow \pi_{1}^{\rm t}(U_{X_{1}})^{\rm ab}.$ Moreover, since the image of $I_{\widetilde e_{1}, \rm ab}$ does not depend on the choices of $\widetilde e_{1}$, we may write $I_{e_{1}}$ for $I_{\widetilde e_{1}, \rm ab}$. The structure of maximal prime-to-$p$ quotient of $\pi_{1}^{\rm t}(U_{X_{1}})$ implies that $\pi_{1}^{\rm t}(U_{X_{1}})^{\rm ab}$ is generated by $\{I_{e_{1}}\}_{e_{1} \in D_{X_{1}}}$, and that there exists a generator $a_{e_{1}}$, $e_{1} \in D_{X_{1}}$, of $I_{e_{1}}$ such that $\prod_{e_{1} \in D_{X_{1}}}a_{e_{1}}=1.$ We define $$I_{e_{1, \infty}} \rightarrow \mbZ/(p^{r}-1)\mbZ, \ a_{e_{1, \infty}} \mapsto 1,$$ $$\ I_{e_{1, 0}} \rightarrow \mbZ/(p^{r}-1)\mbZ, \ a_{e_{1, 0}} \mapsto (\sum_{e_{1}\in D_{X_{1}}\setminus \{e_{1, \infty}, e_{1, 0}\}}b'_{e_{1}})-1,$$ and $$I_{e_{1}} \rightarrow \mbZ/(p^{r}-1)\mbZ, \ a_{e_{1}} \mapsto -b'_{e_{1}},  \ e_{1} \in D_{X_{1}}\setminus \{e_{1, \infty}, e_{1, 0}\}.$$ Then the homomorphisms of inertia subgroups defined above induces a surjection $\delta_{1}: \pi_{1}^{\rm t}(U_{X_{1}}) \twoheadrightarrow\pi_{1}^{\rm t}(U_{X_{1}})^{\rm ab} \twoheadrightarrow \mbZ/(p^{r}-1)\mbZ.$ Note that $\text{ker}(\delta_{1})$ does not depend on the choices of the generators $\{a_{e_{1}}\}_{e_{1} \in D_{X_{1}}}$.

Let $I_{\widetilde e_{2}} := \phi(I_{\widetilde e_{1}})$, $\widetilde e_{1} \in D_{\widetilde X_{1}}$, and $I_{e_{2}}$, $e_{2} \in D_{X_{2}}$ be the image of the natural homomorphism $I_{\widetilde e_{2}}  \hookrightarrow \pi_{1}^{\rm t}(U_{X_{2}}) \twoheadrightarrow \pi_{1}^{\rm t}(U_{X_{2}})^{\rm ab}.$ Since $(p, p^{r}-1)=1$, by Theorem \ref{them-2}, $\delta_{1}$ and the isomorphism $\phi^{p'}: \pi_{1}^{\rm t}(U_{X_{1}})^{p'} \xrightarrow{\simeq} \pi_{1}^{\rm t}(U_{X_{2}})^{p'}$ imply the following homomorphisms of inertia subgroups: $$I_{e_{2, \infty}} \rightarrow \mbZ/(p^{r}-1)\mbZ, \ a_{e_{2, \infty}} \mapsto 1,$$ $$\ I_{e_{2, 0}} \rightarrow \mbZ/(p^{r}-1)\mbZ, \ a_{e_{2, 0}} \mapsto (\sum_{e_{1}\in D_{X_{1}}\setminus \{e_{1, \infty}, e_{1, 0}\}}b'_{e_{1}})-1,$$ and $$I_{e_{2}} \rightarrow \mbZ/(p^{r}-1)\mbZ, \ a_{e_{2}} \mapsto -b'_{e_{1}},  \ e_{2} \in D_{X_{2}}\setminus \{e_{2, \infty}, e_{2, 0}\},$$ where $a_{e_{2}}$, $e_{2} \in D_{X_{2}}$, denotes the element induced by $a_{e_{1}}$, $e_{1} \in D_{X_{1}}$, via $\phi$. Then the homomorphisms of inertia subgroups defined above induces a sujection $\delta_{2}: \pi_{1}^{\rm t}(U_{X_{2}}) \twoheadrightarrow \pi_{1}^{\rm t}(U_{X_{2}})^{\rm ab} \twoheadrightarrow\mbZ/(p^{r}-1)\mbZ.$

We put $H_{\delta_{i}}:=\text{ker}(\delta_{i}),\ M_{\delta_{i}} := H^{\rm ab}_{\delta_{i}} \otimes\mbF_{p}, \ i \in \{1, 2\}.$ Write $(X_{H_{\delta_{i}}}, D_{X_{H_{\delta_{i}}}})$ for the smooth pointed stable curve over $k_{i}$ induced by $H_{\delta_{i}}$, where  $H_{\delta_{1}}=\phi^{-1}(H_{\delta_{2}})$. The $\mbF_{p}$-vector space $M_{\delta_{i}}$ admits a natural action of $I_{\widetilde e_{i, \infty}}$ via conjugation which coincides with the action via the following character $$\chi_{I_{\widetilde e_{i, \infty}}, r}: I_{\widetilde e_{i, \infty}} \hookrightarrow \pi_{1}^{\rm t}(U_{X_{i}})  \overset{\delta_{i}}\twoheadrightarrow \mbZ/(p^{r}-1)\mbZ=I_{\widetilde e_{i, \infty}}/(p^{r}-1)\hookrightarrow \mbF_{\widetilde e_{i, \infty}}^{\times}, \ i \in \{1, 2\}.$$ We put $M_{\delta_{i}}[\chi_{I_{\widetilde e_{i, \infty}}, r}]:= \{a\in M_{\delta_{i}}\otimes_{\mbF_{p}} \mbF_{\widetilde e_{i, \infty}} \ |\ \sigma(a)=\chi_{I_{\widetilde e_{i, \infty}}, r}(\sigma)a \ \text{for all} \ \sigma \in I_{\widetilde e_{i, \infty}}\}$ (in fact, $\text{dim}_{\mbF_{\widetilde e_{i, \infty}}}(M_{\delta_{i}}[\chi_{I_{\widetilde e_{i, \infty}}, r}])$ is the first generalized Hasse-Witt invariant associated to the tame covering of $U_{X_{i}}$ corresponding to $H_{\delta_{i}} \subseteq \pi_{1}^{\rm t}(U_{X_{i}})$ (see \cite[Section 2.2]{Y5})). Since the action of $I_{\widetilde e_{i}, \infty}$ on $M_{\delta_{i}}$ is semi-simple, we obtain a surjection $M_{\delta_{1}}[\chi_{I_{\widetilde e_{1, \infty}}, r}] \twoheadrightarrow M_{\delta_{2}}[\chi_{I_{\widetilde e_{2, \infty}}, r}]$ induced by $\phi|_{H_{\delta_{1}}}$ and $\theta_{\phi, \widetilde e_{1, \infty}, \widetilde e_{2, \infty}}$. On the other hand, the third and the final paragraphs of the proof of \cite[Lemma 3.3]{T2} imply that the linear condition  $$\sum_{e_{1}\in D_{X_{1}}\setminus \{e_{1, \infty}, e_{1, 0}\}}b_{e_{1}}e_{1}=e_{1, 0}$$ with respect to $e_{1, \infty}$ and $e_{1, 0}$ on $(X_{1}, D_{X_{1}})$ holds if and only if $M_{\delta_{1}}[\chi_{I_{\widetilde e_{1, \infty}}, r}]=0$. Thus, we obtain $M_{\delta_{2}}[\chi_{I_{\widetilde e_{2, \infty}}, r}] =0$. Then the third and the final paragraphs of the proof of \cite[Lemma 3.3]{T2} imply that the linear condition $$\sum_{e_{1}\in D_{X_{1}}\setminus \{e_{1, \infty}, e_{1, 0}\}}b_{e_{1}}\rho_{\phi}(e_{1})=e_{2, 0}$$ with respect to $e_{2, \infty}$ and $e_{2, 0}$ on $(X_{2}, D_{X_{2}})$ holds. This completes the proof of the lemma.
\end{proof}

\begin{remark}\label{rem-lem-8-1}
Note that, if $X_{1}=\mbP_{k}^{1}$, then the linear condition is as follows: $$\sum_{e_{1}\in D_{X_{1}}\setminus \{\infty, 0\}}b_{e_{1}}e_{1}=0$$ with respect to $\infty$ and $0$.
\end{remark}

\subsubsection{}
Now, we prove the first main theorem of the present paper.

\begin{theorem}\label{them-3}
We maintain the notation and settings introduced above. Then we have the following claims.
\begin{enumerate}
  \item $d_{(X_{i}, D_{X_{i}})}$, $i\in \{1, 2\}$, can be mono-anabelian reconstructed from $\pi_{1}^{\rm t}(U_{X_{i}})$.
  \item  Suppose  $k_{1}^{\rm m} \cong \overline \mbF_{p, 1}$. Then the set of open continuous homomorphisms $${\rm Hom}^{\rm op}_{\rm pg}(\pi_{1}^{\rm t}(U_{X_{1}}), \pi_{1}^{\rm t}(U_{X_{2}}))$$ is non-empty if and only if $U_{X^{\rm m}_1} \cong U_{X^{\rm m}_2}$ as schemes. In particular, if this is the case, we have $k_{2}^{\rm m} \cong \overline \mbF_{p, 2}$ and $${\rm Hom}^{\rm op}_{\rm pg}(\pi_{1}^{\rm t}(U_{X_{1}}), \pi_{1}^{\rm t}(U_{X_{2}})) = {\rm Isom}_{\rm pg}(\pi_{1}^{\rm t}(U_{X_{1}}), \pi_{1}^{\rm t}(U_{X_{2}})).$$
\end{enumerate}

\end{theorem}

\begin{proof}
Firstly, let us prove (2). The ``if" part of (2) is trivial. We treat the ``only if" part of (2). Suppose that ${\rm Hom}^{\rm op}_{\rm pg}(\pi_{1}^{\rm t}(U_{X_{1}}), \pi_{1}^{\rm t}(U_{X_{2}}))$ is a non-empty set, and let $\phi \in {\rm Hom}^{\rm op}_{\rm pg}(\pi_{1}^{\rm t}(U_{X_{1}}), \pi_{1}^{\rm t}(U_{X_{2}})).$ Then Lemma \ref{lemsurj} implies that $\phi$ is a surjection.

We identify $D^{\rm gp}_{X_{i}},$ $i\in \{1, 2\}$, with $D_{X_{i}}$ via the bijection $\vartheta_{X_{i}}: D_{X_{i}}^{\rm gp} \xrightarrow{\simeq} D_{X_{i}}$. Since $\phi$ is a surjection, Lemma \ref{lem-4} implies that $\phi$ induces a bijection $\rho_{\phi}: D_{X_{1}} \xrightarrow{\simeq} D_{X_{2}}.$ We put $e_{2, 0}:= \rho_{\phi}(e_{1, 0})$ and $e_{2, \infty}:=\rho_{\phi}(e_{1, \infty})$. Let $\widetilde e_{2, 0} \in D_{\widetilde X_{2}}$ be a point over $e_{2, 0}$. Theorem \ref{them-2} implies that there exists a point $\widetilde e_{1, 0} \in D_{\widetilde X_{1}}$ over $e_{1, 0}$ such that $\phi(I_{\widetilde e_{1, 0}})=I_{\widetilde e_{2, 0}}$. Then $\mbF_{\widetilde e_{i, 0}}:= (I_{\widetilde e_{i, 0}}\otimes_{\mbZ} (\mbQ/\mbZ)_{i}^{p'}) \sqcup \{*_{\widetilde e_{i, 0}}\}, \ i\in \{1, 2\},$ admits a structure of field. Moreover, Proposition \ref{pro-4} implies that the field structure can be mono-anabelian reconstructed from $\pi_{1}^{\rm t}(U_{X_{i}})$, and that $\phi$ induces a field isomorphism $\theta_{\phi, \widetilde e_{1, 0}, \widetilde e_{2, 0}}: \mbF_{\widetilde e_{1, 0}} \xrightarrow{\simeq} \mbF_{\widetilde e_{2, 0}}.$

Proposition \ref{proposition 1} (1) implies that $n$ can be mono-anabelian reconstructed from $\pi_{1}^{\rm t}(U_{X_{i}})$, $i \in \{1, 2\}$. If $n=3$, (ii) is trivial, so we may assume $n\geq 4.$ Moreover, since $k_{1}^{\rm m}\cong \overline \mbF_{p, 1}$, without loss of generality, we may assume $k_{1}=\overline \mbF_{p, 1}=\mbF_{\widetilde e_{1, 0}},$ $X_{1}=\mbP^{1}_{\overline \mbF_{p, 1}}$, and $$D_{X_{1}}= \{e_{1, \infty}=\infty, e_{1, 0}=0, e_{1, 1}=1, e_{1, 2}, \dots, e_{1, n-2}\}.$$ Here, $e_{1, 2}, \dots, e_{1, n-2} \in \overline \mbF_{p, 1}\setminus \{e_{1, 0}, e_{1, 1}\}$ are distinct from each other.

{\bf Step 1}: In this step, we will construct a linear condition on a certain tame covering of  $(X_{1}, D_{X_{1}})$.

We see that there exists a natural number $r$ prime to $p$ such that $\mbF_{p}(\zeta_{r})$ contains $r$th roots of $e_{1, 2}, \dots, e_{1, n-2}$, where $\zeta_{r}$ denotes a fixed primitive $r$th root of unity in $\overline \mbF_{p, 1}$. Let $s:= [\mbF_{p}(\zeta_{r}): \mbF_{p}]$. For each $e_{1, u} \in \{e_{1, 2}, \dots, e_{1, n-2}\}$, we fix an $r$th root $e_{1, u}^{1/r}$ in $\overline \mbF_{p, 1}$. Then we have $$e_{1, u}^{1/r}=\sum_{v=0}^{s-1}b_{1, uv}\zeta_{r}^{v}, \ u \in \{2, \dots, n-2\},$$ where $b_{1, uv} \in \mbF_{p}$ for each $u\in \{2, \dots, n-2\}$ and each $v\in \{0, \dots, s-1\}$.

Let $X_{1} \setminus \{e_{1, \infty}\} =\spec \overline \mbF_{p, 1}[x_{1}]$, $f_{H_{1}}: (X_{H_{1}}, D_{X_{H_{1}}}) \rightarrow(X_{1}, D_{X_{1}})$ the Galois tame covering over $\overline \mbF_{p, 1}$ with Galois group $\mbZ/r\mbZ$ determined by the equation $y_{1}^{r}=x_{1}$, and $H_{1}$ the open normal subgroup of $\pi_{1}^{\rm t}(U_{X_{1}})$ induced by the tame covering $f_{H_{1}}$. Then $f_{H_{1}}$ is totally ramified over $\{e_{1, \infty}=\infty, e_{1, 0}=0\}$ and is \'etale over $D_{X_{1}} \setminus \{\infty, 0\}$. Note that $X_{H_{1}}= \mbP_{\overline \mbF_{p, 1}}^{1}$, and the points of $D_{X_{H_{1}}}$ over $\{e_{1, \infty}, e_{1, 0}\}$ are $\{e_{H_{1}, \infty}:=\infty, e_{H_{1}, 0}:= 0\}$. We put $$e_{H_{1}, u}:= e_{1, u}^{1/r} \in D_{X_{H_{1}}}, \ u \in \{2, \dots, n-2\},\ e_{H_{1}, 1}^{v}:=\zeta^{v}_{r} \in D_{X_{H_{1}}}, \ v \in \{0, \dots, s-1\}.$$ Thus, we obtain a linear condition
$$e_{H_{1}, u}=\sum_{v=0}^{s-1}b_{1, uv}e_{H_{1}, 1}^{v}$$ with respect to $e_{H_{1}, \infty}$ and $e_{H_{1}, 0}$ on $(X_{H_{1}}, D_{X_{H_{1}}})$ for each $\ u \in \{2, \dots, n-2\}.$

{\bf Step 2}: In this step, we will prove that the linear condition on a certain tame covering of $(X_{1}, D_{X_{1}})$ constructed in Step 1 induces a linear condition on a certain tame covering of $(X_{2}, D_{X_{2}})$ via the surjection $\phi$.

Write $H_{2}$ for $\phi(H_{1})$. Since $(r, p)=1$, we have the following commutative diagram of profinite groups:
\[
\begin{CD}
H_{1} @>\phi|_{H_{1}}>> H_{2}
\\
@VVV@VVV
\\
\pi_{1}^{\rm t}(U_{X_{1}})@>\phi>>\pi_{1}^{\rm t}(U_{X_{2}})
\\
@VVV@VVV
\\
\mbZ/r\mbZ@=\mbZ/r\mbZ.
\end{CD}
\]
We denote by $f_{H_{2}}: (X_{H_{2}}, D_{X_{H_{2}}}) \rightarrow (X_{2}, D_{X_{2}})$ the Galois tame covering over $\overline \mbF_{p, 2}$ with Galois group $\mbZ/r\mbZ$ induced by $H_{2}$. Note that Lemma \ref{lem-6} implies that $(X_{H_{1}}, D_{X_{H_{1}}})$ and $(X_{H_{2}}, D_{X_{H_{2}}})$ are equal types. Moreover, Lemma \ref{lem-5} implies that the following commutative diagram can be mono-anabelian reconstructed from the commutative diagram of profinite groups above:
\[
\begin{CD}
D_{X_{H_{1}}} @>\rho_{\phi|_{H_{1}}}>> D_{X_{H_{2}}}
\\
@VVV@VVV
\\
D_{X_{1}}@>\rho_{\phi}>>D_{X_{2}}.
\end{CD}
\]
We put $$e_{2, \infty} := \rho_{\phi}(e_{1, \infty}), \ e_{2, u} := \rho_{\phi}(e_{1, u}), \ u\in \{0, \dots, n-2\},$$ $$e_{H_{2}, \infty} := \rho_{\phi|_{H_{1}}}(e_{H_{1}, \infty}), \ e_{H_{2}, 0} := \rho_{\phi|_{H_{1}}}(e_{H_{1}, 0}), \ e_{H_{2}, u} := \rho_{\phi|_{H_{1}}}(e_{H_{1}, u}), \ u\in \{2, \dots, n-2\},$$ and $$e_{H_{2}, 1}^{v} := \rho_{\phi|_{H_{1}}}(e_{H_{1}, 1}^{v}), \ v\in \{0, \dots, s-1\}.$$

Remark \ref{rem-lem-5-1} implies that $f_{H_{2}}$ is totally ramified over $\{e_{2, \infty}, e_{2, 0}\}$ and is \'etale over $X_{2} \setminus \{e_{2, \infty}, e_{2, 0}\}$. Then we may assume that $X_{2}=\mbP^{1}_{k_{2}}$, and that $e_{2, \infty}=\infty, e_{2, 0}=0, e_{2,1}=1$. We regard $e_{2, u}$, $u \in \{2, \dots, n-2\}$, as an element of $k_{2} \setminus \{e_{2, 0}, e_{2, 1}\}$. Moreover, we have $e_{H_{2}, \infty}=\infty$ and $e_{H_{2}, 0}=0$.

We put $\xi_{r}:=\theta_{\phi, \widetilde e_{1, 0}, \widetilde e_{2, 0}}(\zeta_{r})$ which is an $r$th root of unity in $\mbF_{\widetilde e_{2, 0}}$. Since $\zeta_{r}(e_{H_{1}, 1}^{v})=e_{H_{1}, 1}^{v+1}$, we obtain $\xi_{r}(e_{H_{2}, 1}^{v})=e_{H_{2}, 1}^{v+1}, \ v\in\{0, \dots, s-2\}.$ By applying Lemma \ref{lem-8} for $\phi|_{H_{1}}: H_{1} \twoheadrightarrow H_{2}$, the following linear condition
$$e_{H_{2}, u}=\sum_{v=0}^{s-1}b_{1, uv}\xi_{r}^{v}(e^{0}_{H_{2}, 1})$$ with respect to $e_{H_{2}, \infty}$ and $e_{H_{2}, 0}$ on $(X_{H_{2}}, D_{X_{H_{2}}})$
holds for each $u\in \{2, \dots, n-2\}$. Since $(e_{H_{2}, u})^{r}=e_{2, u}$, $u\in \{2, \dots, n-2\}$, we obtain $$e_{2, u}=(\sum_{v=0}^{s-1}b_{1, uv}\xi_{r}^{v}(e^{0}_{H_{2}, 1}))^{r}.$$ Moreover, if we put $e_{H_{2}, 1}^{0}=1$, then we obtain that $$e_{2, u}=(\sum_{v=0}^{s-1}b_{1, uv}\xi_{r}^{v})^{r}$$ for each $u\in \{2, \dots, n-2\}$. Since $\theta_{\phi, \widetilde e_{1, 0}, \widetilde e_{2, 0}}(\zeta_{r})=\xi_{r}$, we have $$U_{X_{1}}=U_{X_{1}^{\rm m}}=\mbP^{1}_{\overline \mbF_{p, 1}} \setminus \{e_{1, \infty}=\infty, e_{1, 0}=0, e_{1, 1}=1, e_{1, 2}, \dots, e_{1, n-2}\}$$ $$\xrightarrow{\simeq} \mbP^{1}_{\mbF_{\widetilde e_{2, 0}}} \setminus \{e_{2, \infty}=\infty, e_{2, 0}=0, e_{2, 1}=1, \theta_{\phi, \widetilde e_{1, 0}, \widetilde e_{2, 0}}(e_{1, 2}), \dots, \theta_{\phi, \widetilde e_{1, 0}, \widetilde e_{2, 0}}(e_{1, n-2})\}$$ $$\cong\mbP^{1}_{\overline \mbF_{p, 2}} \setminus \{e_{2, \infty}=\infty, e_{2, 0}=0, e_{2, 1}=1, e_{2, 2}, \dots, e_{2, n-2}\}$$ and $$\mbP^{1}_{\overline \mbF_{p, 2}} \setminus \{e_{2, \infty}=\infty, e_{2, 0}=0, e_{2, 1}=1, e_{2, 2}, \dots, e_{2, n-2}\} \times_{\overline \mbF_{p, 2}} k_{2} \cong U_{X_{2}}.$$ This means $U_{X^{\rm m}_{1}} \cong U_{X^{\rm m}_{2}}$ as schemes. In particular, we have $k_{2}^{\rm m} \cong \overline \mbF_{p,2}$.


Finally, we prove that ${\rm Hom}^{\rm op}_{\rm pg}(\pi_{1}^{\rm t}(U_{X_{1}}), \pi_{1}^{\rm t}(U_{X_{2}})) = {\rm Isom}_{\rm pg}(\pi_{1}^{\rm t}(U_{X_{1}}), \pi_{1}^{\rm t}(U_{X_{2}})).$ The ``$\supseteq$" part is trivial. We only need to prove the ``$\subseteq$" part. We may assume ${\rm Hom}^{\rm op}_{\rm pg}(\pi_{1}^{\rm t}(U_{X_{1}}), \pi_{1}^{\rm t}(U_{X_{2}})) \neq \emptyset.$  Let $\phi' \in {\rm Hom}^{\rm op}_{\rm pg}(\pi_{1}^{\rm t}(U_{X_{1}}), \pi_{1}^{\rm t}(U_{X_{2}}))$. Then $\pi_{1}^{\rm t}(U_{X_{1}})$ is isomorphic to $\pi_{1}^{\rm t}(U_{X_{2}})$ as abstract profinite groups. By Lemma \ref{lemsurj}, $\phi'$ is a surjection. Then \cite[Proposition 16.10.6]{FJ} implies that $\phi'$ is an isomorphism. Thus, we obtain $\phi' \in {\rm Isom}_{\rm pro\text{-}gps}(\pi_{1}^{\rm t}(U_{X_{1}}), \pi_{1}^{\rm t}(U_{X_{2}})).$ This completes the proof of (2).

Next, let us prove (1). Without loss of generality, we only treat the case where $i=1$. Moreover, let $(X, D_{X}) := (X_{1}, D_{X_{1}})$, $$D_{X}=\{e_{\infty}=\infty, e_{0}=0, e_{1}=1, e_{2}, \dots, e_{n-2}\},$$ $k:= k_{1}$, and $\overline \mbF_{p} :=\overline \mbF_{\widetilde e_{0} }$. Let $(r, Q)$ be a pair such that the following two conditions hold:
\begin{itemize}
  \item  $(r, p)=1$;
  \item $Q$ is an open normal subgroup of $\pi_{1}^{\rm t}(U_{X})$ such that $\pi_{1}^{\rm t}(U_{X})/Q \cong \mbZ/r\mbZ$, and that the Galois tame covering $f_{Q}: (X_{Q}, D_{X_{Q}}) \rightarrow(X, D_{X})$ over $k$ induced by $Q$ is totally ramified over $\{e_{\infty}, e_{0}\}$ and is \'etale over $D_{X} \setminus \{e_{\infty}, e_{0}\}$.
\end{itemize}

By applying Theorem \ref{them-2}, we see immediately that the set of pairs defined above can be mono-anabelian reconstructed from $\pi_{1}^{\rm t}(U_{X})$.

We fix a primitive $r$-th root of unity $\zeta_{r}$ in $\overline \mbF_{p}$ and put $s_{r} := [\mbF_{p}(\zeta_{r}): \mbF_{p}]$. Moreover, we put $$e_{Q, \infty}:=\infty, \ e_{Q, 0} := 0, \ e_{Q, 1}^{v}:= \zeta_{r}^{v} \in D_{X_{Q}}, \ v\in \{0, \dots s_{r}-1\},$$ and let $e_{Q, u} \in D_{X_{Q}}$, $u \in \{2, \dots, n\}$, such that $f_{Q}(e_{Q, u})=e_{u}$. Denote by $$L_{Q, u}:= \{e_{Q, u}-\sum_{v=0}^{s_{r}-1}b_{uv}e_{Q, 1}^{v}  \ | \ b_{uv}\in \mbF_{p}\} \cap \{0\}, \ u \in \{2, \dots, n-2\}.$$ By applying arguments similar to the arguments given in the proof of (2) above, we have that $d_{(X, D_{X})}=0$ if and only if there exists a pair $(r, Q)$ defined above such that $L_{Q, u} \neq \emptyset$ for each $u \in \{2, \dots, n-2\}$. Then the third and the final paragraphs of the proof of \cite[Lemma 3.3]{T2} implies that $L_{Q, u}$, $u \in \{2, \dots, n-2\}$, can be mono-anabelian reconstructed from $Q$. Thus, $d_{(X, D_{X})}$ can be mono-anabelian reconstructed from $\pi_{1}^{\rm t}(U_{X})$. This completes the proof of the theorem.
\end{proof}

\begin{remark}
Note that Theorem \ref{them-3} also holds if we replace $\pi_{1}^{\rm t}(U_{X_{i}})$, $i \in \{1, 2\}$, by its maximal pro-solvable quotient $\pi_{1}^{\rm t}(U_{X_{i}})^{\rm sol}$. Then we obtain the following solvable version of Theorem \ref{them-3} which is slightly stronger than the original theorem:
\begin{quote}
{\it We maintain the notation introduced above. Then $d_{(X_{i}, D_{X_{i}})}$, $i\in \{1, 2\}$, can be mono-anabelian reconstructed from $\pi_{1}^{\rm t}(U_{X_{i}})^{\rm sol}$. Moreover, suppose that $k_{1}^{\rm m} \cong \overline \mbF_{p, 1}$. Then the set of open continuous homomorphisms $${\rm Hom}^{\rm op}_{\rm pg}(\pi_{1}^{\rm t}(U_{X_{1}})^{\rm sol}, \pi_{1}^{\rm t}(U_{X_{2}})^{\rm sol})$$ is non-empty if and only if $U_{X^{\rm m}_1} \cong U_{X^{\rm m}_2}$ as schemes. In particular, if this is the case, we have $k_{2}^{\rm m} \cong \overline \mbF_{p, 2}$ and $${\rm Hom}^{\rm op}_{\rm pg}(\pi_{1}^{\rm t}(U_{X_{1}})^{\rm sol}, \pi_{1}^{\rm t}(U_{X_{2}})^{\rm sol}) = {\rm Isom}_{\rm pg}(\pi_{1}^{\rm t}(U_{X_{1}})^{\rm sol}, \pi_{1}^{\rm t}(U_{X_{2}})^{\rm sol}).$$}
\end{quote}
\end{remark}

\subsection{The second main theorem}\label{sec-6}

In this subsection, by using Theorem \ref{them-3}, we prove a result concerning pointed collection conjecture and the weak Hom-version conjecture (i.e. Theorem \ref{them-4}).
 We maintain the notation introduced in \ref{moduli212}.

\subsubsection{}\label{def-4}
Let $q \in M_{0, n}^{\rm ord}$ be an arbitrary point, $\overline {k(q)}$ an algebraic closure of $k(q)$, and $$U_{X_{q}} \simeq \mbP^{1}_{\overline {k(q)}} \setminus \{a_{1}=1, a_{2}=0, a_{3}=\infty, a_{4}, \dots, a_{n}\}$$ as $\overline {k(q)}$-schemes. We shall say that $q$ is a {\it coordinated point} if either $q=q_{\rm gen}$ or the following three conditions are satisfied:
\begin{itemize}
  \item $\text{dim}(V_{q})=\text{dim}(M_{0, n}^{\rm ord})-1$;
  \item there exists $i \in \{4, \dots, n\}$ such that $a_{i} \in \overline \mbF_{p}$;
  \item let $\omega^{i}_{n, n-1}: M_{0, n}^{\rm ord} \rightarrow M_{0, n-1}^{\rm ord}$ be the morphism induced by the morphism $\mcM_{0, n}^{\rm ord} \rightarrow \mcM_{0, n-1}^{\rm ord}$ obtained by forgetting the $i$th marked point; then $\omega^{i}_{n, n-1}(q)$ is the generic point of $M_{0, n-1}^{\rm ord}$.
\end{itemize}

Let $t$ be a closed point of $M_{0, n}^{\rm ord}$. Then there exists a set of coordinated points $P_{t}:= \{q_{t, 4}, \dots, q_{t, n}\}$ such that $$\{t\}=\bigcap_{q_{t, j} \in P_{t}}V_{q_{t, j}}.$$ 

\subsubsection{} Now, we prove the second main theorem of the present paper.

\begin{theorem}\ \label{them-4}
\begin{enumerate}
  \item For each closed point $t \in M_{0, n}^{\rm ord, cl}$, the set $\mcC_{t}$ associated to $t$  is a pointed collection (Definition \ref{def-3}). Moreover, for each pointed collection $\mcC \in \msC_{q_{\rm gen}}$, there exists a closed point $s \in M_{0, n}^{\rm ord, cl}$ such that $\mcC=\mcC_{s}$.
  \item Let $q \in M_{0, n}^{\rm ord}$ be an arbitrary point. Then the the natural map ${\rm colle}_{q}: \msV_{q}^{\rm cl} \rightarrow \msC_{q}, \ [t] \mapsto \mcC_{t},$ is an injection.
  \item Let $q \in M_{0, n}^{\rm ord}$ be an arbitrary point. Suppose that there exists a set of coordinated points $P_{q}$ such that $$V_{q} = \bigcap_{u \in P_{q}} V_{u}.$$ Then the pointed collection conjecture holds for $q$. In particular, the pointed collection conjecture holds for each closed point of $M_{0, n}^{\rm ord}$.
      \item Let $q_{i} \in M_{0, n}^{\rm ord}$, $i\in \{1, 2\}$, be an arbitrary point. Suppose that there exists a set of coordinated points $P_{q_{1}}$ such that $$V_{q_1} = \bigcap_{u \in P_{q_{1}}} V_{u}.$$ Then the weak Hom-version conjecture holds. In particular, the weak Hom-version conjecture holds when $q_{1}$ is a closed point.
\end{enumerate}

\end{theorem}

\begin{proof}
Let us prove (1). We put $F_{t}:= \{t' \in M_{0,n}^{\rm ord, cl} \ |\ t \sim_{fe} t'\}$. Let $t''$ be an arbitrary point of $\bigcap_{G \in \pi^{\rm t}_{A}(t)}U_{G}$. Then, for each $G \in \pi^{\rm t}_{A}(t)$, $\text{Hom}_{\rm pg}^{\rm surj}(\pi_{1}^{\rm t}(t''),  G)$ is non-empty, where $\text{Hom}_{\rm pg}^{\rm surj}(-, -)$ denotes the subset of $\text{Hom}^{\rm open}_{\rm pg}(-, -)$ whose elements are surjections. Since $\pi_{1}^{\rm t}(t'')$ is topologically finitely generated, we obtain that the set $\text{Hom}_{\rm pg}^{\rm surj}(\pi_{1}^{\rm t}(t''),  G)$ is finite. Then the set of open continuous homomorphisms $$\varprojlim_{G \in \pi^{\rm t}_{A}(t)}\text{Hom}_{\rm pg}^{\rm surj}(\pi_{1}^{\rm t}(t''),  G)=\text{Hom}^{\rm surj}_{\rm pg}(\pi_{1}^{\rm t}(t''), \pi_{1}^{\rm t}(t))$$ is non-empty. Thus, Theorem \ref{them-3} implies $t'' \in F_{t}$. This means  $$(\bigcap_{G \in \pi^{\rm t}_{A}(t)}U_{G}) \cap M_{g, n}^{\rm ord, cl}=F_{t}.$$ Since $U_{X_{t}}$ can be defined over a finite field, $F_{t}$ is a finite set. Then $\mcC_{t}$ is a pointed collection.

Let $\mcC \in \msC_{q_{\rm gen}}$ be a pointed collection and $s$ a closed point of $\bigcap_{G \in \mcC}U_{G}$. By replacing $t$ by $s$, and by applying arguments similar to the arguments given in the proof above, we obtain $\mcC=\mcC_{s}$.

(2) follows immediately from Theorem \ref{them-3}. Let us prove (3). If $n=4$, then $M_{0, 4}^{\rm ord}$ is a one dimensional scheme. For each $q\in M_{0,4}^{\rm ord}$, the pointed collection conjecture follows immediately from Theorem \ref{them-3}. Then we may assume $n\geq 5.$ To verify (iii), (ii) implies that we only need to prove that $\text{colle}_{q}$ is a surjection. Suppose that $q$ is a closed point of $M_{0, n}^{\rm ord}$, then (iii) follows immediately from Theorem \ref{them-3}.

Suppose that $q$ is a non-closed point. This means $\text{dim}(V_{q})\geq1.$ If $q=q_{\rm gen}$, (3) follows from (1) and (2). Let us treat the case where $q\neq q_{\rm gen}$. First, suppose that $q$ is a coordinated point, and that $$U_{X_{q}}\simeq\mbP^{1}_{\overline {k(q)}} \setminus \{1, 0, \infty, a_{4}, \dots, a_{n}\}.$$ Without loss of generality, we may assume  $a_{n} \in \overline \mbF_{p}$.

For each pointed collection $\mcC \subseteq \msC_{q}$, by applying (1), there exists a closed point $t_{1} \in M_{g, n}^{\rm ord, cl}$ such that $\mcC_{t_{1}}=\mcC$.  Then we have an open continuous surjective homomorphism $\pi_{1}^{\rm t}(q) \twoheadrightarrow \pi_{1}^{\rm t}(t_{1}).$ Let $\omega^{\setminus n}_{n, 4}: M_{0, n}^{\rm ord} \rightarrow M_{0, 4}^{\rm ord}$ be the morphism induced by the morphism $\mcM_{0, n}^{\rm ord} \rightarrow \mcM_{0, 4}^{\rm ord}$ obtained by forgetting the marked points except the first, the second, the third, and the $n$th marked points. We put $t_{1}'':= \omega^{\setminus n}_{n, 4}(t_{1})$ and $q'':= \omega^{\setminus n}_{n, 4}(q)$. Note that $t_{1}''$ and $q''$ are closed points of $M_{0, 4}$. Then Theorem \ref{them-2} implies that the surjection $\pi_{1}^{\rm t}(q) \twoheadrightarrow \pi_{1}^{\rm t}(t_{1})$ induces an open continuous surjective homomorphism $\pi_{1}^{\rm t}(q'') \twoheadrightarrow\pi_{1}^{\rm t}(t_{1}'').$ Thus, by Theorem \ref{them-3}, we obtain that $q'' \sim_{fe} t_{1}''.$ Then without loss of generality, we may assume  $$U_{X_{t_{1}}}\simeq \mbP^{1}_{\overline \mbF_{p}} \setminus \{1, 0, \infty, b_{4}, \dots, b_{n-1}, a_{n}\}$$ over $\overline \mbF_{p}$, where $b_{i} \in \overline \mbF_{p}$ for each $i \in \{4, \dots, n-1\}$.

On the other hand, let $\omega^{n}_{n, n-1}: M_{0, n}^{\rm ord} \rightarrow M_{0, n-1}^{\rm ord}$ be the morphism induced by the morphism $\mcM_{0, n}^{\rm ord} \rightarrow \mcM_{0, n-1}^{\rm ord}$ obtained by forgetting the $n$-th marked point. We put $t_{1}':=\omega^{n}_{n, n-1}(t_{1})$ and $q':= \omega^{n}_{n, n-1}(q)$, respectively. Since $q$ is a coordinated point, $q'$ is the generic point of $M_{0, n-1}^{\rm ord}$. Then we obtain $t_{1}' \in V_{q'}^{\rm cl}$. Moreover, we see $V_{q}=\omega_{n, n-1}^{-1}(q').$ Thus, $t_{1}=\omega^{-1}_{n, n-1}(t'_{1})$ is a closed point of $V_{q}$. Then the pointed collection conjecture holds for $q$ when $q$ is a coordinated point.

Next, we prove the general case. If $V_{q} = \bigcap_{u \in P_{q}} V_{u}$, then $V_{q}^{\rm cl}=\bigcap_{u \in P_{q}} V_{u}^{\rm cl}$ and $\bigcap_{u \in P_{q}} \msC_{u}=\msC_{q}$. Moreover, since we have a bijection $\text{colle}_{u}:\msV_{u}^{\rm cl} \xrightarrow{\simeq}\msC_{u}$ for each $u \in P_{q},$ we have that $$\text{colle}_{q}: \msV^{\rm cl}_{q}= \bigcap_{u \in P_{q}} \msV^{\rm cl}_{u} \rightarrow \bigcap_{u \in P_{q}} \msC_{u}=\msC_{q}$$ is a bijection. This completes the proof of (3).

Let us prove (4). We only need to prove the ``only if" part of the weak Hom-version conjecture. Suppose that $V_{q_2}$ is not essentially contained in $V_{q_{1}}$. This implies that there exists a closed point $t_{2} \in V_{q_{2}}^{\rm cl}$ such that $F_{t_{2}} \cap V_{q_{1}}=\emptyset$, where $F_{t_{2}}:= \{t_{2}' \in M_{0,n}^{\rm ord, cl} \ |\ t_{2} \sim_{fe} t_{2}'\}$. By (3), we have $\mcC_{t_{2}} \not\in \msC_{q_{1}}$. Thus, by Lemma \ref{lemsurj}, we obtain that $$\text{Hom}^{\rm op}_{\rm pg}(\pi_{1}^{\rm t}(q_{1}), \pi_{1}^{\rm t}(t_{2}))=\emptyset.$$ This provides a contradiction to the assumption that $\text{Hom}^{\rm op}_{\text{pg}}(\pi_{1}^{\rm t}(q_{1}), \pi_{1}^{\rm t}(q_{2}))$ is non-empty. This completes the proof of (4).
\end{proof}

\begin{remark}
Let $q \in M_{g, n}$ be an arbitrary point. Stevenson posed a question as follows (see \cite[Question 4.3]{Ste} for the case of $n=0$): Does $\bigcap_{G \in \pi_{A}^{\rm t}(q)} U_{G}$ contain any closed points of $M_{g, n}$? By \cite[Theorem 0.3]{T5}, $\bigcap_{G \in \pi_{A}^{\rm t}(q)} U_{G}$ contains a closed point of $M_{g, n}$ if and only if $q$ is a closed point of $M_{g, n}$. Furthermore, when $g=0$ and $q$ is a closed point, the proof of Theorem \ref{them-4} (1) implies that $$(\bigcap_{G \in \pi_{A}^{\rm t}(q)} U_{G})\cap M_{0, n}^{\rm cl}=F_{q},$$ where $F_{q}:=\{q' \in M_{0,n}^{\rm cl} \ |\ q \sim_{fe} q'\}$.

\end{remark}

\markright{ }

\end{document}